\documentclass[onefignum,onetabnum]{siamart171218}

\usepackage{lipsum,xspace}
\usepackage{amsfonts,amssymb,amsmath}
\usepackage{graphicx}
\graphicspath{{./figures}}
\usepackage{wrapfig}
\usepackage{sidecap}

\usepackage{algorithm}
\usepackage[noend]{algorithmic}
\algsetup{indent=2em}

\newcommand{\algostep}[1]{{\footnotesize\textup{#1:}}\xspace}

\usepackage{bbm}
\usepackage{multirow}
\usepackage{multicol}
\usepackage{mathtools}
\usepackage{threeparttable}
\usepackage{rotating}
\usepackage{relsize}
\usepackage{MnSymbol}
\usepackage{thmtools}
\usepackage{bm}
\usepackage{subfig}

\usepackage{doi}

\usepackage{color, colortbl}
\definecolor{Gray}{gray}{0.9}

\newcommand{\ds}{\displaystyle}
\newcommand{\until}[1]{\{1,\dots, #1\}}

\newcommand{\subscr}[2]{#1_{\textup{#2}}}
\newcommand{\supscr}[2]{#1^{\textup{#2}}} \newcommand{\setdef}[2]{\{#1
  \; | \; #2\}} 
\newcommand{\bigsetdef}[2]{\big\{#1 \; | \; #2\big\}}
\newcommand{\Bigsetdef}[2]{\Big\{#1 \; \big| \; #2\Big\}}
\newcommand{\map}[3]{#1: #2 \rightarrow #3}

\newcommand{\mcA}{\mathcal{A}}

\newcommand{\mcE}{\mathcal{E}}
\newcommand{\mcH}{\mathcal{H}}
\newcommand{\Bt}{B^\top}

\newcommand{\norm}[2]{\left\|#1\right\|_{#2}}

\newcommand{\pactive}{\subscr{p}{active}}
\renewcommand{\pactive}{\subscr{p}{sd}}
\newcommand{\wpactive}{\subscr{\widehat{p}}{sd}}

\newcommand\oprocendsymbol{\hbox{$\triangle$}}
\newcommand\oprocend{\relax\ifmmode\else\unskip\hfill\fi\oprocendsymbol}

\usepackage[fulladjust]{marginnote}

\newcommand{\zeropi}{[0,\pi)}

\usepackage{enumitem}
\renewcommand{\theenumi}{(\roman{enumi})}

\newcommand{\subscript}[2]{$#1 _ #2$}

\DeclareSymbolFont{bbold}{U}{bbold}{m}{n}
\DeclareSymbolFontAlphabet{\mathbbold}{bbold}
\newcommand{\vect}[1]{\mathbbold{#1}}

\newcommand{\vectorzeros}[1][]{\vect{0}_{#1}}
\newtheorem{problem}{Problem}

\newcommand{\scirc}{\raise1pt\hbox{$\,\scriptstyle\circ\,$}}
\newcommand{\real}{\mathbb{R}}
\newcommand{\Scircle}{\mathbb{T}^1}
\newcommand{\complex}{\mathbb{C}}
\newcommand{\integer}{\mathbb{Z}}

\newcommand{\torus}{\mathbb{T}} 
 
 \newcommand{\closure}{\mathrm{closure}}

\newsiamthm{remark}{Remark} 
\newsiamthm{remarks}{Remarks} 
\newsiamthm{assumption}{Assumption} 
\newsiamthm{example}{Example}

 \newcommand{\imagunit}{\mathrm{i}}

\newcommand{\prjcyc}{\mathcal{P}_{L_{\min}}}

\DeclareMathOperator{\diag}{diag}

\DeclareMathOperator{\Ker}{\mathrm{Ker}}
\DeclareMathOperator{\Img}{\mathrm{Img}}

\definecolor{Gray}{gray}{0.9}

\newcommand{\zeropiovertwo}{\ensuremath{[0,\tfrac{\pi}{2})}}

\newcommand{\myclearpage}{\clearpage}
\renewcommand{\myclearpage}{}
\ifpdf
  \DeclareGraphicsExtensions{.eps,.pdf,.png,.jpg}
\else
  \DeclareGraphicsExtensions{.eps}
\fi

\headers{Elastic and Flow Networks on the $n$-torus}{S. Jafarpour, E. Y. Huang, K. D. Smith, and Francesco Bullo}

\title{Flow and Elastic Networks on the $n$-torus: \\ Geometry, Analysis,
  and Computation\thanks{This work was supported in part by the Solar
    Energy Technologies Office of the U.S.~Department of Energy under
    Contract No.~DE-EE0000-1583 and by the Defense Threat Reduction Agency
    under Contract No.~HDTRA1-19-1-0017. Figures 2--9 and 12
    are licensed under Creative Commons Attribution-ShareAlike 4.0
    International License (CC BY-SA 4.0) and are reproduced 
    from~\cite{FB:20}. }}

\author{Saber Jafarpour\thanks{Center for Control, Dynamical Systems, and
    Computation and Department of Mechanical Engineering, University of
    California, Santa Barbara (\email{saber, eyhuang, kevinsmith,
      bullo@ucsb.edu}).}  \and Elizabeth Y.\ Huang$^{\dagger}$ \and Kevin
  D.\ Smith$^{\dagger}$ \and Francesco Bullo$^{\dagger}$.}

\usepackage{amsopn}

\usepackage[font={small}]{caption}

\ifpdf
\hypersetup{
  pdftitle={Flow and elastic networks on the $n$-torus},
  pdfauthor={S. Jafarpour, E. Y. Huang, K. D. Smith, and F. Bullo}
}
\fi

\begin{document}

\maketitle
\tableofcontents

\begin{abstract} 
  Networks with phase-valued nodal variables are central in modeling
  several important societal and physical systems, including power grids,
  biological systems, and coupled oscillator networks. One of the
  distinctive features of phase-valued networks is the existence of
  multiple operating conditions corresponding to critical points of an
  energy function or feasible flows of a balance equation. For
  {\color{black}networks} with phase-valued states, it is not yet fully
  understood how many operating conditions exist, how to characterize them,
  and how to compute them efficiently. A deeper understanding of feasible
  operating conditions, including their dependence upon network structures,
  may lead to more reliable and efficient network systems.
        
  This paper introduces flow and elastic network problems on the $n$-torus
  and provides a rigorous and comprehensive framework for their
  study. Based on a monotonicity assumption, this framework localizes the
  solutions, bounds their number, and leads to an algorithm to compute
  them. Our analysis is based on a novel \emph{winding partition} of the
  $n$-torus into winding cells, induced by \emph{Kirchhoff's angle law} for
  undirected graphs. The winding partition has several useful properties,
  including that each winding cell contains at most one solution. The
  proposed algorithm is based on a novel contraction mapping and is
  guaranteed to compute all solutions. Finally, we apply our results to
  numerically study the active power flow equations in several tests cases
  and estimate power transmission capacity and congestion of a power
  network.
\end{abstract}

\begin{keywords}
Network systems, graph theory, $n$-torus, cycle structure of a graph
\end{keywords}

\begin{AMS}
37C25, 93C10, 37C75, 37C65, 37N35
\end{AMS}

\myclearpage

\section{Introduction}
\label{sec:intro}
\subsection*{Problem description and motivation}

{\color{black} Complex networks are ubiquitous in the natural and engineered
  world, arising in such disparate fields as biology, sociology, and
  infrastructure systems.  In many networks of interest, the relevant
  quantities are represented as nodal and edge variables, and governing
  equations regulate their relationships and evolution. Nodal variables are
  typically real-valued and represent physical quantities, such as mass and
  density, or abstract quantities, such as information and
  opinions. Examples of engineered networks with real-valued nodal
  variables include water supply networks, gas distribution networks, and
  direct current (DC) power grids.  However, in some important
  applications, nodal variables are more accurately modeled as
  phases, i.e., points on the circle $\torus^1$. In these
  systems, the network state belongs to the $n$-torus $\torus^n$, instead
  of the Euclidean space $\real^n$; we refer to such systems as
  \emph{networks on the $n$-torus}.  Alternating-current (AC) power grids
  are well-known examples of networks on the $n$-torus, whereby power
  suppliers and consumers are the nodes and transmission lines are the
  edges.  The operating condition of an AC power grid is described by the voltage
  angle at each node (nodal variables) and the power flow along each
  transmission line (edge variables).  Other examples of networks on the
  $n$-torus arise in the study of synchronization of coupled oscillators
  and collective motion of multi-agent systems.


  In networks on the $n$-torus, solutions of the governing equations may be
  fixed-points of an algebraic equation, equilibrium points of a dynamical
  system, or critical points of an optimization problem.
  One of the distinctive behaviors of the network governing equations is
  the coexistence of multiple solutions.
  This behavior is observed in the existence of multiple stable operating
  conditions in AC power grids~\cite{TC-RD-IA-PJ:16}, diverse
  frequency-synchronized states in coupled oscillators~\cite{DM-MT-DW:17},
  and rich spatial patterns of motion in engineered and biological
  networks~\cite{NEL-DP-FL-RS-DMF-RD:07}.
  Formally, determining the solutions to governing equations is often
  equivalent to computing zeros of continuous maps from $\torus^n$ to
  $\real^n$. And, indeed, the Poincar\'e\textendash{}Hopf
  Theorem~\cite{JWM:97} forbids the uniqueness of such zeros.
  In all applications of networks on the $n$-torus, a comprehensive
  understanding of the multiplicity of solutions is critical for
  monitoring, predicting, and controlling the network system.
  

  For systems with real-valued nodal variables, it is well-known that
  solutions to the network governing equations are related to structural properties of 
  the network, and algebraic graph theory provides powerful
  tools for analyzing this relationship.  However, for networks on the
  $n$-torus, the geometry of the ambient space complicates the connection
  between the structure of the network and its behavior.  Despite numerous
  efforts in various scientific disciplines, many fundamental questions are
  not yet well understood: How many solutions exist? How can they be
  localized and computed? What is the role of the network structure?  In
  this paper, we address these problems by developing a novel algebraic
  graph theory on the $n$-torus.}

\subsection*{Flow and elastic networks on the $n$-torus}
{\color{black}We introduce two classes of networks on the $n$-torus, namely
  \emph{flow networks} and \emph{elastic networks}, motivated by important
  example systems. We start with some notation. A network is described by a
  weighted undirected graph $G$ with node set $\{1,\ldots,n\}$, edge set
  $\mathcal{E}$ of cardinality $m$, and edge weights $a_{ij}\in
  \real_{>0}$, for $(i,j)\in \mathcal{E}$. Given an arbitrary orientation
  and ordering of the edges, let $B\in \real^{n\times m}$ denote the
  incidence matrix of $G$.  We let $\torus^1$ denote the unit circle
  and $\torus^n$ denote the $n$-torus. Given angles
  $\alpha,\beta\in\Scircle$, the geodesic distance between $\alpha$ and
  $\beta$, denoted by $|\alpha-\beta|$, is the length of the shortest arc
  in $\Scircle$ connecting $\alpha$ to $\beta$. If we identify $\Scircle
  \simeq [-\pi,\pi)$, the counterclockwise difference
    $\subscr{d}{cc}(\alpha,\beta)$ is defined by
    $\subscr{d}{cc}(\alpha,\beta) =\mathrm{mod}((\beta - \alpha) , 2\pi)-
    \pi$. By abuse of notation, we write $\alpha-\beta$ to refer to
    $\subscr{d}{cc}(\alpha,\beta)$. $\vect{1}_n$ and $\vect{0}_n$ are the
    all-one and all-zero column vectors of length $n$, respectively, and
    $\vect{1}_n^{\perp}$ is the subspace of $\real^n$ consisting of vectors
    perpendicular to $\vect{1}_n$.}

A flow network is characterized by a commodity being exchanged between
adjacent nodes. The direction and magnitude of this flow is a function
of the phase difference between adjacent nodes. Moreover, flows
satisfy a balance equation asserting that, {\color{black} at each node,
  the total outflow (resp.\ inflow) must equal the supply (resp.\
  demand) of commodity}. The solutions of a flow network problem
reveal the directions and magnitudes of commodity flows within the
network, given the {\color{black} external supply and demand} at each
node.

\begin{problem}[Flow networks on the $n$-torus]\label{sinflowproblem}
  A \emph{flow network on the $n$-torus} is a pair $(G,\{h_e\}_{e\in
    \mathcal{E}})$, where $G$ is a weighted undirected graph
  {\color{black}with $n$ nodes} and, for every $e\in \mathcal{E}$,
  $\map{h_e}{\real}{\real}$ is a continuously differentiable
  $2\pi$-periodic odd function, called the \emph{flow function} on edge
  $e$.
  
  Consider a flow network on the $n$-torus $(G,\{h_e\}_{e\in
    \mathcal{E}})$, a balanced supply/demand vector $\pactive\in
  \vect{1}^{\perp}_n$, and an angle $\gamma\in\zeropi$. Then the \emph{flow
    network problem on the $n$-torus} for $(G,\{h_e\}_{e\in
    \mathcal{E}},\pactive,\gamma)$ is to compute pairs of flows and phase
  angles $(f,\theta)\in \real^m\times\torus^n$ such that
  \begin{subequations} \label{eq:flow_form}
    \begin{align}
        &B f = \pactive, \label{eq:f-KCL}\\
        &f_e = a_{ij} h_e(\theta_i-\theta_j),\qquad\mbox{ for }e=(i,j)\in \mathcal{E} \label{eq:f-physics}
      \\
        &|\theta_i-\theta_j| \le \gamma, \qquad\quad\qquad\mbox{ for } (i,j)\in \mathcal{E}.\label{eq:f-constraints}
    \end{align}
  \end{subequations}
\end{problem}
These notions {\color{black} are illustrated in Figure~\ref{fig:apf} and
  interpreted as follows}.  Equation~\eqref{eq:f-KCL} is the \emph{flow
  balance equation} for a commodity flowing on a network, similar to a
Kirchhoff's current conservation law, and is illustrated in
Figure~\ref{fig:apfpConservation}.  Equation~\eqref{eq:f-physics} is the
\emph{flow equation} and codifies how the flow along an edge is a function
of the counterclockwise difference between the angular variables at the two
nodes; such a function is illustrated in Figure~\ref{fig:apfpPowerAngle}.
The oddness of the flow functions encodes the physical property that, on
every edge, the flow equations are independent of the flow direction.
Finally, the inequality~\eqref{eq:f-constraints} is the \emph{angle
  constraint}, often relevant in applications.

\begin{figure}[ht]
  \begin{center}
    \subfloat[\label{fig:apfpConservation}]
             {\includegraphics[width=.325\textwidth]{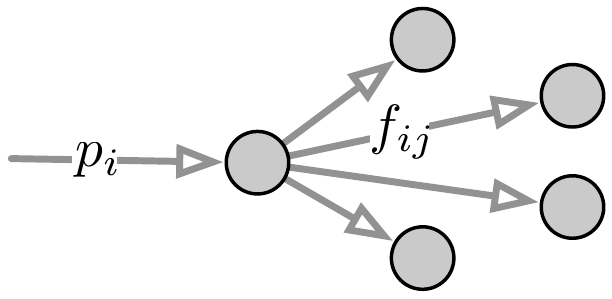}}
             \quad\hfil
             \subfloat[\label{fig:apfpPowerAngle}]
                      {\includegraphics[width=.44\textwidth]{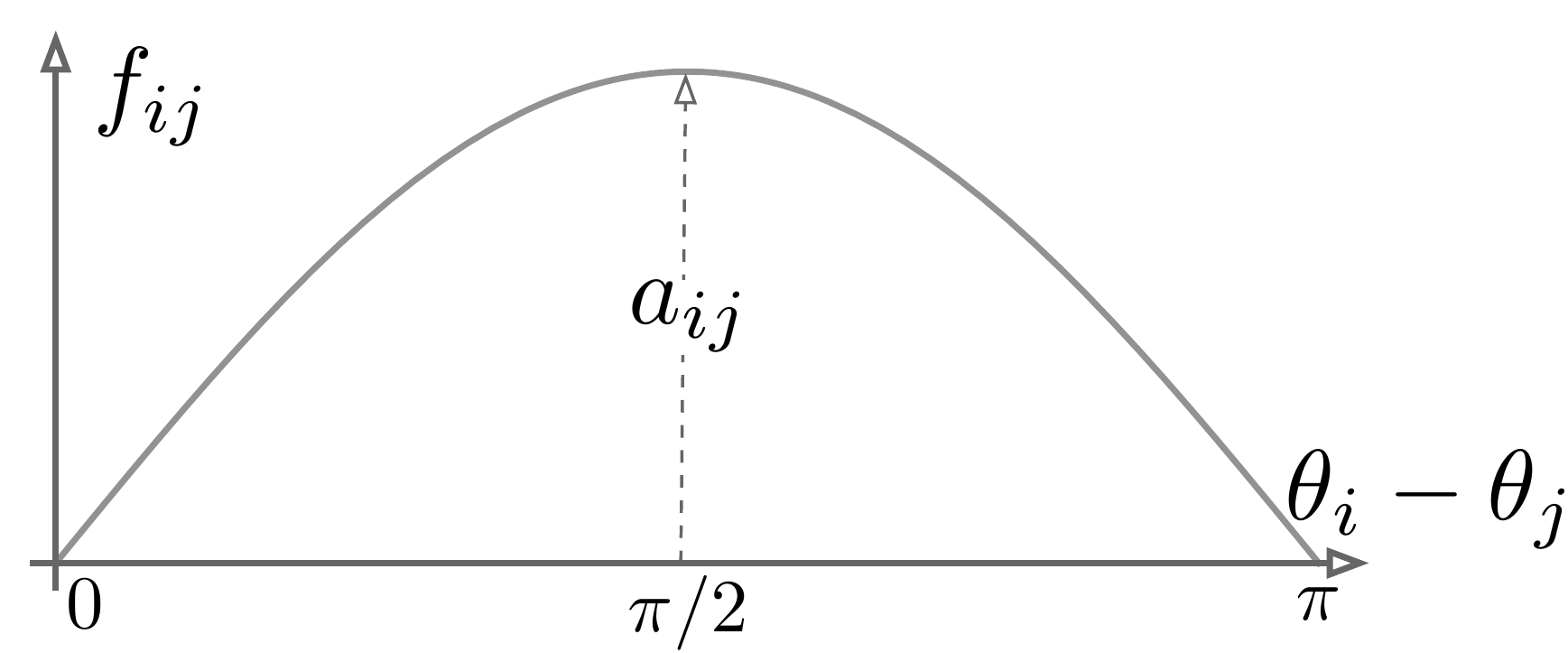}}
  \end{center}
  \caption{Interpretation of the flow network problem~\eqref{eq:flow_form}
    for the networks where the flow function has a sine
    form. Figure~\ref{fig:apfpConservation} illustrates that the flow
    satisfies a conservation law at each node: $\ds p_i =
    \sum\nolimits_{j=1}^nf_{ij}$. Figure~\ref{fig:apfpPowerAngle} shows the
    proportional relationship between the flow transferred along each edge
    and the sine of its angle difference.}
\label{fig:apf}
\end{figure}

Flow networks on the $n$-torus arises in chemical
oscillations~\cite{JCN:79,YK:84}, where the commodity is mass.  They
also arise in models of AC power networks~\cite{ARB-DJH:81},
droop-controlled inverters in microgrids~\cite{JWSP-FD-FB:12u}, and
networks of pacemaker cells in the heart~\cite{DCM-EPM-JJ:87}, where
the commodity is electricity.  Models of brain
networks~\cite{FV-JPL-ER-JM:01}, deep brain stimulation~\cite{PAT:03},
and social influence systems~\cite{AP-VL-AR:05} can also take the form
of a flow network on the $n$-torus, with commodities like information
and opinions.

{\color{black} An elastic networks is characterized by an interaction
  energy between adjacent nodes.} The energy of an edge is a function
of the phase difference between the adjacent nodes and these
interaction energies give rise to forces on nodes.  When an elastic
network is in steady-state, the forces arising from these potential
energies must cancel out with external forces, resulting in a force
balance equation. The solutions of an elastic network reveal the
possible steady-state energies and forces along edges, given external
forces at each node. 

\begin{problem}[Elastic networks on the $n$-torus]\label{torqueproblem}
  An \emph{elastic network on the $n$-torus} is a pair $(G,\{H_e\}_{e\in
    \mathcal{E}})$, where $G$ is a weighted undirected graph
  {\color{black}with $n$ nodes} and, for every edge $e\in \mathcal{E}$,
  $\map{H_e}{\real}{\real}$ is a twice-differentiable $2\pi$-periodic even
  function called the \emph{elastic energy function} of edge $e$. The
  elastic energy of the elastic network is
  \begin{equation}\label{eq:energy_function}
    \mcH(\theta) = \sum_{e=(i, j) \in \mathcal E} a_{ij} H_e(\theta_i -
    \theta_j).
  \end{equation}
  Consider an elastic network $(G,\{H_e\}_{e\in \mathcal{E}})$, a balanced
  torque vector $\tau\in \vect{1}^{\perp}_n$, and an angle
  $\gamma\in\zeropi$. Then the \emph{elastic network problem on the
    $n$-torus} for $(G,\{H_e\}_{e\in\mathcal{E}},\tau,\gamma)$ is to
  compute the phase angles $\theta \in \torus^n$ such that
  \begin{subequations} \label{eq:elastic_form}
    \begin{align}
      &\tau = \nabla_{\theta}\mcH(\theta) , \label{eq:f-elastic} \\
      &|\theta_i-\theta_j| \le \gamma,\qquad\mbox{for } (i,j)\in \mathcal{E}. \label{eq:force-constraints}
    \end{align}
  \end{subequations}   
\end{problem}
These notions are illustrated in Figure~\eqref{fig:springnetwork} and
interpreted as follows. For every edge $e$, the elastic energy function
$H_e$ codifies how the elastic energy in the edge $e$ is a function of the
counterclockwise difference between the angular variables at its end
points. Equation~\eqref{eq:f-elastic} is the \emph{torque balance
  equation} for torques applied to a node in the network and is similar to
conservation of momentum. Finally, the
inequality~\eqref{eq:force-constraints} is the \emph{angle constraint},
often relevant in applications.

\begin{wrapfigure}{l}{0.4\textwidth}
\centering
\includegraphics[width=0.33\textwidth]{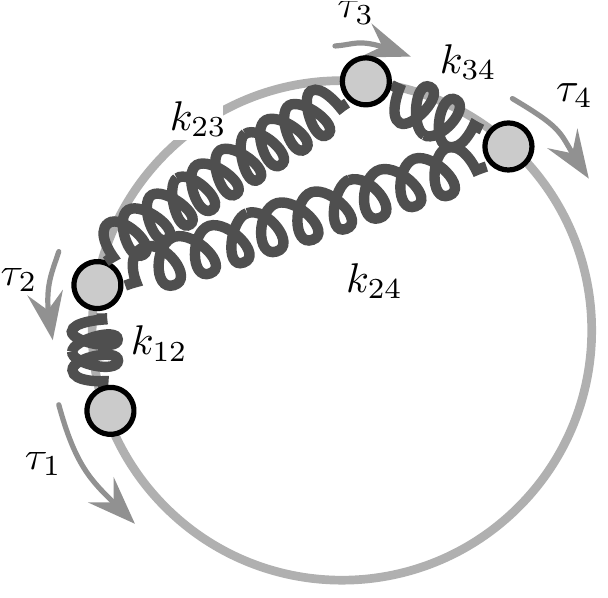}
\caption{Interpretation of the elastic network
  problem~\eqref{eq:elastic_form} as a spring network on a ring. Each
  particle on the ring is subject to an external torque, a damping torque,
  and is interconnected by ideal elastic springs to other
  particles.}\label{fig:springnetwork}\vspace{-1.0cm}
\end{wrapfigure}
Elastic networks on the $n$-torus are present in many engineering
systems, including vehicle
coordination~\cite{DAP-NEL-RS-DG-JKP:07,RS-DP-NEL:07}, nullforming in
wireless networks~\cite{AK-RM-SD-MMR-DRB-UM-TPB:15}, and
nanoelectromechanical oscillator
networks~\cite{MM-JE-WF-AC-AS-MR-JL-MHB-MP-LDO-MM-JPC-MCC-RMD-MLR:19}.
They arise in physics in the context of solid-state
circuits~\cite{AM-MEH-RB-SC-AAA:07}, spin glass models~\cite{HD:92},
and Josephson junctions in quantum mechanics~\cite{KW-PC-SHS:98}. They
can also model collective behavior in biological
networks~\cite{TV-AC-EBJ-IC-OS:95}.

Elastic networks can also arise from certain network optimization problems on
the $n$-torus, in which a cost function $\map{\mcH}{\torus^n}{\real}$ can be decomposed into the sum of edge
costs $\map{H_e}{\torus^n}{\real}$, for $e\in \mathcal{E}$. Then it is easy
to see that the optimization problem
\begin{align}\label{eq:optimization_general}
  \begin{split}
    &\min \qquad \mcH(\theta) = \sum_{(i,j)\in \mathcal{E}}H_e(\theta_i-\theta_j), \\
    &\mbox{subject to: }\qquad\theta \in \torus^n.
  \end{split}
\end{align}
gives rise to an elastic network on the $n$-torus with elastic energy
$\mcH$. The elastic network problem, for the torque vector
$\tau=\vect{0}_n$, is equivalent to finding critical points of the elastic
energy function $\mcH(\theta)$ satisfying the constraint
$|\theta_i-\theta_j|\le \gamma$.

{\color{black}
Having introduced flow and elastic networks on the $n$-torus, we now elaborate
on several research questions that naturally arise for these systems.
Given a flow or elastic network on the $n$-torus,
\begin{enumerate}[leftmargin=2.5em,
    label=\subscript{Q}{{\arabic*}}:,ref=(\subscript{Q}{{\arabic*}})]
\item\label{Q1-how-many-eqs} does a solution exist and, if so, is it
  unique?  Alternatively, how do we enumerate the solutions as a function
  of all system parameters?
\item\label{Q2-localize-eqs} how can the solutions be distinguished and
  localized on the $n$-torus?
\item\label{Q3-compute-eqs} can we design an algorithm that is guaranteed to
  compute all solutions?
\end{enumerate}
Additionally, in many network problems, the edge variables themselves are
at least as interesting as the underlying nodal phases.  Indeed, edge
variables describe the flows of the commodity and the elastic energy among
the network nodes.  Therefore, we will consider the following additional
questions:
\begin{enumerate}[leftmargin=2.5em,
    label=\subscript{ Q}{{\arabic*}}:,ref=(\subscript{Q}{{\arabic*}})]
  \setcounter{enumi}{3}
\item\label{Q4-loop-flows-at-eq} how do different solutions lead to
  different {\color{black}patterns of edge variables}?
\item\label{Q6-capacity} what is the maximum {\color{black} flow or elastic
  energy in a given edge} as a function of network parameters?
\end{enumerate} }

\subsection*{Applications}

{\color{black} We here present three applications of flow and elastic
  networks on the $n$-torus and briefly discuss the importance of
  questions~\ref{Q1-how-many-eqs}-\ref{Q6-capacity} in their study.}

\paragraph{Active power flow equations}{\color{black}In the
  context of AC power grids with constant voltage magnitudes, the active
  power flow equation for lossless AC circuits is
  \begin{align}\label{eq:active-power-flow}
    p_i = \sum_{j=1}^{n} Y_{ij}V_iV_j\sin(\theta_i-\theta_j),
\end{align}
where $p_i,V_i, \theta_i$ are active power supply/demand, voltage magnitude
and voltage phase at node $i$, respectively, and $Y_{ij}$ is the line
susceptance at edge $(i,j)\in \mathcal{E}$. In practice, the power transmitted on a power
line is limited by thermal constraints.  The angular difference
$\theta_i-\theta_j$ is called the power angle of the line $(i,j)$, and the
line thermal constraints can be expressed as
\begin{align}\label{eq:thermal_constraints}
  |\theta_i-\theta_j|\le \gamma,\qquad\mbox{ for all } (i,j)\in \mathcal{E},
\end{align}
for some maximum power angle $\gamma\in\zeropiovertwo$. Notably, the active
power flow equation~\eqref{eq:active-power-flow} together with the thermal
constraints~\eqref{eq:thermal_constraints} are precisely of the
form~\eqref{eq:f-KCL}-\eqref{eq:f-constraints} with $h_{e}(y)=\sin(y)$ on
each edge $e$. 

An important feature in AC power grids is the existence of
multistable operating points in the network corresponding to solutions of
active power flow equation~\eqref{eq:active-power-flow}.
It is well-known that the transition between these distinct stable
operating points is sometimes accompanied by large, undesirable changes in
the circulating power flows~\cite{TC-RD-IA-PJ:16}. Circulating flows do not
deliver usable power to customers, increase power losses and line
congestion, and might cause pricing volatility. The Lake Erie loop, which
consists of roughly $1000$ miles of transmission lines through Ontario and
five American states, presents a classic example of the challenge that
circulating power flows present to power grid stability~\cite{NERC-04}. In
order to prevent large failures and blackouts caused by the circulant
flows, a theoretical understanding of different operating points of the
system and their loop flows is essential.}

\paragraph{Collective motion in engineering and biology}
{\color{black} An insightful example of elastic networks on the $n$-torus
  arises in the study of collective motion in biological and engineering
  networks. Consider a group of $n$ agents moving with unit speed on a
  two-dimensional plane with the pairwise communication and sensing
  patterns between the agents given by the edge set $\mathcal{E}$.  The
  motion of each agent $k$ is described by its position $r_k = x_k +
  \imagunit y_k\in \complex\simeq\real^2$, its heading angle $\theta_k\in
  \Scircle$, and a steering control law $u_k = u_k(r,\theta)$ depending on
  the position and heading angles of the other agents in the
  group. Formally, the motion of agent $k$ is given by:
  \begin{align}
    \label{eq:motion-bio-eng}
    \dot{r}_k = e^{\imagunit\theta_k},\qquad \dot{\theta}_k = u_k(r,\theta).
  \end{align}
For a network of agents described by the model~\eqref{eq:motion-bio-eng}, a
well-studied form of collective motion is \emph{flocking}, whereby
particles asymptotically reach consensus on their heading angles. A second
important form of collective motion is the so-called \emph{circular motion}
(or \emph{vortex motion}). In circular motion, agents exhibit periodic
movements around a circle, where their headings are different vertices on
the circle~\cite{NEL-DP-FL-RS-DMF-RD:07}. One can define the elastic energy
using the so-called~\cite{DAP-NEL-RS-DG-JKP:07} \emph{spacing potential} on
the $n$-torus:
\begin{align}\label{eq:spacing_potential}
  \mcH(\theta) =\sum_{(k,j)\in \mathcal{E}} (1-\cos(\theta_k-\theta_j)).
\end{align}
The spacing potential has several critical points and its global maximum
and minimum corresponds to the circular motion and flocking,
respectively~\cite{NEL-DP-FL-RS-DMF-RD:07}. Using a gradient flow controller
$u_k = \omega_0 -\frac{\partial \mcH(\theta)}{\partial \theta_k}$ which depends on a scalar $\omega_0\in\real$, the network of particles~\eqref{eq:motion-bio-eng} can be considered an elastic network on the $n$-torus
where the elastic energy at each edge $(k,j)$ is given by
$1-\cos(\theta_k-\theta_j)$. In engineering networks, the energy function
approach can be used to design more complicated patterns of collective
motion. The idea is to utilize appropriate stabilizing control laws to
steer the system to desired critical points of the energy
function~\cite{NEL-DP-FL-RS-DMF-RD:07}. However, existence of other stable
critical points for the energy function and their locations have a large
effect on transient and steady-state behavior of the closed-loop
system. This motivates a comprehensive study of the critical points of
various energy functions on the $n$-torus. }

\paragraph{Coupled oscillators and associative memory networks}
{\color{black} Synchronization in networks of coupled oscillators is an
  emerging phenomenon with important applications in biology, physics, and
  engineering. The celebrated Kuramoto model is one of the simplest
  coupled-oscillator model exhibiting synchronization. In this model, each
  oscillator has a phase $\theta_i\in \Scircle$ and a natural frequency
  $\omega_i\in \real$; the interactions between oscillators are described
  by a weighted undirected graph $G$, whose edge weights $a_{ij}$ denote
  the coupling strength between oscillators $i$ and $j$. The dynamics of
  the $i$th oscillator is given by:
\begin{align}\label{eq:kuramoto-assoc}
 \dot{\theta}_i = \omega_i - \sum_{j=1}^{n} a_{ij}\sin(\theta_i-\theta_j)
\end{align}
Simple calculations show that (i) synchronous trajectories (i.e.,
trajectories with identical frequencies) of the Kuramoto model are
precisely solutions to a flow network
problem~\eqref{eq:f-KCL}-\eqref{eq:f-physics} with sinusoidal flow
functions and that (ii) solutions satisfying the
constraint~\eqref{eq:f-constraints} for $\gamma<\pi/2$ are locally
exponentially stable~\cite[Lemma 2]{FD-MC-FB:11v-pnas}. One of the
applications of coupled oscillators is associative memory
networks. Associative memory is a type of content-addressable memory for
recognizing special patterns using partial information. In~\cite{JJH:1982},
Hopfield proposed to consider patterns of memory as dynamically stable
attractors. Motivated by this idea, the Kuramoto model and its
generalizations together with a Hebb’s learning rule for the couplings have
been used to design associative memory
networks~\cite{FCH-EMI:00,TN-YCL-FCH:04}. In these networks, binary
patterns of memory are encoded as stable frequency-synchronized solutions
of the coupled oscillator network~\cite{FCH-EMI:00,TN-YCL-FCH:04}. For
these models of associative memory, the performance is measured by the number and the size of
basin of attraction of stable frequency-synchronized states~\cite{TN-YCL-FCH:04}. Therefore, to study the performance of these
associative memory networks, it is essential to develop a theoretical
framework for computing the stable frequency synchronized solutions of
coupled-oscillator networks.  }

\subsection*{Relevant literature}{\color{black}Numerous research works
  have studied the multiplicity of solutions of networks on the $n$-torus
  and its connections with the structural properties of the network. In the
  physics community, much effort has focused on studying the equilibrium
  points for the Kuramoto coupled-oscillator model and its
  generalizations. To our knowledge, \cite{GBE:85(b)} is the first paper
  that studies existence of multi-stable equilibria in a ring of Kuramoto
  oscillators. The phenomenon of multistability has been further explored
  in the generalized Kuramoto model with small world graph~\cite{GSM:14}
  and with Cayley graphs~\cite{GSM-XT:15}. An algorithm for computing the
  multi-stable equilibria of ring network of coupled oscillators is
  proposed by~\cite{JAR-DA:04}. For Kuramoto coupled oscillators with zero
  natural frequencies, the uniqueness of the zero stable equilibrium point
  for dense graph has been studied in~\cite{RT:12,SL-RX-ASB:19}. The region
  of attraction of multistable equilibrium points on a $k$-nearest neighbor
  graphs is studied in~\cite{DAW-SHS-MG:06}. For Kuramoto model with second
  order couplings, \cite{TN-YCL-FCH:04} studies the stability and basin of
  attraction of equilibrium points. In~\cite{DM-MT-DW:17}, the notion of
  winding vector for an arbitrary graph is introduced as the sum of the
  phase differences along a cycle basis. Moreover, it is shown that for a
  planar graph, there is a one-to-one correspondence between the winding
  vectors and the equilibrium points of the Kuramoto model with phase
  differences less than $\pi/2$~\cite[Lemma~3]{DM-MT-DW:17}. The
  analysis in~\cite{DM-MT-DW:17}
  also provides upper and lower bounds on the number of stable equilibrium
  points of Kuramoto coupled oscillators with planar
  topology. \cite{JB-CIB:82} uses Morse Theory to provide an upper bound on
  the number of stable equilibrium points of the Kuramoto
  model. \cite{TF:18} provide an asymptotic estimate on the number of
  equilibrium points of the Kuramoto model by identifying the equilibrium
  points with suitable lattice points. Numerical algorithms based on
  homotopy continuation and algebraic geometry have been developed
  in~\cite{DM-NSD-FD:15,OC-JDH-HH-DKM:18} to compute the equilibrium points
  of the Kuramoto model.



  In the power system community, multiplicity of solutions to active power
  flow equations has been studied intensively. \cite{AJK:72} is the first
  paper to rigorously study multistable solutions of the active power flow
  equation and their associated loop flows. Several
  papers~\cite{TC-RD-IA-PJ:16,RD-TC-PJ:16,NJ-AK:03} have investigated
  mechanisms that create large loop flows and have shown that loop flows
  persist even when the network returns to its normal operating
  condition. In~\cite{TC-RD-IA-PJ:16}, the basins of attraction for loop
  flows has been studied using Lyapunov theory. The
  papers~\cite{TC-RD-IA-PJ:16,RD-TC-PJ:16,NJ-AK:03} acknowledge the
  connection between loop flows and winding numbers of the cycles in the
  graph. The holomorphic embedding load-flow method (HELM) is proposed
  in~\cite{AT:12} to find all solutions of power flow equations. Although
  HELM is guaranteed to find the operable solution of the power flow
  equations, it is reported to be much slower than the
  Newton\textendash{}Raphson
  methods~\cite{SR-YF-DJT-MKS:16}. In~\cite{WM-JST:93}, an algorithm based
  on a topological continuation argument is developed to solve the power
  flow problem. However, a counterexample for this method using a five-bus
  system has been constructed in~\cite{DKM-BCL-HC:13}. This continuation
  method is revisited and modified in~\cite{BL-DW:15}. Techniques for
  solving the optimal power flow problem (OPF) can also be used to solve
  the power flow problem. OPF problems have been studied extensively in the
  power network literature, e.g.,
  see~\cite{JAM-RA-MEE:99-1,JAM-RA-MEE:99-2,DKM-FD-HS-SHL-SC-RB-JL:17}. Unfortunately,
  due to the non-convex nature of the OPFs, these algorithms are not
  guaranteed to compute feasible solutions for arbitrary
  topologies~\cite{SHL:14}.
}

\subsection*{Contributions}

This paper provides a rigorous graph-theoretic and geometric framework to
analyze the solutions of flow and elastic network problems on the
$n$-torus.  Our framework provides comprehensive answers to
questions~\ref{Q2-localize-eqs} and~\ref{Q3-compute-eqs}, and it provides
novel insights into
questions~\ref{Q1-how-many-eqs},~\ref{Q4-loop-flows-at-eq},
and~\ref{Q6-capacity}.  Specifically, this paper provides the following
four main contributions:
\begin{enumerate}[leftmargin=3em]
	\item We introduce a unifying formalism to study network problems on the $n$-torus.
	\item We present a novel partition of the $n$-torus, called the \emph{winding 
	partition}, induced by cycles in the underlying graph.
	\item We prove that each winding cell contains at most a unique solution of the flow 
	network problem and elastic network problem.
	\item We present a complete search algorithm, which finds every solution to the flow 
	and elastic network problems on a suitable subset of the $n$-torus.
\end{enumerate} 
We next discuss these contributions in detail.

In this paper, we formally introduce flow and elastic networks on the
$n$-torus, and we state the problem of finding all solutions of these
networks.  While these problems have been studied in various
disciplines, we provide a novel, unifying framework to formalize
them. We review several applications of flow and elastic network
problems on the $n$-torus. {\color{black}In Section~\ref{sec:problems},
  we show that, by establishing a suitable correspondence between the
  supply/demand vectors and the torque vectors and between the flow
  functions and the elastic energy functions, the solutions of flow
  network problems and elastic network problems coincide.}

In Section~\ref{sec:winding-partition}, we provide a novel algebraic graph
theory on the $n$-torus.  We establish a fundamental property of angle
differences around cycles in graphs, which amounts to a \emph{Kirchhoff's
  angle law}, generalization of the classic voltage law.  {\color{black} We
  introduce the notion of winding vector as a map from the $n$-torus to a
  discrete set and thereby define a novel partition of the $n$-torus
  into \emph{winding cells}.}

In Section~\ref{sec:at-most-uniqueness}, we focus on flow network problems
with monotone flow functions over suitable domains.  We show that, in each
winding cell, there exists at most one solution for the flow network
problem.  This ``at-most uniqueness'' result shows that winding cells can
be used to localize the multiple solutions of the flow network problem,
thereby providing a graph-theoretic and geometric answer to
question~\ref{Q2-localize-eqs}.  This result also shows that the winding
partition is a key geometric {\color{black}concept} to study the flow
network problem.  Finally, this result allows us to upper bound the number
of solutions to the flow network problem (on an arbitrary graph), partially
answering question~\ref{Q1-how-many-eqs}. 
The rest of Section~\ref{sec:at-most-uniqueness} addresses 
question~\ref{Q4-loop-flows-at-eq} by examining ``circulating'' or ``loop'' flows around 
cycles in the network.
We establish a bijection between a solution's winding vector and the circulating flow 
associated with that solution {\color{black} for flow networks with
  arbitrary topology and arbitrary monotone flow functions. It is
  worth mentioning that this bijection has been studied in the
  literature for Kuramoto coupled-oscillator networks with the ring
  topology in~\cite{RD-TC-PJ:16} and with the planar
  topology in~\cite{DM-MT-DW:17}}.  
We show that, in a flow network with strictly increasing flow functions and a single 
cycle, the circulating flow around the cycle increases monotonically with respect to the 
winding number.

In Section~\ref{sec:complete-solver}, we recast the flow and elastic
network problems as a system of equations with explicit dependence on the
winding vector of the solution.  Unlike the flow network problem, this new
formulation has at most one solution.  We propose an iterative algorithm,
called the \textit{projection iteration}, that exploits these equations to
either solve the problem in a given winding cell, or determine that no such
solution exists.  Using this iteration, we propose a complete search
algorithm to find all solutions of flow and elastic network problems on the $n$-torus, as stated in
question~\ref{Q3-compute-eqs}. {\color{black}In
  Section~\ref{subsec:compcomp}, we provide a comprehensive analysis of the
  computational complexity of each component of this algorithm.} The power
systems literature contains many algorithms to find these solutions in the
context of active power flow solvers; however, convergence is only
guaranteed for some special graphs, whereas our algorithm provably
converges on arbitrary topologies.


Finally, in Section~\ref{sec:numerical-experiments}, we present
numerical experiments on special cases of the active power flow
equations.  We consider two different test cases: a 12-node ring and
the IEEE RTS 24 testcase.  First, we introduce a notion of
transmission capacity and a notion of network congestion.  We utilize
the projection iteration to numerically study the capacity and
congestion in a 12-node ring power grid under two different power
profiles.  Our numerical analysis shows that, under certain cases, an
increase in loop flow can lead to an increase in the network's power
transmission capacity.  For the IEEE RTS 24 testcase, we modify the
power supply/demand vector and find a solution for the active power
flow equations with non-zero winding vector. Moreover, we compare the time
efficiency of the projection iteration against the
Newton\textendash{}Raphson method for computing the solutions of the
active power flow equations.  The numerical analysis in this section
sheds some light on question~\ref{Q6-capacity}; however, a more
comprehensive answer is still the subject of future research.

\myclearpage

\section{Preliminary concepts and results}\label{sec:problems}

{\color{black} We start this section by introducing the notation that are essential for our framework in this paper.}

\subsection{Notation and mathematical conventions}
Let $\real$ and $\complex$ denote the set of real and complex numbers,
respectively. For every $x\in \real^n$ and every $r >0$, we define the open
disk $\mathrm{D}_{\infty}(\mathbf{x}, r)$ by
\begin{align*}
\mathrm{D}_{\infty}(\mathbf{x}, r) = \setdef{\mathbf{y}\in \real^n}{\|\mathbf{y} - \mathbf{x}\|_{\infty} < r }.
\end{align*}
For an interval $I\subseteq \real$, the scalar function $\map{f}{I}{\real}$
is monotone on $I$, if it is either strictly increasing or strictly
decreasing on $I$. For a matrix $A\in \real^{n\times m}$, the image of
$A$ is denoted by $\Img(A)$ and the kernel of $A$ is denoted by
$\Ker(A)$. For a positive-definite diagonal matrix $D\in
\real^{n\times n}$, the $D$-weighted vector-norm is 
\begin{align*}
  \|v\|_D = \|D^{\frac{1}{2}}v\|_2,\qquad\mbox{ for all } v\in \real^n,
\end{align*}
and the induced $D$-weighted matrix-norm is
\begin{align*}
  \|X\|_D = \|D^{\frac{1}{2}}X D^{\frac{-1}{2}}\|_2,\qquad\mbox{ for
  all } X\in \real^{n\times n}. 
  \end{align*}

\subsubsection*{Algebraic graph theory}
Let $G$ denote an undirected graph with nodes $\until{n}$ and edge set
$\mcE$ with cardinality $m$.  Given an arbitrary orientation and ordering
of the edges, let $B$ denote the corresponding incidence matrix defined
component-wise as $B_{kl}=1$ if node $k$ is the sink node of edge $l$ and
as $B_{kl} = -1$ if node $k$ is the source node of edge $l$; all other
elements are zero~\cite[\S~8.3]{CDG-GFR:01}. {\color{black}If $x\in\real^n$ is a
nodal variable} (with $x_i$
associated to the $i$th node), then the flow vector associated to $x$
is the edge variable $B^{\top}x\in\real^m$ (with $(B^{\top}x)_e=x_i-x_j$ associated to the
oriented edge $e=(i,j)$). A \emph{cycle} in $G$ is an ordered sequence of
three or more nodes such that there is an edge between every two
consecutive nodes, the first and last nodes are the same, and no other node
is repeated. An undirected graph $G$ is \emph{cyclic} if it contains at
least one cycle.

For each cycle $\sigma$ in $G$, {\color{black}the length of $\sigma$ is
  denoted by $n_{\sigma}$} and the \emph{signed cycle vector}
$v_{\sigma}\in \{-1,0,1\}^m$ is composed of entries
\begin{align*}
  \left(v_{\sigma}\right)_e = \begin{cases}
    +1, \quad \null &\mbox{ if the edge $e$ is traversed positively by $\sigma$},\\
    -1, &\mbox{ if the edge $e$ is traversed negatively by $\sigma$},\\
    0, &\mbox{otherwise}.
  \end{cases}
\end{align*}
for each $e \in \mathcal E$~\cite[\S~14.2]{CDG-GFR:01}. The \emph{cycle
  space} of $G$ is the {\color{black}subspace of $\real^m$ spanned by
  the signed cycle vectors of all cycles in $G$}; equivalently, the cycle
space is $\Ker(B)$~\cite[\S~14.2]{CDG-GFR:01}.  A set of cycles $\Sigma
=\{\sigma_1,\ldots,\sigma_{m-n+1}\}$ is a \emph{cycle basis} for $G$ if the
corresponding set of signed cycle vectors $\{v_{\sigma_1},\ldots,
v_{\sigma_{m-n+1}}\}$ is a basis for the cycle space. {\color{black} If
  $\Sigma =\{\sigma_1,\ldots,\sigma_{m-n+1}\}$ is a cycle basis for $G$,
  then the length of $\Sigma$ is $\sum_{i=1}^{m-n+1}n_{\sigma_i}$. A cycle
  basis with minimal length is called a \emph{minimum cycle basis}.}

Note that many authors in the graph theory literature use similar but \textit{discrete} 
definitions of cycle space and cycle basis. In \cite{JDH:87,TK-CL-KM-DM-RR-TU-KZ:09}, for 
example, a ``cycle'' is a subgraph where every vertex has even degree, and a ``cycle 
basis'' is a set of ``cycles'' that can generate any ``cycle'' by the logical XOR 
operation, so the ``cycle space'' is a vector space on the field $\text{GF}(2)$ (instead 
of $\real$). Fortunately, a straightforward consequence of \cite[Lemma 
2.4]{TK-CL-KM-DM-RR-TU-KZ:09} is that the signed cycle vectors corresponding to these 
discrete basis ``cycles'' are also a basis for $\Ker(B)$. Therefore, we can rely on this 
literature in future sections for the computation of cycle bases.

Given a connected undirected graph $G$ with cycle basis
$\Sigma=\{\sigma_1,\ldots,\sigma_{m-n+1}\}$, the \emph{cycle-edge matrix}
is
\begin{equation}
  \label{defeq:cycle-edge}
        {C}_{\Sigma} =
        \begin{bmatrix}v^{\top}_{\sigma_1} \\ \vdots \\ v^{\top}_{\sigma_{m-n+1}}\end{bmatrix} \in \real^{(m-n+1)\times m}.
      \end{equation}
      {\color{black} We refer to Lemma~\ref{thm:shift_property} for additional properties
        of the cycle-edge matrix}. Figure~\ref{fig:polygonal-graphs}
      shows some polygonal graphs with oriented edges and cycle bases.
      \begin{figure}[ht]\centering
        \includegraphics[width=.75\linewidth]{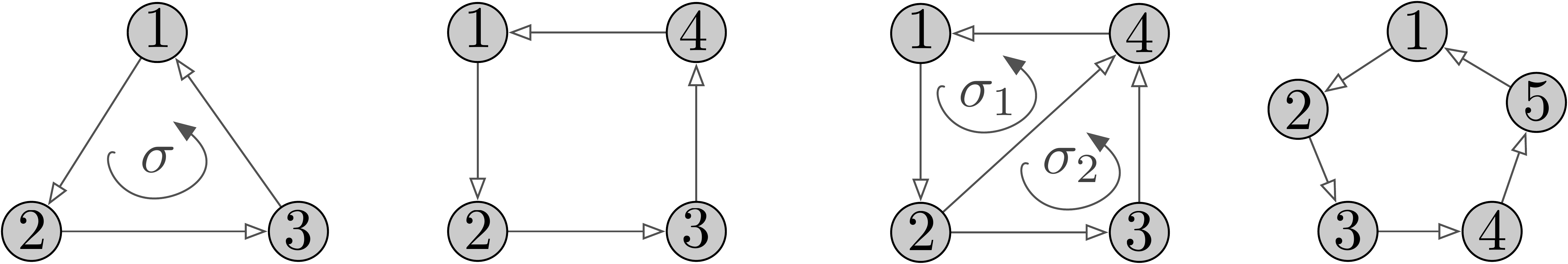}
        \caption{The triangle, the square, the square-with-diagonal,
          and the pentagon graph. The triangle and the
          square-with-diagonal graphs are drawn with a cycle basis.}
        \label{fig:polygonal-graphs}
      \end{figure}

Next, assume each edge $(i,j)$ has a weight $a_{ij}$ and define the
diagonal edge weight matrix $\mcA=\diag(a_{ij})\in\real^{m\times m}$.  The
  Laplacian matrix of $G$ is $L=B\mcA\Bt$. Recall that $L$ is singular and
  its Moore\textendash{}Penrose pseudoinverse $L^{\dagger}$ has the
  following properties: $LL^{\dagger}L=L$,
  $L^{\dagger}LL^{\dagger}=L^{\dagger}$,
  $L^{\dagger}L=(L^{\dagger}L)^{\top}$, and
  $LL^{\dagger}=(LL^{\dagger})^{\top}$. 
For a positive definite diagonal matrix $D\in \real^{m\times m}$, the \emph{$D$-weighted cycle
  projection matrix} $\mathcal{P}_{D}$ is the oblique projection onto $\Ker(B)$ parallel to $\Img(D\mathcal{A}B^{\top})$ given by
\begin{align}\label{eq:cycle-projection}
  \mathcal{P}_D=I_m - D\mathcal{A}B^{\top}(BD\mcA B^{\top})^{\dagger}B.
\end{align}
{\color{black} If $G$ is connected, then the $D$-weighted cycle
  projection matrix $\mathcal{P}_D$ is idempotent and its eigenvalues
  are $0$ and $1$ with algebraic (and geometric) multiplicity $n-1$ and
$m-n+1$, respectively. Moreover, one can show that
  $\Ker(\mathcal{P}_D D\mathcal{A}) = \Img(B^{\top})$. We refer to
  Lemma~\ref{thm:cycle-projection} for the proof of these properties.}  Additional properties of this projection are in~\cite[Theorem 5]{SJ-FB:16h}.

\subsubsection*{The $n$-torus}
Let $\torus^1$ be the unit circle with its standard Riemannian
structure. Given angles $\alpha,\beta\in\Scircle$, the geodesic distance between
$\alpha$ and $\beta$, denoted by $|\alpha-\beta|$, is the length of the
shortest arc in $\Scircle$ connecting $\alpha$ to $\beta$. The
\emph{counterclockwise difference} on $\Scircle$ is the map
$\map{\subscr{d}{cc}}{\Scircle\times\Scircle}{[-\pi,\pi)}$ defined
  by:
\begin{align*}
\subscr{d}{cc}(\alpha,\beta) =
\begin{cases}
  |\alpha -\beta|, \quad \null & \mbox{if the counterclockwise arc from $\alpha$ to
    $\beta$ is shorter than $\pi$},\\
  -|\alpha -\beta|, & \mbox{otherwise}.
\end{cases}
\end{align*}
If we identify $\Scircle \simeq [-\pi,\pi)$, for every $\alpha,\beta\in
  [-\pi,\pi)$, the counterclockwise difference $\subscr{d}{cc}(\alpha,\beta)$
    is given by $\subscr{d}{cc}(\alpha,\beta) =\mathrm{mod}((\beta -
    \alpha) , 2\pi)- \pi$. Let $G$ denote an undirected graph with nodes $\until{n}$ and edge set
$\mcE$ with cardinality $m$.  Given an arbitrary orientation and ordering
of the edges, let $B$ denote the corresponding incidence matrix. If
$\theta\in \torus^n$ (with $\theta_i$ associated to the $i$th node
of $G$), then, by abuse of notation, we write $\theta_i-\theta_j$
to refer to $d_{\mathrm{cc}}(\theta_i,\theta_j)$ and we write {\color{black}$(B^{\top}\theta)$}
to refer to the vector in $\real^m$ defined by
\begin{align}
  \label{def:BTtheta}
  (B^{\top}\theta)_e := \theta_i-\theta_j =
  d_{\mathrm{cc}}(\theta_i,\theta_j),\qquad\mbox{for every }e=(i,j)\in \mathcal{E}.
  \end{align}
The \emph{punctured $n$-torus} $\torus_0^{n}$ is
\begin{align*}
  \torus_0^{n} = \setdef{\theta\in\torus^n}{|\theta_i-\theta_j|<\pi \text{ for each } (i,j) \text{ edge of } G}.
\end{align*}
Note that $\closure(\torus^n_0) = \torus^n$. We also need the notion of quotient space. For every angle
$s\in\Scircle\simeq [-\pi,\pi)$, define the rotation operator
$\map{\mathrm{rot}_s}{\torus^n}{\torus^n}$ by $\mathrm{rot}_s(\theta) =
(\theta_1+s, \ldots,\theta_n+s)^{\top}$.  Using this rotation operator as a
group action, we define the \emph{reduced $n$-torus} as the quotient space
$\torus^n/\Scircle$ and the \emph{reduced punctured $n$-torus} as the
quotient space $\torus_0^n/\Scircle$. For every $\theta\in
\torus^n$, the equivalent class of $\theta$ in the reduced
$n$-torus $\torus^n/\Scircle$ is denoted by $[\theta]$.

\subsection{Equivalence of flow and elastic network problems on the $n$-torus}\label{subsec:equival}

{\color{black}Having introduced the flow network
  problem~\ref{sinflowproblem} and the elastic network
  problem~\ref{torqueproblem} on the $n$-torus, we investigate the
  connection between their solutions. Adopting the analogy that (i)
  supply/demand vectors are torque vectors, and (ii) flow functions are the
  derivatives of elastic energy functions, the next result establishes a
  one-to-one correspondence between the solutions of flow networks and
  elastic networks. }

\begin{theorem}[Equivalence of network problems on the $n$-torus]
  \label{thm:equivalence}%
  Let $G$ be a undirected graph, $q\in \vect{1}_n^{\perp}$ be a
  balanced vector, $\gamma\in\zeropi$ be a phase angle,
  $\{h_e\}_{e\in\mathcal{E}}$ be a family of continuously differentiable
  $2\pi$-periodic odd functions, and
  $\{H_e\}_{e\in \mathcal{E}}$ be a family of twice differentiable
  $2\pi$-periodic even functions such that, for each edge $e$ and each angle
  $\alpha\in(-\gamma, \gamma)$,
  \begin{align*}
    H_{e}(\alpha) - H_e(0) = \int_{0}^{\alpha} h_e(\beta)d\beta
    \qquad\iff\quad \frac{d H_e}{d\alpha}(\alpha) = h_e(\alpha).
  \end{align*}
  Then, for $\theta\in \torus^n$, the following statements are equivalent:
  \begin{enumerate}
  \item\label{p1:flownet} there exists $f\in \real^m$ such that
    $(f,\theta)$ solves the flow network problem~\eqref{eq:flow_form} for
    $(G,\{h_e\}_{e\in \mathcal{E}},q,\gamma)$;
  \item\label{p2:elasticnet} $\theta$ solves the elastic network
    problem~\eqref{eq:elastic_form} for $(G,\{H_e\}_{e\in
      \mathcal{E}},q,\gamma)$.
    \end{enumerate}
\end{theorem}

{\color{black}Theorem~\ref{thm:equivalence} guarantees that any result about the flow
network setting is directly applicable to the elastic network setting
and vice versa. For the rest of this paper, we focus on flow
  network problems on the $n$-torus.}

\subsection{Acyclic flow network problems on the $n$-torus}\label{subsec:acyclic}

{\color{black} In this part, we study the flow network
  problem~\eqref{eq:flow_form} over connected and acyclic graphs. It
  turns out that, for acyclic networks, the solvability of the flow network
  problem~\eqref{eq:flow_form} on the $n$-torus 
  can be completely characterized using a simple algebraic
  inequality. Moreover, one can find a closed-form formula for
  the flows.}

\begin{theorem}[Acyclic flow network problem on the $n$-torus]\label{thm:acyclic}
  Consider the flow network problem~\eqref{eq:flow_form} for
  $(G,\{h_e\}_{e\in \mathcal{E}},\pactive,\gamma )$ and suppose that
  $G$ is a connected and acyclic graph and 
  each function $h_e$, $e\in \mathcal{E}$ is monotone on the interval
  $[-\gamma,\gamma]$. Then the following statements are equivalent:
  \begin{enumerate}
  \item{\color{black}\label{p1:acyclic} $\left|(B^{\top}L^{\dagger}\pactive)_e\right|\le |h_e(\gamma)|$, for
    all $e\in \mathcal{E}$,} 
  \item\label{p2:acyclic} there exists a unique solution
    $(f^*,\theta^*)$ for the problem~\eqref{eq:flow_form} on the $n$-torus. 
  \end{enumerate}
  Moreover, if any of the equivalent conditions~\ref{p1:acyclic}
  or~\ref{p2:acyclic} holds, then $f^* = \mathcal{A}B^{\top}L^{\dagger}\pactive$.  
\end{theorem}
{\color{black} Compared to the literature,
  Theorem~\ref{thm:acyclic} extends~\cite[Theorem
  2]{FD-MC-FB:11v-pnas} and~\cite[Corollary 2]{DM-MT-DW:17} to the flow
  networks on the $n$-torus with arbitrary monotone flow
  functions. Theorem~\ref{thm:acyclic} highlights the role of the
  network's cycle structure in multiplicity of solutions of the flow
  network problem~\eqref{eq:flow_form} on the $n$-torus.} In the rest
of this paper, we assume that the graph $G$ is connected (without loss
of generality) and cyclic.

\myclearpage
\section{Algebraic graph theory on the $n$-torus}
\label{sec:winding-partition}

Algebraic graph theory provides a widely established framework to study
$\real$-valued functions on graphs using linear algebraic
structures~\cite{NB:97}. In this section, we generalize classical concepts
from algebraic graph theory to the setting where the graph nodal variables
take value on the circle $\Scircle$.  The starting idea is to extend the
classic Kirchhoff's voltage law (KVL) to $\Scircle$-valued functions on the
graph. For $\real$-valued functions, KVL states that the sum of nodal
differences along each simple cycle is zero. Remarkably, $\Scircle$-valued
functions on a graph exhibit richer behavior. For such functions, one can
associate an integer, called the winding number, to each simple cycle and
show that the sum of angular differences along the cycle is equal to its
winding number. Loosely speaking, the winding number of a cycle counts the
number of turns (in counterclockwise direction) of the $\Scircle$-valued
function along the cycle. We start our treatment by introducing the notion
of winding number.

\begin{definition}[Winding number, vector, and map]\label{def:wind}
  Let $G$ be a cyclic connected undirected graph and
  $\theta\in\torus^n_0$. Then
\begin{enumerate}
\item\label{def:wind_num} for every cycle $\sigma$ in $G$ with $n_{\sigma}$
  nodes, the \emph{winding number of $\theta$ along $\sigma$} is
  \begin{align*}
    w_{\sigma}(\theta) = \frac{1}{2\pi}\sum_{i=1}^{n_\sigma}
    \subscr{d}{cc}(\theta_{i},\theta_{i+1}), 
  \end{align*}
  where we assume $\sigma=(1,\ldots,n_\sigma,1)$ without loss of generality
  and $\theta_{n_\sigma+1}=\theta_1$ by convention; 
   
 \item\label{def:wind_covector} for every cycle basis $\Sigma
   =\{\sigma_1,\ldots,\sigma_{m-n+1}\}$ in $G$, the \emph{winding vector of
     $\theta$} along the basis $\Sigma$ is the vector
   \begin{align*}
     \begin{bmatrix}w_{\sigma_1}(\theta),&\ldots,&w_{\sigma_{m-n+1}}(\theta)\end{bmatrix}^{\top};
   \end{align*}
   
\item\label{def:wind_map} for every cycle basis $\Sigma
  =\{\sigma_1,\ldots,\sigma_{m-n+1}\}$ in $G$, the \emph{winding map}
  $\map{\mathbf{w}_\Sigma}{\torus_0^n}{\integer^{m-n+1}}$ along the basis
  $\Sigma$ is 
\begin{align*}
     \mathbf{w}_\Sigma (\theta)=\begin{bmatrix}
     w_{\sigma_1}(\theta),&\ldots,&w_{\sigma_{m-n+1}}(\theta)\end{bmatrix}^{\top}.
   \end{align*}
\end{enumerate}
\end{definition} {\color{black} The notion of winding number is a
  classical concept in mathematics and its rigorous definition dates back
  to the early work of Alexander~\cite{JWA:28}. In differential topology,
  the notions of \emph{index of a smooth vector field} and \emph{degree of
    a continuous map} can be considered as generalizations of the winding
  number~\cite{VG-AP:10}. To the best of our knowledge,
  \cite{DAW-SHS-MG:06} is the first work that introduces the notion of
  winding number to characterize the fixed points of a coupled oscillator
  network. The connection between the number of fixed points of the
  Kuramoto model and the winding number of the solution has been first
  explored in~\cite{JO-PFG:10} for ring graphs. More recently, the notions
  of winding number and winding vector have been used in~\cite{DM-MT-DW:17}
  and~\cite{RD-TC-PJ:16} to study multistability in networks of Kuramoto
  oscillators.}

The winding number, winding vector, and winding map are invariant under the
rotation operator $\mathrm{rot}_s$, for $s\in [-\pi,\pi)$, that is, every
  cycle $\sigma$ and every $\theta\in \torus^n_0$ satisfy
  $w_{\sigma}(\theta) = w_{\sigma}(\mathrm{rot}_s(\theta))$. Therefore the
  winding number, winding vector, and winding map are well-defined on the
  reduced punctured $n$-torus $\torus_0^n/\Scircle$. The winding vector of
  $\theta$ along the basis $\Sigma$ collects in one quantity all the
  information required to compute the winding number of each cycle of the
  graph. The winding number, vector and map $\mathbf{w}_{\Sigma}$ and its
  image $\Img(\mathbf{w}_{\Sigma})$ depend on the cycle basis $\Sigma$
  chosen for $G$. However, one can define a basis-independent winding map
  on $\torus^n_0$ (see Appendix~\ref{app:basis-indep}).
  

  \begin{remark}[Algorithms for computing cycle
    basis]\label{rem:cyclebasis} Several algorithms exist in the literature
    for computing the cycle basis of an undirected graph. A simple approach
    is based on finding a spanning tree for $G$ and runs in
    $\mathcal{O}(mc)$ time, where $c$ is the length of the largest cycle in
    the graph. {\color{black} To find a cycle basis with the desired
      extremum properties, one can use the greedy algorithm over the set of
      all cycles in the graph. However, for many graph families, the number
      of cycles grows exponentially with the number of nodes (for
      instance the number of cycles in the complete graph with $n$ nodes is
      $\sum_{k=3}^{n}\frac{1}{2}(k-1)! {n\choose k}$ which grows
      exponentially with $n$). Therefore, the run time of the greedy
      algorithm can grow exponentially with the number of nodes $n$.} To
    find the minimum cycle basis of undirected unweighted graphs,
    Horton~\cite{JDH:87} provided the first-known polynomial-time
    algorithm; Horton's algorithm runs in $\mathcal{O}(m^3n)$ time or, in
    its improvement proposed in~\cite{KM-DM:09}, in
    $\mathcal{O}(m^2n/\mathrm{log}(n) + n^2m)$. Finally, \cite{ME-CL-RR:07}
    proposes an algorithm to construct a cycle basis of length
    $\mathcal{O}(m \log(n)\log\log(n))$ for unweighted graphs. We refer
    to~\cite{TK-CL-KM-DM-RR-TU-KZ:09} for a survey on algorithms and
    computational complexity of computing minimum cycle bases.
    \end{remark}

 To illustrate the notion of winding number we consider the triangle graph
 with cycle $\sigma=(1,2,3,1)$ in Figure~\ref{fig:polygonal-graphs}.  Since
 the cycle space of the triangle graph is spanned by $\sigma$, the winding
 map with respect to $\sigma$ is $\mathbf{w}_\Sigma (\theta) =
 w_{\sigma}(\theta)$.  One can show that, if the angles
 $(\theta_1,\theta_2,\theta_3)\in\torus^3_0$ are contained in an arc of
 length less than or equal to $\pi$, then $w_{\sigma}(\theta) = 0$ and,
 otherwise, $w_{\sigma}(\theta) \in\{-1,+1\}$;
 see~Figure~\ref{fig:triangle-winding}.
 \begin{figure}[ht]\centering
   \includegraphics[width=.75\linewidth]{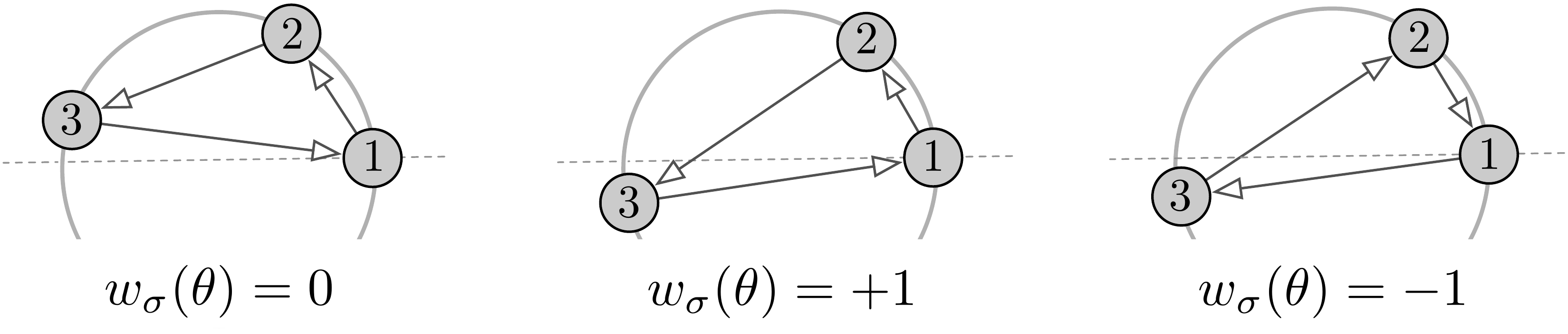}
   \caption{For the triangle graph, if the angles are strictly contained in
     an half circle, the winding number is zero. Otherwise, the winding
     number is $\pm1$ depending upon the node numbering and edge
     orientations. }
   \label{fig:triangle-winding}
 \end{figure}
 For example, for $\phi= (0,\tfrac{\pi}{3}, \tfrac{2\pi}{3})^{\top}$ and
 $\psi = (0,\tfrac{2\pi}{3}, \tfrac{4\pi}{3})^{\top}$, simple book-keeping
 shows
 \begin{align*}
   w_{\sigma}(\phi) &= 
   \tfrac{1}{2\pi}(\tfrac{\pi}{3}+\tfrac{\pi}{3}-\tfrac{2\pi}{3}) = 0,
   \enspace\text{and}\enspace
   w_{\sigma}(\psi) =
   \tfrac{1}{2\pi}(\tfrac{2\pi}{3}+\tfrac{2\pi}{3}+\tfrac{2\pi}{3}) = 1.
 \end{align*}

Now, we prove some useful properties for the winding number and winding
vectors.

\begin{theorem}[Kirchhoff's angle law]\label{thm:wind_prop}
  Let $G$ be a cyclic connected undirected graph with $n$ nodes, and $m$
  edges.  Let $\sigma$ be a cycle on $G$ with $n_{\sigma}$ nodes and
  $\Sigma=\{\sigma_1,\ldots,\sigma_{m-n+1}\}$ be a cycle basis for $G$. Let
  $\theta\in \torus^n_0$. Then
  \begin{enumerate}
  \item\label{p1:wind_integer_bound} the winding number
    $w_{\sigma}(\theta)$ is an integer and $\ds|w_{\sigma}(\theta)| \le
    \left\lceil\frac{\ n_{\sigma}}{2}\right\rceil -1$;
  \item\label{p2:wind_piece_constant} the winding map $\mathbf{w}_\Sigma$
    is piecewise constant and
    \begin{equation*}
      \left|\Img(\mathbf{w}_\Sigma)\right|\le
    \prod_{i=1}^{m-n+1} \left(2\left\lceil\frac{\ n_{\sigma_i}}{2}\right\rceil -1\right).
    \end{equation*}
  \end{enumerate}
  \end{theorem}

Theorem~\ref{thm:wind_prop}\ref{p1:wind_integer_bound} generalizes the
classic Kirchhoff's voltage law (KVL) to the setting of graphs with nodal
variables in $\Scircle$: the sum of the nodal differences along every
simple cycle is equal to $2\pi w$, where $w\in \integer$ is the winding
number of the cycle.


In order to illustrate the Kirchhoff's angle law and various properties in
Theorem~\ref{thm:wind_prop}, we get back to the polygonal graphs in
Figure~\ref{fig:polygonal-graphs}.  We consider the square graph with a
diagonal edge and define its cycles $\sigma_1=(1,2,4,1)$,
$\sigma_2=(2,3,4,2)$, and $\sigma_3=(1,2,3,4,1)$, as
in~Figure~\ref{fig:polygonal-graphs}.  Note that only two cycles are
independent in the sense that the signed cycle vectors satisfy
$v_{\sigma_1}+v_{\sigma_2} = v_{\sigma_3}$ and so the set $\Sigma =
\{v_{\sigma_1},v_{\sigma_2}\}$ is a basis for cycle space of $G$. The
winding map of the graph $G$ with respect to basis $\Sigma$ is given by
$\mathbf{w}_\Sigma (\theta) = \begin{bmatrix}w_{\sigma_1} &
  w_{\sigma_2}\end{bmatrix}^\top$.
Theorem~\ref{thm:wind_prop}\ref{p1:wind_integer_bound} implies that
$|w_{\sigma_i}(\theta)|\le \lceil\frac{3}{2}\rceil -1 = 1$, for every
$\theta\in \torus_0^4$ and every $i\in\{1,2\}$. Therefore,
$\Img(\mathbf{w}_\Sigma)\subseteq \{-1,0,+1\}^2$. Moreover, on can see that
the winding vectors $\begin{bmatrix}1 & 1\end{bmatrix}^\top$ and
  $\begin{bmatrix}-1 & -1\end{bmatrix}^\top$ are not possible for the
    square graph with a diagonal. This implies that
    $\Img(\mathbf{w}_\Sigma)\subset \{-1,0,+1\}^2$ and therefore, in
    general, the inequality in
    Theorem~\ref{thm:wind_prop}\ref{p2:wind_piece_constant} can be strict.

Since winding vector is a piecewise constant map, one can partition
the $n$-torus into the regions where the winding vector assume a fixed integer
vector. These regions are called winding cells.

\begin{definition}[Winding cell]\label{def:windingcelldef}
  Consider the cyclic connected undirected graph $G$ with cycle basis
  $\Sigma$ and winding map $\mathbf{w}_\Sigma$. For
  $\mathbf{u}\in\Img(\mathbf{w}_\Sigma)$, the \emph{$\mathbf{u}$-winding
    cell} is the subset of $\torus_0^n$ defined by
  \begin{align*}
    \Omega^G_{\mathbf{u}} = \mathbf{w}_\Sigma^{-1}(\mathbf{u}). 
  \end{align*}
\end{definition}

For the triangle graph in Figure~\ref{fig:polygonal-graphs}, one can
visualize the winding cells. Note that the triangle graph with only cycle
  $\sigma=(1,2,3,1)$ has $\Img(\mathbf{w}_\Sigma) = \{-1,0,+1\}$.
  Therefore, the winding cells of $\torus_0^3$ are
  \begin{align*}
   \Omega^G_{-1}  &=\bigsetdef{\theta=(\theta_1,\theta_2,\theta_3)^{\top}\in
  		\torus_0^3}{\sum_{i=1}^{3}\subscr{d}{cc}(\theta_i,\theta_{i+1})=-2\pi}, \\
    \Omega^G_{0} &=\bigsetdef{\theta=(\theta_1,\theta_2,\theta_3)^{\top}\in
      \torus_0^3}{\mbox{there exists an arc of
        length }\pi \mbox{ containing }\theta_1,\theta_2,\theta_3}, \\
    \Omega^G_{1} &=\bigsetdef{\theta=(\theta_1,\theta_2,\theta_3)^{\top}\in
      \torus_0^3}{\sum_{i=1}^{3}\subscr{d}{cc}(\theta_i,\theta_{i+1})=2\pi}.
  \end{align*}
  Figure~\ref{fig:T3partition-phase} illustrates these winding cells.	
  \begin{figure}[ht]\centering
    \includegraphics[width=\linewidth]{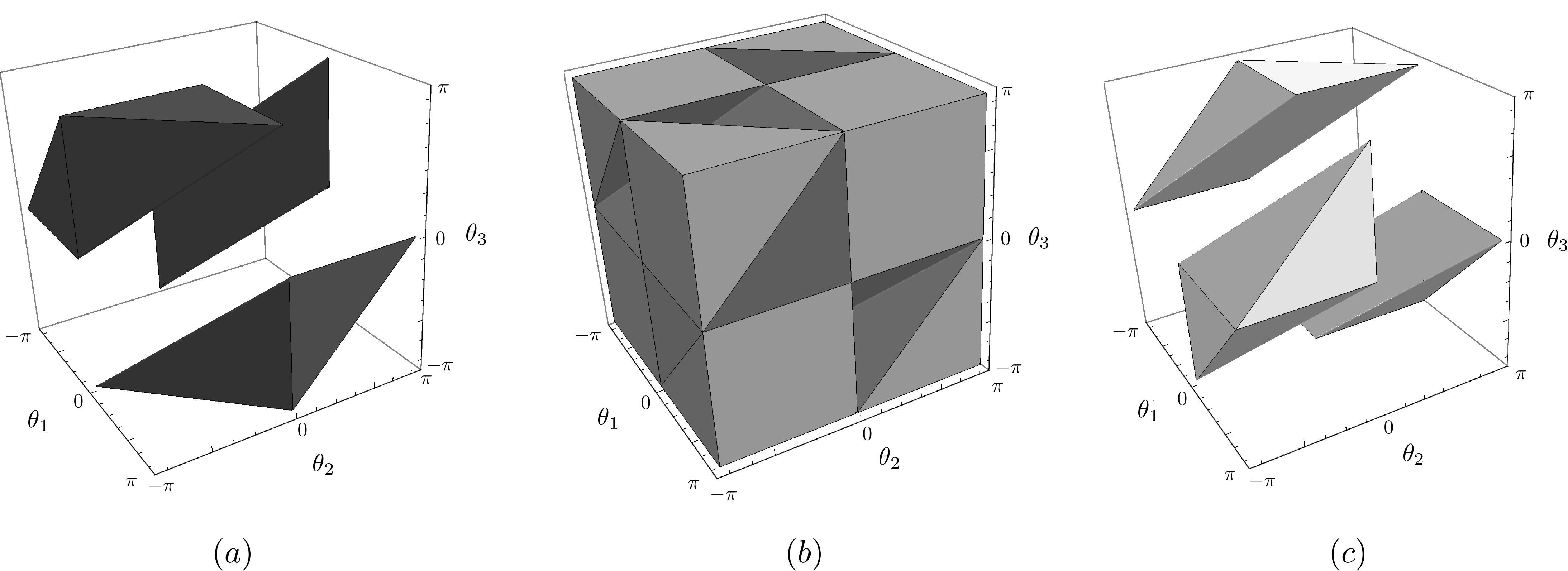}
    \caption{The winding cells $\Omega_{-1}^G$, $\Omega_0^G$, and
      $\Omega_1^G$, induced by the triangle graph on $\torus_0^3$ are
      shown in Figures (a), (b), and (c), respectively. The three axes are
      the three angles $\theta_1, \theta_2, \theta_3$, each taking values
      in the interval $[-\pi, \pi)$. These axes are periodic on the
        3-torus, so all three winding cells are path-connected, despite
        their disconnected representations in $\real^3$.}
      \label{fig:T3partition-phase}
  \end{figure}

  We are now ready to state the main result of this section.  Since the
  winding map is piecewise constant and finite valued, it partitions the
  $n$-torus into a finite number of regions. {\color{black} This partition
    captures the connection between the geometry of the $n$-torus and the
    cycle structure of the network. As we will see later, it also
    plays a crucial role in localizing the solutions of flow network
    problems on the $n$-torus. }

\begin{theorem}[Winding partition of the $n$-torus]
  \label{thm:partition}
    Consider the cyclic connected undirected graph $G$ with cycle basis
    $\Sigma$ and winding map $\mathbf{w}_\Sigma$.  Then the set
    $\setdef{\closure(\Omega^G_{\mathbf{u}})}{\mathbf{u}\in\Img(\mathbf{w}_\Sigma)}$
    is a partition of $\torus^n$, called the \emph{$G$-winding partition}
    of $\torus^n$, in the sense that
\begin{enumerate}

  \item \label{p2:union} $\torus^n = \bigcup_{\mathbf{u}\in
    \Img(\mathbf{w}_\Sigma)} \closure(\Omega^G_{\mathbf{u}})$;

  \item \label{p3:intersection} for every $\mathbf{u},\mathbf{v}\in
    \Img(\mathbf{w}_\Sigma)$, $\mathbf{u}\ne\mathbf{v}$ implies
    $\Omega^G_{\mathbf{u}}\cap \Omega^G_{\mathbf{v}}=\emptyset$.

\end{enumerate}
\end{theorem}

{\color{black}While Definition~\ref{def:windingcelldef} is mathematically
  rigorous, it does not provide insight into the topological properties of
  the winding cells. These insights are particularly important for
  visualization of the winding partition of the $n$-torus developed in
  Theorem~\ref{thm:partition}. In the next theorem, we provide a
  continuous bijection which allows one to visualize the winding cells as convex polytopes in Euclidean
  spaces. Using this characterization, we show that the winding cells enjoy
  several useful properties. (Recall the definition of cycle-edge matrix
  from equation~\eqref{defeq:cycle-edge} and of the vector
  {\color{black}$(B^{\top}\theta)$} from equation~\eqref{def:BTtheta}.)}

\begin{theorem}[Polytopic characterization of winding cells]
  \label{thm:linerization+correspondence}
  Consider a connected cyclic undirected graph $G$ with cycle basis
  $\Sigma$, winding map $\mathbf{w}_\Sigma$, and cycle-edge matrix
  ${C}_{\Sigma}\in \real^{(m-n+1)\times m}$. Pick $\mathbf{u}\in
  \Img(\mathbf{w}_\Sigma)$. Then
  \begin{enumerate}
  \item\label{p1:diff} for every $[\theta]\in \Omega^G_{\mathbf{u}}/\Scircle$, there
    exists a unique $\mathbf{x}\in \vect{1}_n^{\perp}$ such that {\color{black}$(B^{\top}\theta) =
    B^{\top}\mathbf{x}+2\pi C^{\dagger}_{\Sigma}\mathbf{u}$};
  \item\label{p1.5:connec-comp} define the open convex polytope
    \begin{align*}
      P_{\mathbf{u}} =\setdef{\mathbf{x}\in\vect{1}_n^{\perp}}
      {\|B^{\top}\mathbf{x}+2\pi C^{\dagger}_{\Sigma}\mathbf{u}\|_{\infty}<\pi}.
    \end{align*}
   {\color{black}Then, the map $[\theta] \mapsto x$ defined in part~\ref{p1:diff} is a continuous bijection
    between the reduced winding cell $\Omega^G_{\mathbf{u}}/\Scircle$
    and the polytope $P_{\mathbf{u}}$, and it is a homeomorphism on
    every compact subset of $\Omega^G_{\mathbf{u}}/\Scircle$.}

  \end{enumerate}
\end{theorem}
In Figure~\ref{fig:T34partition-dcc}(a) and
Figure~\ref{fig:T34partition-dcc}(b), we plot the reduced winding
cells $\Omega_{-1}^G/\Scircle$, $\Omega_0^G/\Scircle$, and
$\Omega_1^G/\Scircle$ induced by the triangle graph
and square graph on $\torus^3_0/\Scircle$ and $\torus^4_0/\Scircle$,
respectively. {\color{black}In Figure~\ref{fig:T34partition-dcc}(c), we
  plot the reduced winding cells induced by the square-with-diagonal
  graph (Figure~\ref{fig:polygonal-graphs}) on
  $\torus^4_0/\Scircle$. The vector on each winding cell indicates
  its associated winding vector $\mathbf{w}_{\Sigma}$, for
  $\Sigma =\{\sigma_1,\sigma_2\}$.} Figure \ref{fig:T34partition-dcc}
plots the winding cells as functions of counterclockwise distances
$d_{cc}(\theta_i, \theta_{i + 1})$ rather than phases $\theta_i$,
effectively modding out uniform rotations.

  \begin{figure}[ht]\centering
    \includegraphics[width=\linewidth]{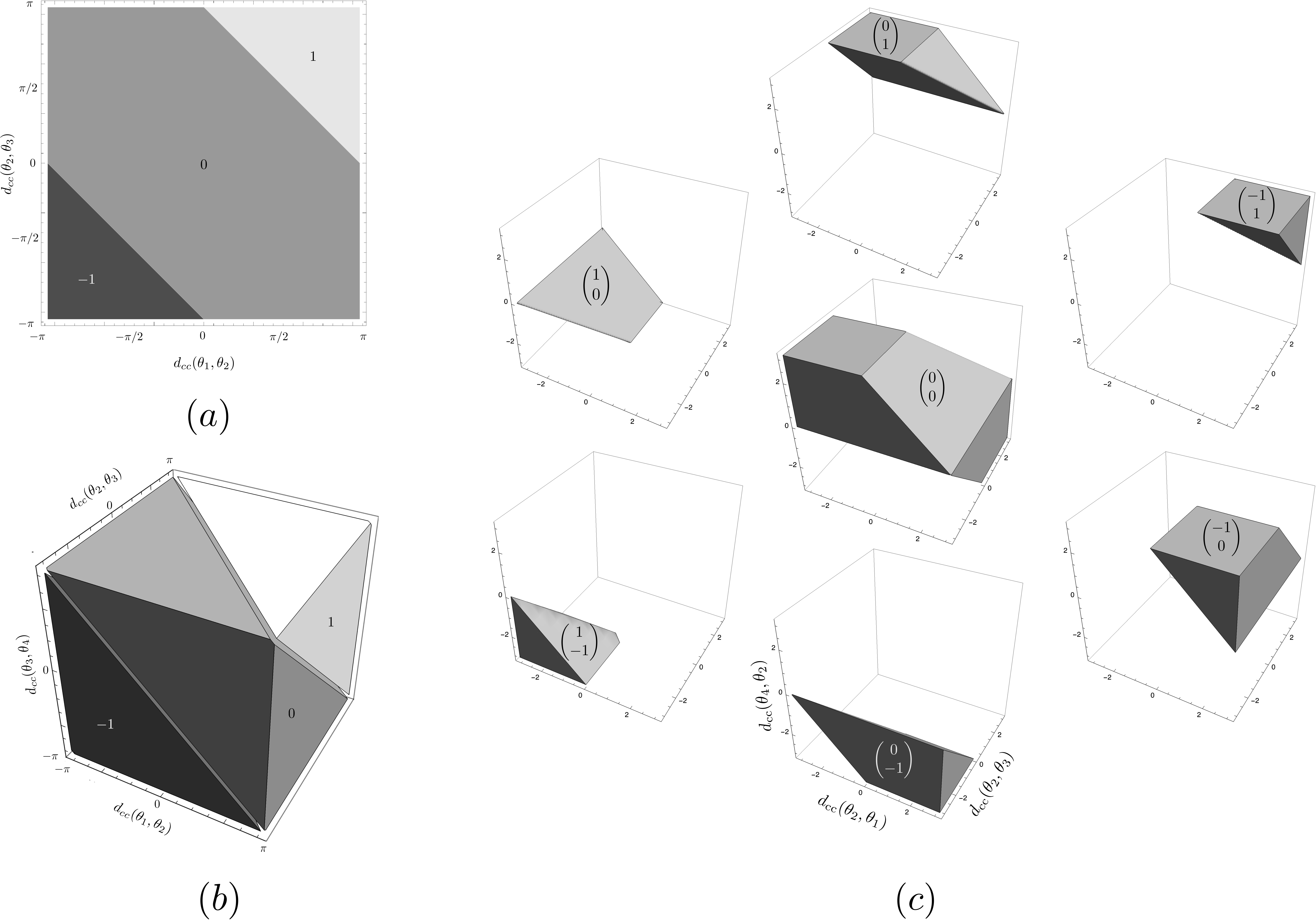}
    \caption{{\color{black}The reduced winding cells on the triangle graph (a), the square
      graph (b), and the square-with-diagonal graph (c) on
      $\torus_0^3/\Scircle$, $\torus_0^4/\Scircle$, and $\torus^4_0/\Scircle$, respectively. The
      number (vector) on each region indicates the winding number
      (winding vector) of that region. In
      all plots, each axis corresponds to an edge in the graph and indicates
      the counter-clockwise difference between the incident nodes, on the
      interval $[-\pi, \pi)$.}}
      \label{fig:T34partition-dcc}
  \end{figure}

We conclude this section by some comments about the role of cycle basis
$\Sigma$ in our framework.
First, while the winding map $\mathbf{w}_{\Sigma}$ in
Definition~\ref{def:wind}~\ref{def:wind_map} depends upon $\Sigma$, the
winding partition of $\torus^n$ in Theorem~\ref{thm:partition} is
independent of it, in the sense that, given any other cycle basis
$\Sigma'$, each winding cell $\Omega^G_{\mathbf{u}}$, for $\mathbf{u}\in
\Img(\mathbf{w}_\Sigma)$, is equal to a cell $\Omega^G_{\mathbf{u}'}$ for
an appropriate $\mathbf{u}'\in \Img(\mathbf{w}_{\Sigma'})$.
Second, note that the upper bound on the cardinality of
$\Img(\mathbf{w}_\Sigma)$ given in
Theorem~\ref{thm:wind_prop}\ref{p2:wind_piece_constant} depends upon
$\Sigma$.
While computing the cycle basis that minimizes this upper bound appears
computationally complex, it is simple to provide the following
approximation. Adopting the shorthand $|\Sigma|=m-n+1$, we compute
\begin{equation*}
  |\Img(\mathbf{w}_\Sigma)|  \le
  \prod_{i=1}^{|\Sigma|} \left(2\left\lceil\frac{\ n_{\sigma_i}}{2}\right\rceil
  -1\right) 
  \le \prod_{i=1}^{|\Sigma|} \left(n_{\sigma_i}+1\right)\le \left(
  \frac{\sum_{i=1}^{|\Sigma|} (n_{\sigma_i}+1)}{|\Sigma|}\right)^{|\Sigma|},
\end{equation*}
where we used the inequality of arithmetic
and geometric means.  Now, the classic result~\cite[Theorem 6]{JDH:87}
states that the length of a minimum cycle bases of an unweighted graph is
at most $3(n-1)(n-2)/2$, so that
\begin{align*}
  |\Img(\mathbf{w}_\Sigma)| \le
  \left(\frac{3(n-1)(n-2)}{m-n+1}+1\right)^{m-n+1}.
\end{align*}
Alternatively, the length of minimum cycle basis of an unweighted graph is
known~\cite{ME-CL-RR:07} to be in $\mathcal{O}(m\log(n)\log(\log(n)))$ so
that there exists a constant $C$ such that
\begin{align*}
  |\Img(\mathbf{w}_\Sigma)| \le \big( C \log n \log\log n \big)^{m-n+1}.
\end{align*}
If the graph admits a basis of cycles with length at most $n_{\sigma}$,
then
\begin{align*}
  |\Img(\mathbf{w}_\Sigma)| \le (2\lceil{n_{\sigma}}/2\rceil -1)^{m-n+1}.
\end{align*}
For example, the complete graph $K_n$ has a cycle basis with $n_{\sigma}=3$
and so, given $m-n+1=(n-1)(n-2)/2$, we know
$|\Img(\mathbf{w}_\Sigma)|\le3^{(n-1)(n-2)/2}$.  The two-dimensional grid
graph $G_{h,k}$ with $n=hk$ nodes and $m=2hk-h-k$ edges, has a cycle basis
with $n_\sigma=4$ so that $|\Img(\mathbf{w}_\Sigma)|\le 3^{hk-h-k+1}$.

\myclearpage

\section{Localization and decomposition of network flows}
\label{sec:at-most-uniqueness}

In this section, we use the tools of Section~\ref{sec:winding-partition} to
study flow network problems on the $n$-torus. Recall that
Theorem~\ref{thm:equivalence} ensures that any result about the flow
network problem is directly applicable to the elastic networks setting.

We first use the winding partition to localize the solutions of the flow
network problem~\eqref{eq:flow_form}. In particular, we show that if the
flow functions are monotone, then flow network problem~\eqref{eq:flow_form}
has at most one solution inside each winding cell.
  
\begin{theorem}[At most uniqueness of solutions]\label{thm:at_most_unique}
 Consider the flow network problem~\eqref{eq:flow_form} for
 $(G,\{h_e\}_{e\in \mathcal{E}},\pactive,\gamma)$ and suppose that each flow function $h_e$, $e\in \mathcal{E}$, is monotone on the interval
 $[-\gamma,\gamma]$. Let $\Sigma$ be a cycle basis for $G$. Then, for any winding vector
 $\mathbf{u}\in\Img(\mathbf{w}_{\Sigma})$, there exists at most one
 solution $(f,\theta)$ for flow network problem~\eqref{eq:flow_form}, such
 that $\theta\in \Omega^{G}_{\mathbf{u}}$.
\end{theorem}

{\color{black}In the context of Kuramoto coupled
  oscillators~\eqref{eq:kuramoto-assoc}, it is known that, if the network
  has a planar topology, then the locally stable synchronous trajectories
  have distinct winding
  vectors~\cite{RD-TC-PJ:16,DM-MT-DW:17}. Theorem~\ref{thm:at_most_unique}
  can be considered as a generalization of these results to flow networks
  on the $n$-torus with arbitrary topology and arbitrary monotone flow
  functions.} For graphs with a cycle basis of given length, the angle
constraint~\eqref{eq:f-constraints} can be used to specify the winding
cells in which there is no solution for the flow network
problem~\eqref{eq:flow_form}.

{\color{black}
\begin{corollary}[Upper bound on the number of solutions]
  \label{cor:graphs-short-cycles}
  Consider the flow network problem~\eqref{eq:flow_form} for
  $(G,\{h_e\}_{e\in \mathcal{E}},\pactive,\gamma)$. Suppose that each flow
  function $h_e$, $e\in \mathcal{E}$, is monotone on $[-\gamma,\gamma]$ and
  $G$ has a cycle basis $\Sigma=\{\sigma_1,\ldots,\sigma_{m-n+1}\}$ with
  maximum cycle length $k$. For $\mathbf{u}\in\Img(\mathbf{w}_{\Sigma})$,
  the following statements hold:
  \begin{enumerate}
  \item\label{p1:shortbasis} if $\|\mathbf{u}\|_{\infty} \le \left\lfloor
    \frac{k\gamma}{2\pi} \right\rfloor $, then problem~\eqref{eq:flow_form}
    has at most one solution $(f,\theta)$ with $\theta\in
    \Omega^G_{\mathbf{u}}$;
  \item\label{p2:shortbasis} if $\|\mathbf{u}\|_{\infty} > \left\lfloor
    \frac{k\gamma}{2\pi} \right\rfloor$, then problem~\eqref{eq:flow_form}
    has no solution $(f,\theta)$ with $\theta\in \Omega^G_{\mathbf{u}}$; and
    \item\label{p3:upperbound} Problem~\eqref{eq:flow_form} has at most
      $\prod_{i=1}^{m-n+1} \left\lfloor \frac{n_{\sigma_i}\gamma}{2\pi}
      \right\rfloor $ solutions $(f,\theta)\in \real^m\times \torus^n$.
    \end{enumerate}
\end{corollary}}

For the special case $\gamma<\pi/2$,
Corollary~\ref{cor:graphs-short-cycles}~\ref{p1:shortbasis} and~\ref{p2:shortbasis} applies to the complete
graphs, complete bipartite graphs, and two-dimensional grids and shows
that the flow network problem has at most one solution, see
Figure~\ref{fig:graph-examples}. {\color{black} In the literature,
  similar upper bounds for the number of stable solutions of the
  Kuramoto model~\eqref{eq:kuramoto-assoc} have been obtained for
  planar networks~\cite{DM-MT-DW:17} and for simple
  cycles~\cite{RD-TC-PJ:16}. Corollary~\ref{cor:graphs-short-cycles}\ref{p3:upperbound}
  generalizes the results in~\cite{DM-MT-DW:17,RD-TC-PJ:16} to provide
  an upper bound on the number of solutions of flow networks with
  arbitrary topology and arbitrary monotone flow functions. }

\begin{figure}[h]\centering
  \subfloat
           {\includegraphics[width=.2\linewidth,height=7em,keepaspectratio]{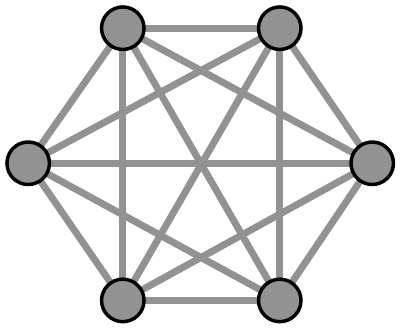}}%
           \hfil
  \subfloat
           {\includegraphics[width=.2\linewidth,height=7em,keepaspectratio]{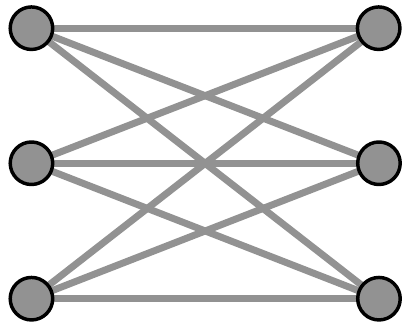}}%
           \hfil
  \subfloat
           {\includegraphics[width=.42\linewidth,height=6em,keepaspectratio]{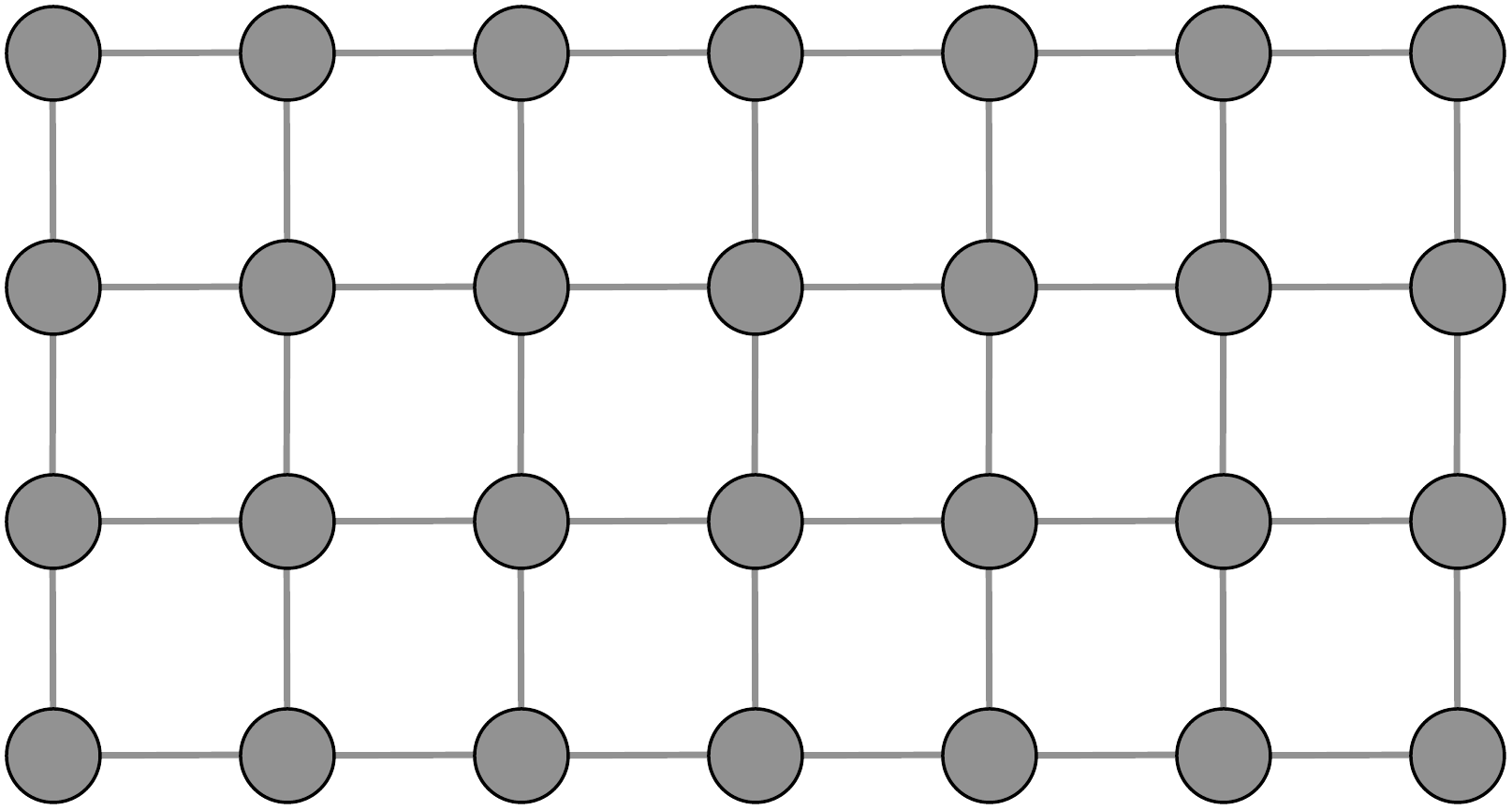}}  
           \caption{Corollary~\ref{cor:graphs-short-cycles} implies
             that, the flow network problem~\eqref{eq:flow_form} with
             $\gamma<\pi/2$ on complete graphs, complete bipartite graphs, and
             two-dimensional grid graphs, have at most one solution, independently of supply/demand vector and
             edge weights.}\label{fig:graph-examples}
\end{figure}

Next, we propose a natural decomposition for the edge space $\real^m$ and
describe how it relates to the flow network
problem~\eqref{eq:flow_form}. For every connected undirected graph $G$, the
edge space of $G$ can be uniquely decomposed as follows:
\begin{align}\label{eq:decompos}
  \real^m = \Img(\mathcal{A}B^{\top})  \oplus \Ker(B).
\end{align}
This oblique decomposition is a special case of the classic orthogonal
decomposition into cutset and cycle space, studied in algebraic graph
theory, see~\cite[\S 2]{NB:97}.  Let $f$ be a flow for the flow network
problem~\eqref{eq:flow_form} on the graph $G$. For every simple cycle
$\sigma$ in $G$, the \emph{loop flow} associated with the cycle $\sigma$ is
the scalar $v_{\sigma}^{\top} f$, where $v_{\sigma}$ is the signed cycle
vector associated to the cycle $\sigma$.

\begin{theorem}[Decomposition of flows]
  \label{thm:prop-power-flow}
  Consider the flow network problem~\eqref{eq:flow_form} for
  $(G,\{h_e\}_{e\in \mathcal{E}},\pactive,\gamma)$ and suppose that each
  flow function $h_e$, $e\in \mathcal{E}$, is monotone on the interval
  $[-\gamma,\gamma]$. Let $\Sigma$ be a cycle basis for $G$ and suppose that $(f,\phi),(g,\psi)\in\real^m\times
  \torus^n$ are two solutions for flow network
  problem~\eqref{eq:flow_form}. Then the following statements hold:
  \begin{enumerate}\setcounter{enumi}{1}
    \item\label{p1:decomposition} $f$ can be decomposed uniquely as
    \begin{align*}
      f= \supscr{f}{cut} + \supscr{f}{cyc},
    \end{align*}
    where $\supscr{f}{cut} = \mathcal{A}B^{\top}L^{\dagger}\pactive \in
    \Img(\mathcal{A}B^{\top})$ is called the \emph{cutset flow} and $\supscr{f}{cyc} \in
    \Ker(B)$ is called the \emph{cycle flow}.
  \item\label{p3:f-unique} $f-g\in\Ker(B)$, that is,
    flows differ by a cycle flow;
      \item\label{p4:f-unique-to-u} $f=g$ $\enspace\iff\enspace$
        $\mathbf{w}_{\Sigma}(\phi) = \mathbf{w}_{\Sigma}(\psi)$
        $\enspace\iff\enspace$ $\phi=\mathrm{rot}_s(\psi)$, for some
        $s\in[-\pi,\pi)$;
  \item\label{p5:flow_comparison} if $G$ has exactly one cycle $\sigma$ and
    each flow function $h_e$, $e\in\mathcal{E}$, is strictly increasing,
    then
    \begin{align*}
      w_{\sigma}(\phi) > w_{\sigma}(\psi) \quad \iff \quad
      v_{\sigma}^{\top}f > v_{\sigma}^{\top}g, 
    \end{align*}
    that is, loop flows are strictly increasing with respect to
    winding numbers.
  \end{enumerate}
\end{theorem}

\newcommand{\Fsd}{\subscr{F}{sd}}
\newcommand{\Hmin}{\subscr{L}{min}}
\newcommand{\Hmax}{\subscr{L}{max}}

\myclearpage

\section{{\color{black}Solving flow network problems on the $n$-torus}}
\label{sec:complete-solver}

In this section, we focus on verifying the existence and computing the
solutions of the flow network problem~\eqref{eq:flow_form} on the
$n$-torus.

One of the classic methods for solving nonlinear equations is the
Newton\textendash{}Raphson method.  This method comes with two main
disadvantages in the context of the flow network
problem~\eqref{eq:flow_form}. First, the Newton\textendash{}Raphson
iterations comes with no global convergence guarantee. Therefore, when the
iteration does not converge, one cannot infer the lack of solutions. The
second issue is the sensitivity of the Newton\textendash{}Raphson method
with respect to initial conditions. This is a critical issue since the flow
network problem~\eqref{eq:flow_form} can often have several solutions.
Consider for example the flow network problem~\eqref{eq:flow_form} on the
$5$-torus over the unit-weighted pentagon graph in
Figure~\ref{fig:polygonal-graphs} with flow function $h_e = \sin$, for
every $e\in \mathcal{E}$. For $\pactive=\vect{0}_5$, the flow network
problem has the phase-synchronous solution $\phi^*=\vect{0}_5$ with winding
number $0$ and the splay-state solution $\psi^* = (0, \tfrac{2\pi}{5},
\tfrac{4\pi}{5}, \tfrac{6\pi}{5}, \tfrac{8\pi}{5})^{\top}$ with winding
number $+1$. Consider the sequence $\{x^{(n)}\}_{n\in \mathbb{Z}_{\ge 0}}$
generated by the Newton\textendash{}Raphson algorithm:
\begin{align*}
  x^{(n+1)} = x^{(n)} - \left(B\mcA\diag(\cos(B^{\top}x^{(n)}))B^{\top}\right)^{\dagger} B \mcA\sin(B^{\top}x^{(n)}),
\end{align*}
with initial condition $x^{(0)} = (0, \tfrac{\pi}{3}, \tfrac{\pi}{2},
\tfrac{2\pi}{3} , 0)^{\top}$. While $\mathbf{w}_{\Sigma}(x^{(0)}) = 0$, the
Newton\textendash{}Raphson iterations starting form $x^{(0)}$ converges to
$\psi^*$ with winding number $+1$.  This example illustrates that the
Newton\textendash{}Raphson method is not consistent with the partition of
the $n$-torus introduced in Theorem~\ref{thm:partition}.

In the previous section, we showed that the winding partition of the
$n$-torus is useful for localizing the solutions of the flow network
problem~\eqref{eq:flow_form}. In the following Section~\ref{subsec:wbe}, we
provide a novel transcription for problem~\eqref{eq:flow_form} which
reveals the role of winding cells. The main advantage of this transcription
is to replace the phase angle $\theta\in \torus^n$ (i.e., a continuous
variable) with the winding vector {\color{black}$\mathbf{w}_{\Sigma}(\theta)\in
\mathbb{Z}^{m-n+1}$} (i.e., a discrete variable). Then, in
Section~\ref{subsec:fpf}, we introduce an appropriate operator and write
this novel transcription as a fixed-point problem. We show that this
fixed-point formulation is amenable to analysis using well-known
contraction techniques.

\subsection{Winding balance equation}
\label{subsec:wbe}
Using the polytopic characterization of winding
cells~\eqref{thm:linerization+correspondence}, one can identify the
counterclockwise angle difference vector {\color{black}$(B^{\top}\theta)$} in the
$\mathbf{u}$-winding cell with elements in the polytope
$P_{\mathbf{u}}$. As a result, we get a transcription of the flow network
problem~\eqref{eq:flow_form} which reveals the role of the winding
partition in this problem. Before we state this winding transcription,
{\color{black} for every $\gamma\in [0,\pi)$}, we introduce the extended flow function
$\map{h_{\gamma}}{\real^m}{\real^m}$, where, for every $e\in
\{1,\ldots,m\}$, its $e$th component is defined by:
 \begin{align*}
   \left(h_{\gamma}\right)_e(y) =
   \begin{cases}
     h_e(y_e), & |y_e|< \gamma,\\
     \frac{\partial h_e}{\partial y_e}(\gamma) (y_e-\gamma) +
     h_e(\gamma), & y_e\ge \gamma,\\
     \frac{\partial h_e}{\partial y_e}(\gamma) (y_e+\gamma) + h_e(-\gamma), & y_e\le -\gamma.
   \end{cases}
 \end{align*}
 Clearly, each function $(h_{\gamma})_e(y)$, $e\in\mathcal{E}$, depends
 only on the variable $y_e$ and it is monotone with respect to $y_e$. As a
 result, $h_{\gamma}$ is invertible on $\real^m$. Furthermore, define two diagonal matrices $\Hmin,\Hmax\in \real^{m\times m}$ by
\begin{align*}
  (\Hmin)_{ee} = \min_{y\in [-\gamma,\gamma]}\frac{\partial h_e(y)}{\partial y}\qquad\mbox{ and }\qquad
  (\Hmax)_{ee} = \max_{y\in [-\gamma,\gamma]}\frac{\partial h_e(y)}{\partial y},
\end{align*}
for all $e\in \{1,\ldots,m\}$.  Since each $h_e$, $e\in \{1,\ldots,m\}$, is
monotone we have $0<(\Hmin)_{ee}(\Hmax)_{ee}^{-1}\le 1$. This implies that
$\left\|I_m - \Hmin\Hmax^{-1}\right\|_{\infty}<1$. For every $\mathbf{u}\in \Img(\mathbf{w}_{\Sigma})$, the
 \emph{$\mathbf{u}$-winding balance equation} for the flow network problem~\eqref{eq:flow_form} is: {\color{black}
\begin{subequations}\label{eq:winding-form}
\begin{align}
  &B f = \pactive,\label{eq:winding-balance}\\
  &\prjcyc\Hmin\mcA (h_{\gamma}^{-1}(\mathcal{A}^{-1}f) - 2\pi C_{\Sigma}^{\dagger}\mathbf{u}) = \vect{0}_m, \label{eq:winding-equations}\\
  & |f_e| \le a_{ij}|h_e(\gamma)|,\qquad\mbox{ for }e=(i,j)\in \mathcal{E},\label{eq:winding-flow-constraint}
\end{align}
\end{subequations}}
where $\prjcyc$ is the $\Hmin$-weighted cycle projection matrix. Note
that the $\mathbf{u}$-winding balance equation has no variable on the
$n$-torus $\torus^n$. The following theorem illustrates the connection
between solutions of $\mathbf{u}$-winding balance
equation~\eqref{eq:winding-form} and solutions of the
flow network problem~\eqref{eq:flow_form}.

\begin{theorem}[Equivalence of $\mathbf{u}$-winding balance equation and
  flow network problem]\label{thm:equiv-kur-flow} Consider the flow network
  problem~\eqref{eq:flow_form} for $(G,\{h_e\}_{e\in
    \mathcal{E}},\pactive,\gamma)$ and suppose that each flow function
  $h_e$, $e\in \mathcal{E}$, is monotone on 
  $[-\gamma,\gamma]$. Let $\Sigma$ be a cycle basis for $G$. Then, for $\mathbf{u}\in\Img(\mathbf{w}_{\Sigma})$
  and $f\in\real^m$, the following statements are equivalent:
  \begin{enumerate}
  \item\label{p1:kur_equi} there exists a unique $\theta\in
    \Omega^G_{\mathbf{u}}$, modulo rotations, such that $(f,\theta)$ is a
    solution for the flow network problem~\eqref{eq:flow_form}; and
  \item\label{p2:loop_flow_equi} $f$ is a solution for the
    $\mathbf{u}$-winding balance equation~\eqref{eq:winding-form}.
    \end{enumerate}
 \end{theorem}

\subsection{Fixed-point formulation and flow network solver}
\label{subsec:fpf}
The winding transcription~\eqref{eq:winding-form} simplifies the analysis
of flow network problem~\eqref{eq:flow_form} on the $n$-torus by reducing
the continuous variable $\theta\in \torus^n$ to the finite discrete
variable $\mathbf{w}_{\Sigma}(\theta)$. We now provide an equivalent
fixed-point formulation for the winding
transcription~\eqref{eq:winding-form} which can be analyzed using
contraction theory. We first define the space of $\pactive$-balanced flows
by
\begin{align*}
 \Fsd = \setdef{f\in \real^m}{Bf = \pactive}.
\end{align*}
To write the $\mathbf{u}$-winding balance equation~\eqref{eq:winding-form}
as a fixed-point problem, we define the {\color{black}\emph{$\mathbf{u}$-winding
  fixed-point}} map $\map{T_{\mathbf{u}}}{\Fsd}{\Fsd}$ by:
\begin{align*}
T_{\mathbf{u}}(f) = f - \prjcyc\Hmin\mcA \left( h_{\gamma}^{-1}(\mathcal{A}^{-1}f) - 2\pi C_{\Sigma}^{\dagger}\mathbf{u}\right),
\end{align*}
where $\prjcyc$ is the $\Hmin$-weighted cycle projection matrix. Using
the {\color{black}$\mathbf{u}$-winding fixed-point map $T_{\mathbf{u}}$}, we can write the $\mathbf{u}$-winding
balance equation~\eqref{eq:winding-form} as the following fixed-point problem:{\color{black}
\begin{subequations}\label{eq:fixed-point-form}
\begin{align}
    &T_{\mathbf{u}}(f) = f, \\
    &|f_e| \le a_{ij}|h_e(\gamma)|,\qquad\mbox{ for }e=(i,j)\in \mathcal{E}.
\end{align}
\end{subequations}}
{\color{black}The fixed-point problem~\eqref{eq:fixed-point-form} suggests
  the definition of the \emph{projection iteration}}:
\begin{align}
  &f^{(k+1)} = T_{\mathbf{u}}(f^{(k)}), \label{eq:seq_converge}\\
  &f^{(0)} \in \Fsd. \label{eq:seq_initial}
\end{align}
for finding solutions of the flow network problem~\eqref{eq:flow_form}. The
following theorem study the connection between solvability of the flow
network problem~\eqref{eq:flow_form} on the $n$-torus and the convergence
of the projection iteration.

\begin{theorem}[Solvability of flow network problem on the 
  $n$-torus]\label{thm:convergence_ifexists} Consider the flow network
  problem~\eqref{eq:flow_form} for $(G,\{h_e\}_{e\in
    \mathcal{E}},\pactive,\gamma)$ and suppose that each flow function
  $h_e$, $e\in \mathcal{E}$, is monotone on the interval
  $[-\gamma,\gamma]$. {\color{black}Let $\Sigma$ be a cycle basis for $G$,
    $\mathbf{u}\in\mathbb{Z}^{m-n+1}$, $f^{(0)}\in \Fsd$, and
    $\{f^{(k)}\}_{k\in \mathbb{Z}_{\ge 0}}$ be the sequence generated by
    the projection iteration~\eqref{eq:seq_converge} starting from
    $f^{(0)}$. Then, the following statements hold:
  \begin{enumerate}
  \item\label{p1:converge} there exists a unique $f_{\mathbf{u}}^*\in \Fsd$
    such that $\{f^{(k)}\}_{k\in \mathbb{Z}_{\ge 0}}$ converges to
    $f_{\mathbf{u}}^*$;
  \item\label{p2:speed} for every $k\in \mathbb{Z}_{\ge 0}$, we have 
    \begin{align*}
      \left\|f^{(k+1)}-f^{(k)}\right\|_{\Hmin\mathcal{A}}
      \;\le\; \left\|I_m - \Hmin\Hmax^{-1}\right\|^k_{\infty}\left\|T_{\mathbf{u}}(f^{(0)}) - f^{(0)}\right\|_{\Hmin\mathcal{A}}.
    \end{align*}
  \end{enumerate}}
  Moreover, the following statements are equivalent:
  \begin{enumerate}\setcounter{enumi}{2}
  \item\label{p1:check}  {\color{black}$\left|(f^*_{\mathbf{u}})_e\right|\le a_{ij}|h_e(\gamma)|$}, for every $e=(i,j)\in \mathcal{E}$;
  \item\label{p2:exists_f} $f_{\mathbf{u}}^*$ is the unique solution to the
    $\mathbf{u}$-winding balance equation~\eqref{eq:winding-form};
  \item\label{p3:exists_theta} for $\theta^* =
    L^{\dagger}B\mathcal{A}(h^{-1}_{\gamma}(\mathcal{A}^{-1}f^*_{\mathbf{u}})-2\pi
    C^{\dagger}_{\Sigma} \mathbf{u})$, the pair
    $(f_{\mathbf{u}}^*,\theta^*)$ is the unique solution for the flow network
    problem~\eqref{eq:flow_form} with $\theta^*\in \Omega^G_{\mathbf{u}}$.
  \end{enumerate}
\end{theorem}

This theorem establishes that the projection
iteration~\eqref{eq:seq_converge} correctly computes the solution of the
$\mathbf{u}$-winding balance equation~\eqref{eq:winding-form}, if one
exists. Moreover, the iteration does so with a convergence rate equal to
$\|I_{m}-\Hmin \Hmax^{-1}\|_{\infty}$ and, remarkably, this rate is
independent of the network size.

We are now finally ready to present an algorithm that summarizes numerous
previous results.  The procedure proposed in
Algorithm~\ref{alg:active_solver} is guaranteed to compute all the
solutions of the flow network problem~\eqref{eq:flow_form} on the
$n$-torus, as we summarize in the next corollary.
  
  \smallskip
  
  \begin{algorithm}
    \caption{\textbf{Flow Network Solver}}
    \label{alg:active_solver}
    \begin{algorithmic}[1]
      \REQUIRE{a cyclic connected weighted undirected graph $G$, a
        supply/demand vector $\pactive$, an angle $\gamma\in [0,\pi)$,
      and a tolerance $\rho>0$ for accuracy of the flows solutions. }
      \ENSURE{all flows for~\eqref{eq:flow_form}}

      \smallskip
      
    \STATE compute a basis $\{\sigma_1,\ldots,\sigma_{m-n+1}\}$ for the
    cycle space of $G$
    
    \FOR{each candidate winding vector $\mathbf{u} =
      (w_1,\ldots,w_{m-n+1})^{\top}$ with $\ds w_ i \le
      \left\lfloor{\frac{\gamma n_{\sigma_i}}{2\pi}}\right\rfloor $}

    
    \STATE set $k\gets0$ and $f^{(0)} \gets \vect{0}_m$

    \REPEAT

    \STATE $f^{(k+1)} \gets T_{\mathbf{u}}(f^{(k)})$

    \STATE $k \gets k+1$
    
    \UNTIL $\left\|f^{(k)}-f^{(k-1)}\right\|_{\infty} < \rho$
    \IF{$\left| f^{(k)}_e \right| \leq
      a_{ij}|h_e(\gamma)|$, for every $e=(i,j)\in \mathcal{E}$} 
    \RETURN $(f^{(k)},\theta^*=L^{\dagger}B\mathcal{A}(h^{-1}_{\gamma}(\mathcal{A}^{-1}f^{(k)})-2\pi
    C^{\dagger}_{\Sigma} \mathbf{u}))$, as the unique solution
    of~\eqref{eq:flow_form} with  $\theta^*\in\Omega^{G}_{\mathbf{u}}$.
    \ELSE
    \RETURN there exists no solution $(f,\theta)$ for~\eqref{eq:flow_form} with $\theta\in\Omega^{G}_{\mathbf{u}}$. 
    \ENDIF
    \ENDFOR
    \end{algorithmic}
  \end{algorithm}
  \smallskip

\begin{corollary}[Complete solver for flow network problem]
  Consider the flow network problem~\eqref{eq:flow_form} for
  $(G,\{h_e\}_{e\in\mathcal{E}},\pactive,\gamma)$ and suppose that
  each flow function $h_e$, $e\in \mathcal{E}$, is monotone on the interval
  $[-\gamma,\gamma]$. Then the~\textbf{Flow Network
    Solver} in Algorithm~\ref{alg:active_solver} finds all solutions
  $(f,\theta)$ for the flow network problem~\eqref{eq:flow_form}.
\end{corollary}

\subsection{Computational complexity of the flow network solver}\label{subsec:compcomp}

{\color{black} The Flow Network Solver~\ref{alg:active_solver} consists of
  two main components: i) an algorithm for computing a cycle basis $\Sigma$
  (step \algostep{1}), and ii) the computation of the projection iteration
  for every feasible winding vector (for-loop at steps
  \algostep{2}-\algostep{11}). Interestingly, the computational complexity
  of these two components are interconnected. Choosing a computationally
  light algorithm for computing the cycle basis (step \algostep{1}) might
  give rise to a cycle basis with long cycles, a large number of feasible
  winding vectors, and, in turn, to a large number of executions of the
  projection iteration (for-loop at step \algostep{2}). On the other hand,
  computing the cycle basis that minimizes the cardinality of feasible
  winding vectors can be computationally heavy. Characterizing the right
  trade-off is out of the scope of this paper and we leave it to the
  future research. In this section, we focus on the minimum cycle basis
  algorithms for step \algostep{1} (see Remark~\ref{rem:cyclebasis}) and
  provide a detailed study of the run time of each step of the Flow Network
  Solver~\ref{alg:active_solver}. We assume that each mathematical
  operation (addition, subtraction, multiplication, and division) is a
  single floating-point operation (flop).

  \begin{theorem}[Computational complexity of the Flow Network Solver]
    \label{thm:compcomp}
    Consider the flow network problem
    $(G,\{h_e\}_{e\in\mathcal{E}},\pactive,\gamma)$ with monotone flow
    functions $h_e$, $e\in \mathcal{E}$, on the interval
    $[-\gamma,\gamma]$. Let $\Sigma = \{\sigma_1,\ldots,\sigma_{m-n+1}\}$
    be a cycle basis for $G$. The Flow Network
    Solver~\ref{alg:active_solver} has the following properties:
    \begin{enumerate}
    \item\label{p1:numberiter} the number of times the for-loop at step
      \algostep{2} is invoked is
      $\prod_{i=1}^{m-n+1}\left\lfloor{\frac{\gamma
          n_{\sigma_i}}{2\pi}}\right\rfloor$;
    \item\label{p2:compeachiter} the run time of (one execution of) steps
      \algostep{3}-\algostep{11} (i.e., the projection
      iteration~\eqref{eq:seq_converge} with its stopping condition) is
      $\mathcal{O}(\log(\rho^{-1}) n^3)$; and
    \item\label{p3:compiter} in summary, the run time of the for-loop at
      steps \algostep{2}-\algostep{11} is
      \begin{align*}
        \mathcal{O} \left( \log(\rho^{-1}) n^3
        \left(\tfrac{\gamma}{2\pi}\right)^{m-n+1} n_{\sigma_1}\ldots
        n_{\sigma_{m-n+1}}\right).
         \end{align*}
       \end{enumerate}
    Moreover, adopting the modified Horton algorithm in~\cite{KM-DM:09}
    (see Remark~\ref{rem:cyclebasis}) to compute the minimum cycle basis,
    \begin{enumerate}\setcounter{enumi}{3}
    \item\label{p4:comptotal} the run time of the Flow Network
      Solver~\ref{alg:active_solver} is
      \begin{align*}
        \mathcal{O}\left(m^2n/\log(n) + n^2m +\log(\rho^{-1}) n^3
        \left(\tfrac{3\gamma (n-1)(n-2)}{2\pi(m-n+1)}\right)^{m-n+1}\right).
      \end{align*}
    \end{enumerate}
  \end{theorem}
  Several remarks are in order.

  \begin{remark}[Bounds on the run time of Flow Network Solver]
    \begin{enumerate}
    \item Theorem~\ref{thm:compcomp}\ref{p4:comptotal} only provides
      an asymptotic upper bound on the run time of Flow Network
      Solver~\ref{alg:active_solver}. In practice, for sparse graphs
      with few number of cycles, this asymptotic bound is conservative. 
     
    \item The run time for computing a minimal cycle basis
      $\Sigma = (\sigma_1,\ldots,\sigma_{m-n+1})$ in step \algostep{2}
      is polynomial in the number of nodes of the network (see
      Remark~\ref{rem:cyclebasis}).
      \item For a given
      winding vector, the run time of the projection
      iteration~\eqref{eq:seq_converge} is cubic in the number of nodes of the network.
      \item For Kuramoto coupled
      oscillators~\eqref{eq:kuramoto-assoc} with identical zero
      natural frequencies, \cite{RT:15} proves that the problem of
      finding non-zero stable equilibrium points is
      NP-hard. Moreover, \cite{RT:15} shows that the problem remains
      NP-hard if it is restricted to finding non-zero equilibrium points satisfying $|\theta_i-\theta_j|<\frac{\pi}{2}$, for every
      adjacent nodes $(i,j)\in \mathcal{E}$. Indeed,
      following a similar argument as in~\cite{RT:15}, one can show that the
      problem of finding all solutions of the flow network problem~\eqref{eq:flow_form} on
      the $n$-torus is NP-hard.
      
    \item For networks with specific topologies, the run time of Flow
      Network Solver~\ref{alg:active_solver} can be exponential. This is illustrated in
      Example~\ref{ex:expo}. The
      key idea is that, for appropriate network
      families, the number of feasible winding vectors
      $(w_1,\ldots,w_{m-n+1})^{\top}$ with
      $\ds w_ i \le \left\lfloor{{\gamma
            n_{\sigma_i}}/{2\pi}}\right\rfloor$ grows exponentially
      with the number of nodes in the network. 
\end{enumerate}
\end{remark}

  \begin{example}[Exponential run time for the Flow Network Solver]\label{ex:expo}
    Consider the flow network problem
    $(G,\{h_e\}_{e\in\mathcal{E}},\vect{0}_n,\gamma)$ with graph $G$ as
    given in Figure~\ref{fig:graph-loop}, flow function $h_e(\cdot) =
    \sin(\cdot)$, for every $e\in \mathcal{E}$, and $\frac{2\pi}{5} \le
    \gamma \le \frac{\pi}{2}$.

    \begin{figure}[ht]\centering
        \includegraphics[width=.75\linewidth]{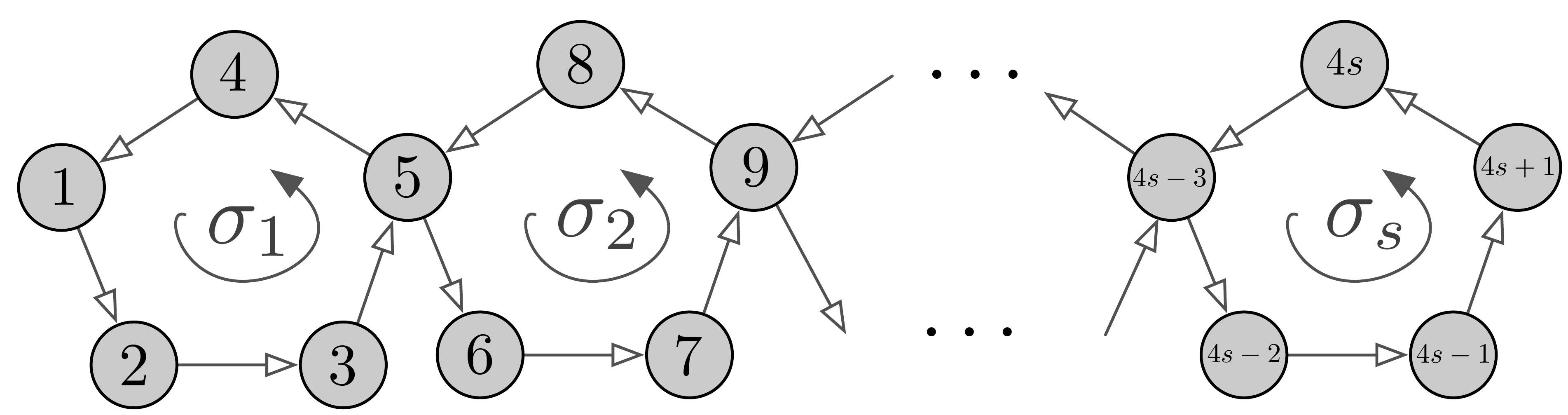}
        \caption{The graph $G$ with $n = 4s+1$ nodes, $m = 5s$ edges, and
          $s$ cycles.}\label{fig:graph-loop}
      \end{figure}
    We use the Flow Network Solver~\ref{alg:active_solver} to find all the
    solutions of the flow network problem~\eqref{eq:flow_form} on the
    $n$-torus. First note that $h_e(\cdot) = \sin(\cdot)$ is an odd
    monotone function on the interval $[-\gamma,\gamma]$. Thus, by
    Theorem~\ref{thm:compcomp}\ref{p2:compeachiter}, the run time of one
    execution of the steps~\algostep{3}-\algostep{11} is
    $\mathcal{O}(\log(\rho^{-1})n^3) $. The only cycle basis for the graph $G$
    is $\Sigma = (\sigma_1,\ldots,\sigma_{s})$. Therefore, the number of
    times the for-loop in steps~\algostep{2}-\algostep{11} is invoked is
    $\prod_{i=1}^{s} \left\lfloor{\frac{\gamma
        n_{\sigma_i}}{2\pi}}\right\rfloor = 3^s =
    \left(\sqrt[\leftroot{-2}\uproot{2}4]{3}\right)^{n-1}$. As a result,
    the run time of the Flow Network Solver~\ref{alg:active_solver} is
    $\mathcal{O}\left(\log(\rho^{-1})n^3\left(\sqrt[\leftroot{-2}\uproot{2}4]{3}\right)^{n-1}\right)$. Moreover,
    one can show that the above flow network problem has $3^s$
    solutions. Thus, the Flow Network Solver~\ref{alg:active_solver} does
    not find all the solutions of above flow network problem in
    polynomial time.
  \end{example}}

\subsection{Comparison with existing methods}

{\color{black} We conclude this section by comparing the Flow Network
  Solver~\ref{alg:active_solver} with two existing numerical algorithms for
  computing solutions of flow networks on the $n$-torus. The first algorithm is the 
  holomorphic embedding method (HELM), which is proposed in~\cite{AT:12} and further 
  developed in~\cite{SR-YF-DJT-MKS:16} to compute all solutions of power flow
  equations. The second algorithm is a numerical method based on parameter homotopy, which
  is proposed in~\cite{DM-NSD-FD:15} to compute all the equilibria of the Kuramoto
  model.

  Our approach differs from these two methods in several ways.
  First, these methods do not characterize the geometry of winding cells
  and do not explain the relationships between winding numbers, solutions,
  and loop flows. On the other hand, our Flow Network
  Solver~\ref{alg:active_solver} provides a geometric localization of
  the solutions based on their winding vectors. 
  Second, while \cite{AT:12} and \cite{DM-NSD-FD:15} do not provide a
  computational complexity analysis of HELM and parameter homotopy, the
  Flow Network Solver~\ref{alg:active_solver}, based upon (i) computing a
  cycle basis and (ii) performing a Banach contraction, lends itself to
  straightforward complexity analysis.
  Third, both HELM and parameter homotopy method use the sinusoidal form
  of the flow function to reformulate the problem as a polynomial
  system. Therefore, these methods are not extendable to flow networks on
  the $n$-torus with arbitrary flow functions.
  Finally, as presented in~\cite{AT:12} and~\cite{DM-NSD-FD:15}, HELM
  and the parameter homotopy method ignore the angle
  constraints~\eqref{eq:f-constraints} and aim to find all solutions
  to equations~\eqref{eq:f-KCL} and~\eqref{eq:f-physics} by
  reformulating them as a polynomial system. Thus, for flow networks
  where the number of solutions of the polynomial system is much larger than the number of solutions
  of the flow network problem~\eqref{eq:flow_form}, these methods are
  not computationally efficient.}

\myclearpage

\section{Numerical experiments and applications}
\label{sec:numerical-experiments}

In this section, we numerically study the solutions of the active power
flow equations~\eqref{eq:active-power-flow} with the thermal
constraint~\eqref{eq:thermal_constraints} as a flow network problem on the
$n$-torus.

\subsection{Loop flows in a simple cycle}\label{sec:ex-ring-power}

In this part, we consider two simple networks with the same underlying
graph $G$, the same edge weight matrix $\mathcal{A}$, and different balanced power supply/demand vectors $\wpactive$ as shown in
Figure~\eqref{fig:inject}. For each of these networks, we scale the
power transmission by a scalar $P\in \real_{\ge 0}$, that is, we
consider the balanced power supply/demand vectors $\pactive =
P\wpactive$. 

\begin{figure}[!htb]\centering
  \includegraphics[width=0.8\linewidth]{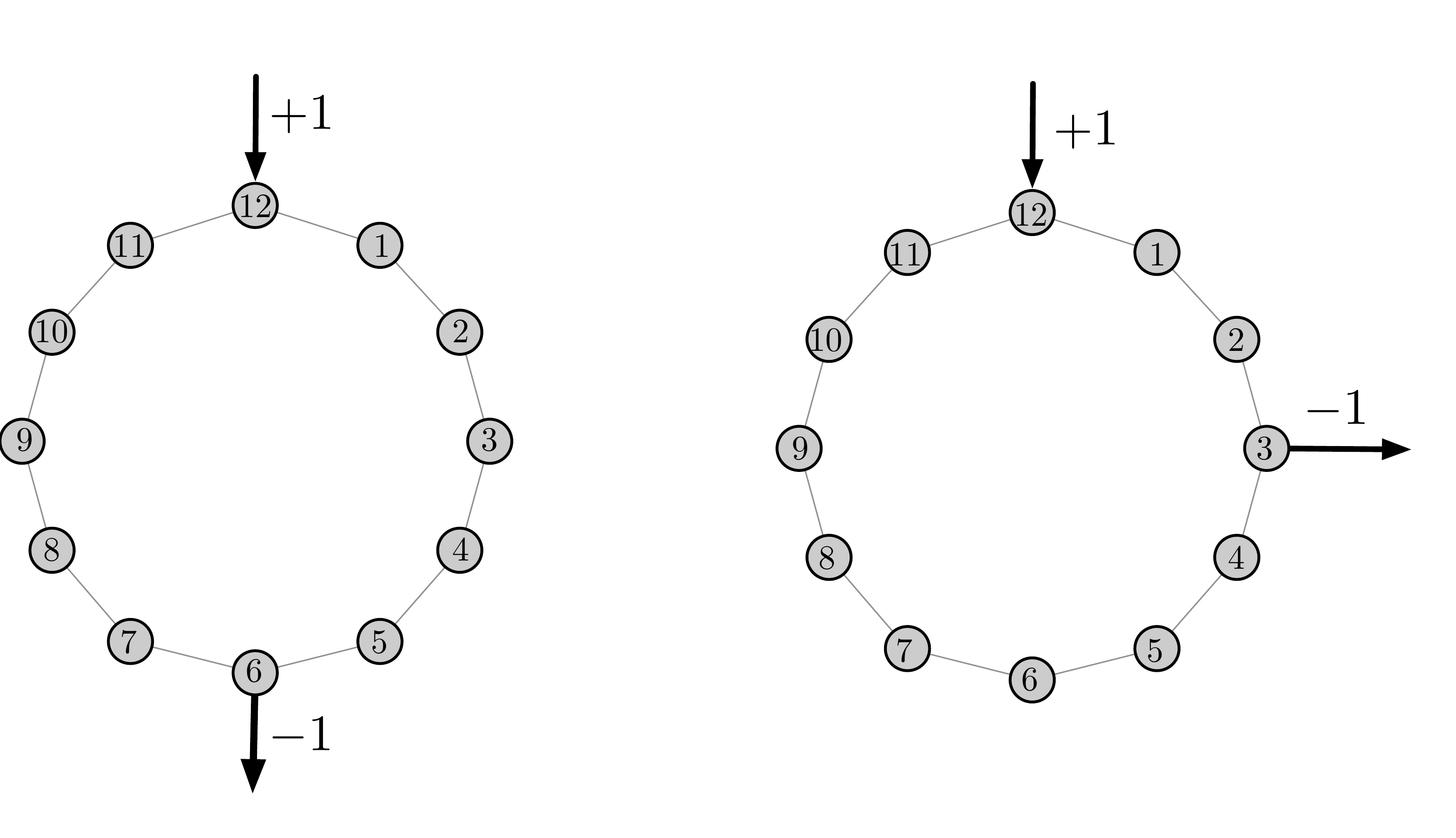}
  \caption{Example networks. Left image: $12$-node ring graph with weight
    matrix $\mcA=I_{12}$ and symmetric power profile with $\wpactive=
    [\vectorzeros[5]^\top,-1,\vectorzeros[5]^\top,1]^{\top}$. Right image:
    $12$-node ring graph with weight matrix $\mcA=I_{12}$ and asymmetric power
    profile with $\wpactive=
    [\vectorzeros[2]^\top,-1,\vectorzeros[8]^\top,1]^{\top}$.}
  \label{fig:inject}
\end{figure}

Using Theorem~\ref{cor:graphs-short-cycles}, the active power flow
equations~\eqref{eq:active-power-flow} with thermal
constraints~\eqref{eq:thermal_constraints} can only have solutions
$(f,\theta)$ with $|w_{\sigma}(\theta)|\le 3$. For every winding number
$u\in \{0,\pm1,\pm2,\pm3\}$, we define the \emph{power transmission capacity (PTC)} of
the network at winding number $u$ as the maximum $P$ for which the active
power flow equations~\eqref{eq:active-power-flow} with thermal
constraints~\eqref{eq:thermal_constraints} have a solutions with winding
number $u$. Moreover, for every winding number $u\in \{0,\pm1,\pm2,\pm3\}$,
we define the maximum network congestion at winding number $u$ by
\begin{align*}
  \max_{e=(i,j)\in \mathcal{E}} |a^{-1}_{ij}f_{e}|
\end{align*}
where $(f,\theta)\in\real^m\times \torus^n$ is the solution for the active power flow
equations~\eqref{eq:active-power-flow} with $w_{\sigma}(\theta) =u$.

We study the effect of scaling the power transmission on the solutions
of the active power flow equations~\eqref{eq:active-power-flow} with
thermal constraints~\eqref{eq:thermal_constraints}. We start from $P=0$ and increase $P$ by increment of $5\times
10^{-6}$ until the given solution of the active power flow ceases to exist. For each $P$, we use the Algorithm~\ref{alg:active_solver}
with tolerance $10^{-6}$ to compute the solutions of the active power
flow equations~\eqref{eq:active-power-flow} with thermal constraints~\eqref{eq:thermal_constraints}. The results of this simulation are shown in Figure~\ref{fig:pc-ring-symmetry}.  

\begin{figure}[!htb]\centering
  \includegraphics[width=0.7\linewidth]{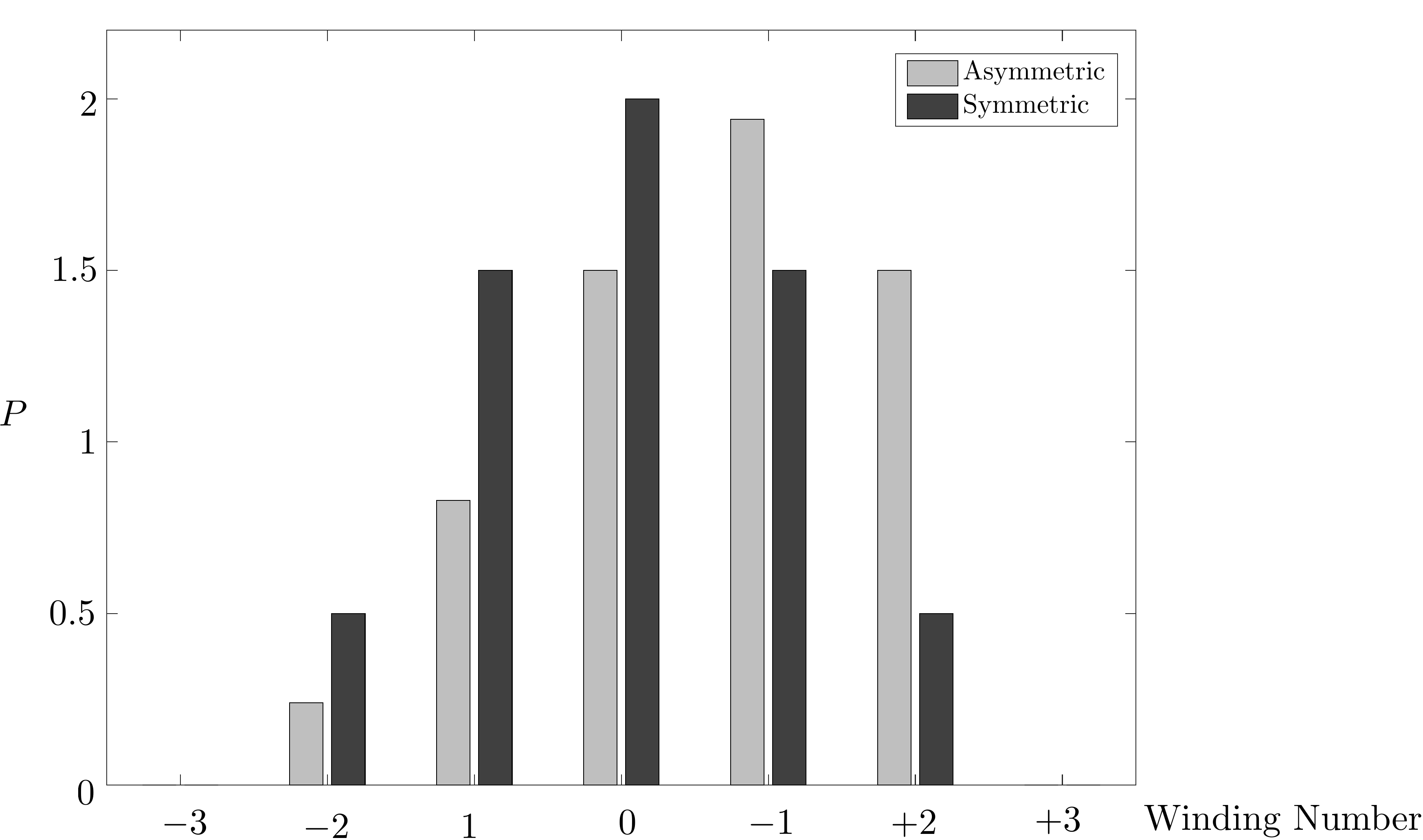}
  \caption{Solutions of active power flow equation with different winding
    numbers for symmetric and asymmetric power profiles shown in
    Figure~\ref{fig:inject}. For the symmetric power profile (left image in
    Figure~\ref{fig:inject}) the largest power transmission capacity is for winding
    number $0$. However, for the asymmetric power profile (middle image in
    Figure~\ref{fig:inject}) the largest power transmission capacity is for winding
    number~$-1$.}
	\label{fig:pc-ring-symmetry}
\end{figure}
\paragraph{Summary evaluation}
The winding number $0$ does not necessarily carry the maximum PTC of a
network. As illustrated in Figure~\ref{fig:pc-ring-symmetry}, the winding
number $-1$ carries the largest network PTC for the asymmetric supply/demand
vector $\pactive=P[\vectorzeros[2],-1,\vectorzeros[8],+1]^{\top}$. For the
asymmetric case, more power can flow along the left path when the winding
number is $-1$ than when the winding number is $0$, due to the thermal
constraint~\eqref{eq:thermal_constraints}.  As a
result, when the winding number is $-1$, the flow along the left path of
the ring is able to better relieve the capacity for the right
path.\oprocend

\subsection{Loop flows in IEEE RTS 24 testcase}

The ongoing shift form fossil-fueled power generation to renewable energy resources is leading to large changes
in the supply/demand structure of the power grids. One of the related
issues is the change to undesirable operating points, where large
powers are flowing over the network. In this part, we study the
existence of undesirable operating points for the modified IEEE RTS 24 testcase. The IEEE
RTS 24 is a portion of the larger IEEE RTS 96 testcase which is designed to
study reliability of the power
networks~\cite{CG-PW-PA-RA-MB-RB-QC-CF-SH-SK-WL-RM-DP-NR-DR-AS-MS-CS:99}. This
testcase can be described by a connected, undirected graph $G$ with
$24$ buses and $34$ branches. The nodal admittance matrix is denoted
by $Y\in\complex^{24\times 24}$. The set of nodes are
partitioned into load buses $\mathcal{V}_1$ and generator buses
$\mathcal{V}_2$. The power demand (resp. power supply) at node
$i\in\mathcal{V}_1$ (resp. $\mathcal{V}_2$) is denoted by the $i$th element
of $\pactive$. The nominal parameters for the test cases can be found
in~\cite{CG-PW-PA-RA-MB-RB-QC-CF-SH-SK-WL-RM-DP-NR-DR-AS-MS-CS:99} and is
shown in the second column of the Table~\eqref{tab:nom-vs-new-power-profile}. We first modify the
branches in IEEE RTS 24 to be lossless without shunt
admittances. Since our theoretical results assume that all nodes are $PV$ nodes, MATPOWER's solver
is used to compute the voltage magnitudes at every node~\cite{RDZ-CEMS-RJT:11}. Then the edge weights of $G$ are set to
$a_{ij}=a_{ji}=V_{i}V_{j}\operatorname{Im}(Y_{ij})>0$. A cycle basis for $G$ is $\Sigma=\{\sigma_1,\ldots,\sigma_{11}\}$ where, for every
$i\in\{1,\ldots,11\}$, the cycle $\sigma_i$ is given in Table~\ref{tab:cycle_basis}.
\begin{center}
\begin{table}[htb] \centering
\begin{tabular}{|c|c|}
\hline
\multicolumn{2}{|c|}{Cycle basis for IEEE RTS 24}\\
\hline
\rowcolor{Gray}
$\sigma_1 =(2,1,3,9,4,2)$ & $\sigma_2=(5,1,3,9,8,10,5)$\\

$\sigma_3 =(10,6,2,4,9,8,10)$ & $\sigma_4 =(11,10,8,9,11)$\\
\rowcolor{Gray}
$\sigma_5 =(12,10,8,9,12)$ & $\sigma_6 =(13,11,9,12,23,13)$\\

$\sigma_7 =(13,12,23,13)$ &  $\sigma_8 =(16,15,24,3,9,11,14,16)$\\
\rowcolor{Gray}
$\sigma_9 =(21,15,24,3,9,11,14,16,17,22,21)$ & $\sigma_{10} =(21,18,17,22,21)$\\

$\sigma_{11}=(23,20,19,16,14,11,9,12,23)$ & \\
\hline
\end{tabular}
\caption{A cycle basis $\Sigma$ for the IEEE RTS 24 testcase}\label{tab:cycle_basis} 
\end{table}\vspace{-0.5cm}
\end{center}

We also modify the nominal power supply/demand of the IEEE RTS 24
testcase. The goal is to increase the penetration of renewable energy units
and remove some of the synchronous generators. Our modification in the
power supply/demand of the IEEE RTS 24 testcase is illustrated in
Figure~\ref{fig:IEEERTS24}.  The modified supply/demand vector is denoted
by $\pactive^{\mathrm{mod}}$ and is given in the third column of
Table~\ref{tab:nom-vs-new-power-profile}.

Using the Algorithm~\ref{alg:active_solver}, we study the active
power flow equation~\eqref{eq:active-power-flow} and the thermal
constraints~\eqref{eq:thermal_constraints} with maximum power angle
$\gamma=1.5\;\; \mathrm{rad}$. First, we observe that, for the nominal
power supply/demand vector of IEEE RTS 24, there exists no solution
for this problem with a nonzero winding vector. However, for the modified supply/demand vector
$\pactive^{\mathrm{mod}}$, there exists exactly one solution associated to the winding vector $\mathbf{u}=\vect{0}_{11}$
and one solution associated to the winding vector $\mathbf{u} =
[\vectorzeros[10]^{\top},-1]$. We then examine the computational
efficiency of the projection iteration~\eqref{eq:seq_converge} for checking the
existence/finding solutions of the active power flow
equations~\eqref{eq:active-power-flow}. We assume that the cycle basis $\Sigma$
illustrated in Table~\ref{tab:cycle_basis} is given. For every winding
vector $\mathbf{u}\in \Img(\mathbf{w}_{\Sigma})$, we focus on the
$\mathbf{u}$-winding balance equations~\eqref{eq:winding-form}. Then $\subscr{t}{fsolve}$ is the
computational time for solving~\eqref{eq:winding-form} using MATLAB's fsolve and
$\subscr{t}{sequence}$ is the computational time for solving~\eqref{eq:winding-form}
using the projection iteration~\eqref{eq:seq_converge}. For
$\mathbf{u}=\vect{0}_{11}$ and $\mathbf{u} =
[\vectorzeros[10]^{\top},-1]$, the computational times $\subscr{t}{sequence}$ and $\subscr{t}{fsolve}$ are compared in Table~\eqref{tab:computational-complexity}.\footnote{The computer
  specifications for these simulations are as follows: Processor Intel Core i5 @ 1.6 GHZ CPU and 4 GB RAM.} 
\begin{center}
	\begin{table}[htb] \centering
		\begin{tabular}{|c|c|c|c|c|c|c|}
			\hline
			\multicolumn{7}{|c|}{IEEE RTS 24} \\ \hline
			Node &$\pactive^{\mathrm{nom}}\ (MW)$ &
                                                          $\pactive^{\mathrm{mod}}\
                                                              (MW)$
               & &  Node &$\pactive^{\mathrm{nom}}\ (MW)$ &
                                                          $\pactive^{\mathrm{mod}}
                                                          \ (MW)$\\ \hline
			$\begin{matrix}1 \\ 2 \\ 3 \\ 4 \\ 5 \\ 6 \\ 7 \\ 8 \\ 9 \\ 10 \\ 11 \\ 12 \end{matrix}$ 
				&$\begin{matrix} 64.00 \\ 75.00 \\ -180.00 \\ -74.00 \\ -71.00 \\ -136.00 \\ 115.00 \\ -171.00 \\ -175.00 \\ -195.00 \\ 00.00 \\ 00.00 \end{matrix}$
				&$\begin{matrix}8.40 \\ 9.27 \\
                                  -268.48 \\ -99.14 \\ -79.96 \\
                                  -68.63 \\ 63.65 \\ -142.41 \\
                                  -245.21 \\ 95.83 \\ 100.00 \\
                                  00.00 \end{matrix}$ & &
$\begin{matrix}13 \\ 14 \\ 15 \\ 16 \\ 17 \\ 18 \\ 19 \\ 20 \\ 21 \\ 22 \\ 23 \\
                  24 \end{matrix}$ & $\begin{matrix}-129.00 \\ -194.00
                  \\ -102.00 \\ 55.00 \\ 00.00 \\ 67.00 \\ -181.00
                  \\ -128.00 \\ 400.00 \\ 300.00 \\ 660.00 \\
                  00.00 \end{matrix}$ & $\begin{matrix}-193.50 \\
                  -143.39 \\ -153.00 \\ 00.00 \\ 00.00 \\ 00.00 \\
                  26.57 \\ 100.00 \\ 00.00 \\ 00.00 \\ 990.00 \\
                  00.00 \end{matrix}$ \\ \hline
		\end{tabular} 
\caption{Nominal supply/demand vector $\pactive^\mathrm{nom}$ versus modified
  supply/demand vector $\pactive^{\mathrm{mod}}$.}
		\label{tab:nom-vs-new-power-profile}
	\end{table}
\end{center}
\begin{figure}[!htb]\centering
  \includegraphics[width=\linewidth]{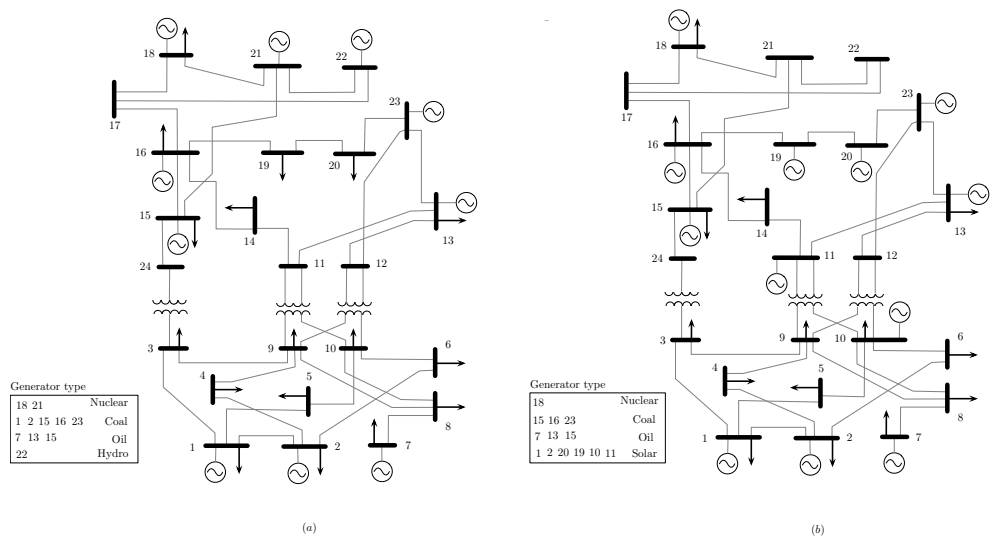}
  \caption{(a) shows the original IEEE RTS 24 testcase with generator types
    and (b) shows the modified IEEE RTS 24 testcase. The modifications
    in power supply/demand vector in Figure (b) are: i) shutting down
    the hydro generator at node 22, the thermal generators at node 1
    and 2, and the nuclear generator at node 21 ii) adding solar energy generations at
    nodes 1,2, 10, 11, 19, 20, 23.}
  \label{fig:IEEERTS24}
\end{figure}
\begin{center}
	\begin{table}[htb] 
		\centering
		\begin{tabular}{|c|c|}
			\hline
				 Winding vector, $\mathbf{u}$
                  &$\subscr{t}{sequence}/\subscr{t}{fsolve}$ \\ \hline
				\rowcolor{Gray}
				$[\vectorzeros[11]^{\top}]$ & 0.2463 \\
				$[\vectorzeros[10]^{\top},-1]$  & 0.0788\\ \hline
		\end{tabular}
\caption{Computation time for the
  projection iteration~\eqref{eq:seq_converge} for $\gamma=1.5$ is denoted by
  $\subscr{t}{sequence}$ and for MATLAB's \emph{fsolve} is denoted by $\subscr{t}{fsolve}$. The values given in the table are the ratio $\subscr{t}{sequence}/\subscr{t}{fsolve}$, averaged over $5$ trials. The computations use a tolerance of $10^{-6}$.}
		\label{tab:computational-complexity}
	\end{table}
\end{center}
\paragraph{Summary evaluation} 
We used the Algorithm~\ref{alg:active_solver} to check for the existence/find all
solutions of the active power flow
equations~\eqref{eq:active-power-flow} for the nominal and modified
IEEE RTS 24 testcase. Table~\ref{tab:computational-complexity} shows that, for the
nominal and modified IEEE RTS 24 testcase, the projection iteration~\eqref{eq:seq_converge}
not only converges to the loop flow solutions but also is faster than the
MATLAB's \textit{fsolve} for checking the existence/computing the
solutions of the active power flow equations. \oprocend

\myclearpage
\section{Conclusion}
In this paper, we have introduced two classes of network problems on the
$n$-torus---flow networks and elastic networks---and we developed a
rigorous framework to study the multiple solutions of these problems.  We
extended Kirchoff's voltage law to networks with phase-valued (instead of
real-valued) nodal variables, and we showed how this law induces a
partition of the $n$-torus into winding cells.  We demonstrated that these
winding cells localize solutions of the flow and elastic network problems,
since each cell contains at most one solution.  In order to compute the
solution in each winding cell (or determine that no solution exists), we
proposed the projection iteration, a novel contraction mapping.  Finally,
we presented several numerical experiments, which investigate the notion of
flow capacity and flow congestion in different test cases, and we verified
the accuracy and efficiency of our methods.

Much work remains on the connection of solutions of flow and elastic
networks with other phenomena in network systems.  For power systems
specifically, we have already exploited the winding partition to derive
sufficient conditions for transient stability in AC grids
\cite{KDS-SJ-FB:18f}.  But this work assumes lossless AC grids with
constant voltage magnitudes, and it is important to investigate power flows
in more realistic scenarios.  See~\cite{JWSP:17a,JCB-TC-LD:18} for recent
work in this direction.  It would also be interesting to compare the
performance of our Flow Network Solver~\ref{alg:active_solver} to
state-of-the-art power flow solvers in the literature.  Other numerical
directions include applying our flow network solver to study collective
motion in engineering networks and the performance of associative memory
networks.  In a more theoretical direction, it would be valuable to apply
the framework and analysis of this paper to study more general coupled
oscillator networks, such as FitzHugh\textendash{}Nagumo and
Hodgkin\textendash{}Huxley
models~\cite{MD-JG-BK-CK-HO-MW:12,GSM-NC:01,AF-GD-RS:14}.  In particular,
we envision that the winding partition may be useful in developing analytic
conditions for synchronization in coupled oscillator networks, which is one
of the central problems in this area of
research~\cite{REM-SHS:05,NC-MWS:08,FD-MC-FB:11v-pnas}.


\myclearpage
\bibliographystyle{plainurl}
\bibliography{alias,Main,FB}

\myclearpage
\appendix

\section{Two useful lemmas}\label{app:proofs+lemmas}

{\color{black} }

\begin{lemma}[Properties of the cycle-edge matrix]\label{thm:shift_property}
  Given a cyclic connected undirected graph $G$ with cycle basis
  $\Sigma=\{\sigma_1,\ldots,\sigma_{m-n+1}\}$, the following statements
  hold:
\begin{enumerate}
\item\label{p1:rank} the matrix ${C}_{\Sigma}$ is of rank
  $m-n+1$, 
\item\label{p2:ker} $\Ker({C}_{\Sigma}) =\Img(B^{\top})$,
\item\label{p3:integer-basis} the system of
  equations $C_\Sigma \mathbf{z} = \mathbf{v}$, for
  $\mathbf{v}\in \Img(\mathbf{w}_{\Sigma})$, has an integer solution.
\end{enumerate}
\end{lemma}

\begin{proof}
Regarding part~\ref{p1:rank}, note that rank of the matrix $C_\Sigma$ is equal
to the number of linearly independent rows in $C_{\Sigma}$. By definition, the set
$\{\sigma_1,\ldots,\sigma_{m-n+1}\}$ consists of linearly independent
vectors in $\real^m$ and thus rank of $C_{\Sigma}$ is equal to $m-n+1$.

Regarding part~\ref{p2:ker}, we first show that
$\Img(B^{\top})\subseteq \Ker(C_{\Sigma})$. Suppose that
$\mathbf{y}\in \Img(B^{\top})$. Then there exists $\mathbf{x}\in
\real^n$ such that $\mathbf{y} = B^{\top}\mathbf{x}$. Thus, for every
$i\in \{1,\ldots,m-n+1\}$, we get 
\begin{align*}
v_{\sigma_i}^{\top} \mathbf{y} = v_{\sigma_i}^{\top} B^{\top}\mathbf{x} =
  (B v_{\sigma_i})^{\top}\mathbf{x}= 0,
\end{align*}
where the last equality holds since $v_{\sigma}\in
\Ker(B)$, for every cycle $\sigma$. This means that
$C_{\Sigma}\mathbf{y}=0$ or equivalently $\mathbf{y}\in
\Ker(C_{\Sigma})$. Therefore, $\Img(B^{\top})\subseteq
\Ker(C_{\Sigma})$. In turn, using the result in part~\ref{p1:rank},
rank of the matrix $C_{\Sigma}$ is $m-n+1$. Thus, by the rank-nullity theorem, we have 
\begin{align*}
\mathrm{dim}(\Ker(C_{\Sigma})) = m - (m-n+1) = n-1 = \mathrm{dim}(\Img(B^{\top})).
\end{align*}
As a result, we have $\Img(B^{\top}) = \Ker(C_{\Sigma})$.

Regarding part~\ref{p3:integer-basis}, suppose that $\Sigma'$
is an integral cycle basis for $G$ (for a definition and existence of an integral cycle
basis see~\cite{TK-CL-KM-DM-RR-TU-KZ:09}). Since $\Sigma$ is a cycle
basis for $G$, there exists an invertible integer matrix $T \in
\real^{(m-n+1)\times (m-n+1)}$ such that 
\begin{align*}
v_{\sigma'_i} = \sum_{j=1}^{n} t_{ji} v_{\sigma_j}. 
\end{align*}
This implies that $C_{\Sigma'} = T C_{\Sigma}$. We first show that the
system of linear equations $C_{\Sigma'}\mathbf{z} = T\mathbf{u}$
has an integer solution. Note that rank of the matrix
$C_{\Sigma'}$ is $m-n+1$. Let $\mathcal{T}$ be a spanning tree in $G$ and without loss of generality,
assume that the first $m-n+1$ column of $C_{\Sigma'}$ are
associated with the edges that are not in the spanning tree $\mathcal{T}$. Then,
we define the matrix $\widehat{C}_{\Sigma'} \in \real^{(m-n+1)\times
  (m-n+1)}$ as follows:
\begin{align*}
\widehat{C}_{\Sigma'} = \begin{bmatrix}\left(C_{\Sigma'}\right)_{1} &
  \ldots & \left(C_{\Sigma'}\right)_{m-n+1}\end{bmatrix}.
\end{align*}
where $\left(C_{\Sigma'}\right)_{i}$ the $i$th column of the matrix
$C_{\Sigma'}$. Since $\Sigma'$ is an integral base, by~\cite[Theorem
3.4]{TK-CL-KM-DM-RR-TU-KZ:09}, we have $|\det(\widehat{C}_{\Sigma'})| =
1$. Since the matrix $\widehat{C}_{\Sigma'}$ has integer entries,
by~\cite[Theorem 1]{FL:96}, there exists matrices $U,V\in \integer^{(m-n+1)\times (m-n+1)}$ with $|\det(U)|
= |\det(V)| =1$ such that
\begin{align*}
  U\widehat{C}_{\Sigma'}V = B = \diag(b_1,\ldots,b_{m-n+1}),
\end{align*}
where $b_i$ are positive integers and $b_i | b_{i+1}$, for every
$i\in\{1,\ldots,m-n\}$. Note that $|\det(\widehat{C}_{\Sigma'})| = 1$
implies that $|\det(B)| = 1$ which in turn implies that $|b_i| =1$, for every $i\in
\{1,\ldots,m-n+1\}$. Therefore, by~\cite[Proposition 2]{FL:96}, for every $\mathbf{v}\in
\integer^{(m-n+1)}$, there exists an integer solution $\mathbf{x}\in
\integer^{(m-n+1)}$ such that $\widehat{C}_{\Sigma'}(\mathbf{x}) = T\mathbf{u}$. 
Then it is straightforward to see that the integer vector
$\mathbf{z}\in \integer^{m}$ defined by $\mathbf{z} = \begin{bmatrix} \mathbf{x}\\ \vect{0}_{n-1}\end{bmatrix}$
is a solution for the system of equations $C_{\Sigma'}\mathbf{z} =
T\mathbf{u}$. This means that $TC_{\Sigma}\mathbf{z} =
T\mathbf{u}$ and since $T$ is invertible, we have
$C_{\Sigma}\mathbf{z} = \mathbf{u}$. This completes the proof of Lemma~\ref{thm:shift_property}. 
\end{proof}

{\color{black}
  \begin{lemma}[Properties of $D$-weighted cycle
    projection]\label{thm:cycle-projection}
  Let $G$ be an undirected weighted connected graph with incidence matrix $B$
  and diagonal weight matrix $\mathcal{A}\in \real^{m\times m}$. Let
  $D\in \real^{m\times m}$ be a positive definite diagonal matrix, and
  $\mathcal{P}_{D}$ be the $D$-weighted cycle projection defined
  in~\eqref{eq:cycle-projection}. Then the following statements hold:
  \begin{enumerate}
  \item\label{p1:cyc-eigen} $\mathcal{P}_{D}$ is idempotent; 
  \item\label{p2:cyc-eigen} the eigenvalues of $\mathcal{P}_{D}$ are $0$
    and $1$ with algebraic (geometric) multiplicity $n-1$ and $m-n+1$,
    respectively;
    \item\label{p3:ImBT} $\Ker(\mathcal{P}_{D}D\mathcal{A}) = \Img(B^{\top})$. 
    \end{enumerate}
  \end{lemma}
  \begin{proof}
    Regarding part~\ref{p1:cyc-eigen}, define the graph $G'$ with the same
    node set, the same edge set as $G$, and with the diagonal weight matrix
    $D\mathcal{A}$. Clearly, $G$ is connected if and only if $G'$ is
    connected. Moreover, by~\cite[Theorem 4]{SJ-FB:16h}, $\mathcal{P}_{D}$
    is an oblique projection onto $\Ker(BD\mathcal{A})$ parallel to
    $\Img(B^{\top})$ and, therefore, is idempotent. Regarding
    part~\ref{p2:cyc-eigen}, the result follows from~\cite[Theorem
      4]{SJ-FB:16h}.  Regarding part~\ref{p3:ImBT}, suppose that $\alpha\in
    \Img(B^{\top})$. This means that there exists $\xi\in \real^{n}$ such
    that $B^{\top}\xi = \alpha$. As a result,
    \begin{align*}
      \mathcal{P}_{D} D\mathcal{A}\alpha & = D\mathcal{A}\alpha -
      D\mathcal{A} B^{\top} (B D\mathcal{A} B^{\top})^{\dagger}B
      D\mathcal{A} \alpha =  D\mathcal{A}\alpha - D\mathcal{A} B^{\top} (B D\mathcal{A} B^{\top})^{\dagger}B
      D\mathcal{A} B^{\top}\xi \\ & = D\mathcal{A}\alpha - D\mathcal{A}
      B^{\top} (I_n-\tfrac{1}{n}\vect{1}_n\vect{1}^{\top}_n)\xi,
    \end{align*}
    where, for the last equality, we used the equality $(B D\mathcal{A}
    B^{\top})^{\dagger}B D\mathcal{A} B^{\top} =
    I_n-\tfrac{1}{n}\vect{1}_n\vect{1}^{\top}_n$. This equality holds by
    the identity in \cite[Lemma~6.12(iii)]{FB:20} and the fact that the
    graph $G'$ is connected with the Laplacian matrix $B D \mathcal{A}
    B^{\top}$~\cite[Lemma~9.1]{FB:20}. Note that $B^{\top}\vect{1}_n =
    \vect{0}_m$~\cite[Section 9.1]{FB:20}. Thus, we get
      \begin{align*}
      \mathcal{P}_{D} D\mathcal{A}\alpha = D\mathcal{A}\alpha - D\mathcal{A}
      B^{\top} (I_n-\tfrac{1}{n}\vect{1}_n\vect{1}^{\top}_n)\xi =  D\mathcal{A}\alpha - D\mathcal{A}
      B^{\top}\xi = D\mathcal{A}\alpha - D\mathcal{A}\alpha = \vect{0}_m.
    \end{align*}
    Therefore, $\alpha \in \Ker(\mathcal{P}_{D}D\mathcal{A})$ and, in turn,
    $\Img(B^{\top})\subseteq \Ker(\mathcal{P}_{D}D\mathcal{A})$. Now
    suppose that $\alpha\in \Ker(\mathcal{P}_{D}D\mathcal{A})$. Therefore,
    \begin{align*}
      \vect{0}_m = \mathcal{P}_{D} D\mathcal{A}\alpha = D\mathcal{A}\alpha -
      D\mathcal{A} B^{\top} (B D\mathcal{A} B^{\top})^{\dagger}B
      D\mathcal{A} \alpha. 
    \end{align*}
    Note that $\mathcal{A}$ and $D$ are invertible. Thus, $\alpha =
    B^{\top} (B D\mathcal{A} B^{\top})^{\dagger}B D\mathcal{A}
    \alpha$. This implies that $\alpha \in \Img(B^{\top})$ and therefore
    $\Ker(\mathcal{P}_{D}D\mathcal{A})\subseteq \Img(B^{\top})$. This
    completes the proof of part~\ref{p3:ImBT}.
    \end{proof}}

\section{Basis-independent winding map}\label{app:basis-indep}

One can define a basis-independent winding map on $\torus^n_0$. Let $\Ker'(B)$ be the dual
  space of the cycle space $\Ker(B)$ and let
  $\Sigma'=\{v'_{\sigma_1},\ldots,v'_{\sigma_{m-n+1}}\}$ be the dual basis on
  $\Ker'(B)$, associated to the basis
  $\Sigma=\{v_{\sigma_1},\ldots,v_{\sigma_{m-n+1}}\}$ on $\Ker(B)$, that is,
  for every $i,j\in \{1.\ldots,m-n+1\}$, we have $v'_{\sigma_i}(v_{\sigma_j})=\delta_{ij}$.  Define the
  winding map $\map{\mathbf{w}}{\torus_0^n}{\Ker'(B)}$ by:
  \begin{align}\label{eq:wind_coor_ind}
    \mathbf{w}(\theta) = \tfrac{1}{2\pi}\sum_{i=1}^{m-n+1}
    (v_{\sigma_i}^{\top}(B^{\top}\theta)) v'_{\sigma_i}, 
  \end{align}
  {\color{black}where the vector $(B^{\top}\theta)$ is defined in
    equation~\eqref{def:BTtheta}.} The winding map $\mathbf{w}_\Sigma$ in
   Definition~\ref{def:wind}\ref{def:wind_covector} is the representation of the map
   $\mathbf{w}$ in equation~\eqref{eq:wind_coor_ind} in the basis
   $\Sigma'=\{v'_{\sigma_1},\ldots,v'_{\sigma_{m-n+1}}\}$ for
   $\Ker'(B)$.

\section{Proofs of results in Section~\ref{sec:problems}}
   
\subsection{Proof of Theorem~\ref{thm:equivalence}}

\begin{proof}
Before we proceed with the proof, we state a useful observation. Consider
the potential energy function
\begin{align*}
  \mcH(\theta) = \sum_{(i,j)\in \mathcal{E}} a_{ij}H_e(\theta_i-\theta_j)
  = \sum_{(i,j)\in \mathcal{E}} a_{ij}H_e(\subscr{d}{cc}(\theta_i,\theta_j))
\end{align*}
We fix $e=(i,j)\in\mathcal{E}$ and we show that $H_e$ is differentiable at
$\theta$. Note that $H_e$ is twice differentiable. Since, $\theta\in
\torus^n$ satisfies the angle constraint~\eqref{eq:f-constraints}, the
counterclockwise difference $\subscr{d}{cc}$ is a differentiable function
at $\theta\in\torus^n$. This implies that $H_e$ is differentiable at
$\theta$. Let $\Sigma$ be a cycle basis for graph $G$. Since $\theta\in \torus^n$ satisfies the angle
constraint~\eqref{eq:f-constraints}, by Theorem~\ref{thm:partition}\ref{p2:union}, there exists
$\mathbf{u}\in \Img(\mathbf{w}_{\Sigma})$ such that
$\theta\in\Omega^G_{\mathbf{u}}$. Now, using Theorem~\ref{thm:linerization+correspondence}, there
exists a {\color{black} bijection} between $[\theta]\in
\Omega^G_{\mathbf{u}}/\torus^1$ and $\mathbf{x}\in P_{\mathbf{u}}$ such that {\color{black}
$(B^{\top}\theta) = B^{\top}\mathbf{x} + 2\pi
C^{\dagger}_{\Sigma}\mathbf{u}$, where the vector $(B^{\top}\theta)$ is defined by 
  equation~\eqref{def:BTtheta}}. This implies
that, for $e=(i,j)\in\mathcal{E}$, 
\begin{align}\label{eq:semiimportant}
   \frac{\partial}{\partial\theta_i}\subscr{d}{cc}(\theta_i,\theta_j)
  =  \frac{\partial}{\partial\theta_i}(B^{\top}\theta)_{e} =
  \frac{\partial}{\partial x_i}(B^{\top}\mathbf{x} + 2\pi
  C^{\dagger}_{\Sigma}\mathbf{u})_{e}  = +1.
  \end{align}

Therefore, for every $i\in \{1,\ldots,n\}$,
\begin{align}\label{eq:important}
  \frac{\partial \mcH}{\partial \theta_i} =\sum_{(k,j)\in\mathcal{E}} a_{kj}\frac{\partial}{\partial
  \theta_i}H_e(\subscr{d}{cc}(\theta_k,\theta_j)) &= \sum_{(k,j)\in\mathcal{E}} a_{kj}
  \frac{d
    H_e}{d\alpha}\left(\subscr{d}{cc}(\theta_k,\theta_j)\right)\frac{\partial}{\partial\theta_i}\subscr{d}{cc}(\theta_k,\theta_j)
  \nonumber\\ &=
  \sum_{j=1}^{n}
  a_{ij}
  h_e(\subscr{d}{cc}(\theta_i,\theta_j))
  = \sum_{j=1}^{n}
  a_{ij}
  h_e(\theta_i-\theta_j).
\end{align}
where for the last equality, we used the fact that
$\frac{dH_e}{d\alpha}(\alpha) = h_e(\alpha)$, for every $e\in
\mathcal{E}$. Now we go back to the proof of the theorem.

Regarding~\ref{p1:flownet} $\implies$ \ref{p2:elasticnet}, if $(f,\theta)$
is a solution to the flow network problem~\eqref{eq:elastic_form}, then
using~\eqref{eq:f-KCL}, for every $i\in \{1,\ldots,n\}$,
\begin{align*}
  (\pactive)_i =\sum_{j=1}^{n}f_{(i,j)}= \sum_{j=1}^{n} a_{ij}h_e(\theta_i-\theta_j).
\end{align*}
Using the equality~\eqref{eq:important}, we obtain $\pactive =
\nabla_{\theta} \mcH(\theta)$. The constraint~\eqref{eq:force-constraints}
is the same as the constraint~\eqref{eq:f-constraints}. Therefore, $\theta$
is a solution for the elastic network problem~\eqref{eq:elastic_form}.

Regarding~\ref{p2:elasticnet} $\implies$ \ref{p1:flownet}, suppose
$\theta$ is a solution for the elastic network problem~\eqref{eq:flow_form}. For
every $e=(i,j)\in\mathcal{E}$, define $f_e$ by
\begin{align*}
  f_e =a_{ij} h_e(\theta_i-\theta_j).
\end{align*}
Then note that, for every $i\in \{1,\ldots,n\}$,
\begin{align*}
  (Bf)_i = \sum_{j=1}^{n}f_{(i,j)} = \sum_{j=1}^n a_{ij}
  h_e(\theta_i-\theta_j) = \frac{\partial \mcH}{\partial
  \theta_i}(\theta) = (\nabla_{\theta} \mcH(\theta))_i=(\pactive)_i,
\end{align*}
where the second last equality is because of equation~\eqref{eq:important}. The
constraint~\eqref{eq:f-constraints} is the same as the constraint~\eqref{eq:force-constraints}. Therefore,
$(f,\theta)$ is a solution for the flow network problem~\eqref{eq:flow_form}. 
\end{proof}

\subsection{Proof of Theorem~\ref{thm:acyclic}}

\begin{proof}
First note that since $G$ is connected and acyclic, we have $\Ker(B^{\top})=\{\vect{0}_{n-1}\}$. This implies that
$f=\mathcal{A}B^{\top}L^{\dagger}\pactive$ is the unique solution of
the flow balance equation~\eqref{eq:f-KCL}. 

Regarding \ref{p1:acyclic} $\Rightarrow$ \ref{p2:acyclic}, note that,
for every $e\in \mathcal{E}$, $h_e$ is an odd monotone function and
{\color{black}$\left|(B^{\top}L^{\dagger}\pactive)_e\right|\le
  |h_e(\gamma)|$}. This implies that there exists a unique $v\in \real^{n-1}$ with
$\|v\|_{\infty}\le \gamma$ such that
$h_e(v_e)=(B^{\top}L^{\dagger}\pactive)_e$, for every $e\in
\mathcal{E}$. Note that the graph $G$ is connected and
acyclic. We start from an arbitrary node $i$ in $G$ and
assign an arbitrary phase angle $\theta_i\in \torus^1$ to
the node $i$. Then, for every node $j$ that is a neighbor of $i$, we assign
$\theta_j\in \torus^1$ such that $\theta_i-\theta_j = v_e$, where
$e=(i,j)\in \mathcal{E}$. Now, we can repeat this process for every
neighbor of the node $i$ and continue to assign the phase angles to
the nodes of $G$. Since $G$ is connected, using this procedure, one would eventually assign
a phase angle to every node of the graph. Additionally, since $G$ is
acyclic, every node will get a unique phase angle. This implies that there exists a unique $\theta\in \torus^n$, modulo rotations, such that
\begin{align*}
v_e = \theta_i-\theta_j,\qquad\mbox{ for all } e=(i,j)\in \mathcal{E}.
\end{align*}
Thus, $(f,\theta)$ is a solution for the flow network
problem~\eqref{eq:flow_form}.

Regarding \ref{p2:acyclic} $\Rightarrow$ \ref{p1:acyclic}, if $(f,\theta)$ is a solution of the flow network
problem~\eqref{eq:flow_form}, then we have $f_e = a_{ij}h_e(\theta_i-\theta_j)$,
for every $e=(i,j)\in \mathcal{E}$ and therefore, {\color{black}
\begin{align*}
  \left|(B^{\top}L^{\dagger}\pactive)_e\right| =
  \left|h_e(\theta_i-\theta_j)\right|\le |h_e(\gamma)|,\qquad\mbox{ for all }
  e=(i,j)\in \mathcal{E},
\end{align*}}
where for the last inequality we used the fact that
$|\theta_i-\theta_j|\le \gamma$.
\end{proof}

\section{Proofs of results in Section~\ref{sec:winding-partition}}
It is worth mentioning that the proofs in this section are completely
independent of the results in Section~\ref{thm:equivalence}. Therefore, we
use some of these results for proving Theorem~\ref{thm:equivalence}.

\subsection{Proof of Theorem~\ref{thm:wind_prop}}

\begin{proof}
  {\color{black}
  We start by noting that, if $\sigma$ is a cycle in
  graph $G$ and $\theta\in \torus^n_{0}$, then  $w_{\sigma}(\theta) =
  v^{\top}_{\sigma}(B^{\top}\theta)$, where the vector $(B^{\top}\theta)$ is defined in
    equation~\eqref{def:BTtheta}.} Regarding
  part~\ref{p1:wind_integer_bound}, we use induction to show that
  $w_{\sigma}(\theta)$ is an integer. We start with $n_{\sigma} =
  3$. Consider the $3$-cycle $\sigma=(1,2,3,1)$.  Suppose that nodes
  $1,2,3$ are contained in an arc of length less than or equal to
  $\pi$ as shown in Figure~\ref{fig:triangle-winding} (left). Then we
  denote $d_{cc}(\theta_1,\theta_2) = \alpha$ and
  $d_{cc}(\theta_2,\theta_3)=\beta$. It is then clear that
  $\alpha+\beta < \pi$. But the counterclockwise arc from node $3$ to
  node $1$ has length larger than $\pi$. As a result, by definition of
  the counterclockwise difference, we have
  $d_{cc}(\theta_3) = -\alpha-\beta$. This means that
  \begin{align*}
    w_{\sigma}(\theta) = d_{cc}(\theta_1,\theta_2) +
    d_{cc}(\theta_2,\theta_3)+ d_{cc}(\theta_3,\theta_1) = \alpha +
    \beta -\alpha - \beta = 0. 
  \end{align*}
  Similar argument can be used to show that in Figure~\ref{fig:triangle-winding} (middle) and
  Figur~\ref{fig:triangle-winding} (right) the winding number of $\sigma$ is $+1$ and $-1$,
  respectively. Now suppose that, for $k\in \mathbb{Z}_{\ge 0}$, $w_{\eta}(\theta)$ is an integer for every cycle
 $\eta$ with $n_{\eta}\le k$. We show that every cycle with length
 $k+1$ has an integer winding number. Consider the cycle $\sigma =
 (1,2,\ldots,k+1,1)$. Define the cycles $\sigma'=(1,2,3,1)$ and $\sigma''
=(1,3,\ldots,k+1,1)$. Then a straightforward calculation shows that
$v_{\sigma} = v_{\sigma'}+v_{\sigma''}$. As a result, we have 
\begin{align*}
w_{\sigma}(\theta) = v^{\top}_{\sigma}(B^{\top}\theta)
  =v^{\top}_{\sigma'}(B^{\top}\theta) + v^{\top}_{\sigma''}(B^{\top}\theta) =
  w_{\sigma'}(\theta) + w_{\sigma''}(\theta). 
\end{align*}
Since both $\sigma'$ and $\sigma''$ has length less than or equal to
$k$, we have $w_{\sigma'}(\theta),w_{\sigma''}(\theta)\in
\mathbb{Z}$. This implies that $w_{\sigma}(\theta)\in \mathbb{Z}$. 

 Now suppose that $\sigma=(1,\ldots, n_{\sigma})$. Note that, for every $\alpha,\beta\in
\Scircle$ such that $|\alpha-\beta|\le \gamma$ we have
$|\subscr{d}{cc}(\alpha,\beta)|\le \gamma$. Thus
\begin{align*}
  |w_{\sigma}(\theta)| =\tfrac{1}{2\pi}
  \left|\sum_{i=1}^{n_{\sigma}}\subscr{d}{cc}(\theta_{i},\theta_{i+1})
  \right| \le \tfrac{1}{2\pi}
  \sum_{i=1}^{n_{\sigma}}|\subscr{d}{cc}(\theta_{i},\theta_{i+1})| \le
  \tfrac{1}{2\pi}n_{\sigma}\gamma(\theta).
\end{align*}
Since $w_{\sigma}(\theta)$ is an integer, we get
$|w_{\sigma}(\theta)|\le\left\lfloor{\frac{\gamma
      n_\sigma}{2\pi}}\right\rfloor$. For the second inequality, note that $|\subscr{d}{cc}(\alpha,\beta)|< \pi$. Thus
\begin{align*}
\frac{\gamma(\theta) n_{\sigma}}{2\pi} < \frac{\ n_{\sigma}}{2}.
\end{align*}
This implies that $\left\lfloor\tfrac{\gamma(\theta) n_{\sigma}}{2\pi}
\right\rfloor< \tfrac{\ n_{\sigma}}{2}$. Since $\left\lfloor\tfrac{\gamma(\theta) n_{\sigma}}{2\pi}
\right\rfloor$ is an integer, we have 
$
\left\lfloor\frac{\gamma(\theta) n_{\sigma}}{2\pi} \right\rfloor\le \left\lceil\frac{\ n_{\sigma}}{2}\right\rceil -1
$. Regarding part~\ref{p2:wind_piece_constant}, first note that
part~\ref{p1:wind_integer_bound} implies $|w_{\sigma_i}| \le
\left\lceil{\frac{n_\sigma}{2}}\right\rceil -1$, for every $i\in
\{1,\ldots,m-n+1\}$. In turn, this implies that
\begin{align*}
      \Img(\mathbf{w}_\Sigma) \subseteq
      \Bigsetdef{[u_1,\ldots,u_{m-n+1}]^{\top}}{u_i\in \mathbb{Z},
        \ |u_i|\le \lceil{{n_{\sigma_i}}/{2}}\rceil -1,\text{ for }
        i\in\{1,\ldots,m-n+1\}}, 
\end{align*} 
and therefore the winding map $\mathbf{w}_\Sigma$ has a finite
range. Moreover, the counterclockwise angle difference map
$\map{d_{\mathrm{cc}}}{\torus}{[-\pi,\pi)}$ is piecewise continuous on
  $\torus$. Thus $\map{\mathbf{w}_\Sigma}{\torus^n_0}{\integer^{m-n+1}}$ is
  a piecewise continuous map with a finite range. Therefore, the winding
  map $\mathbf{w}_\Sigma$ is piecewise constant and this completes the
  proof of part~\ref{p2:wind_piece_constant}.
\end{proof}

\subsection{Proof of Theorem~\ref{thm:partition}}
\begin{proof}
Regarding part~\ref{p2:union}, it is clear
that, for every $\mathbf{u}\in \Img(\mathbf{w}_\Sigma)$, we have
$\Omega^G_{\mathbf{u}}\subset \torus_0^n$. This implies that
$\bigcup_{\mathbf{u}\in\Img(\mathbf{w}_\Sigma)}\Omega^G_{\mathbf{u}}\subseteq
\torus_0^n$. Moreover, for every $\theta\in \torus_0^n$, we have
$\theta\in \Omega^G_{\mathbf{w}_{\Sigma}(\theta)}$. Therefore, we have
$\bigcup_{\mathbf{u}\in\Img(\mathbf{w}_\Sigma)}\Omega^G_{\mathbf{u}}=\torus_0^n$. Finally,
by taking closure of both side of this equality and noting that
$\closure(\torus_0^n) = \torus^n$, we get 
\begin{align*}
\bigcup_{\mathbf{u}\in\Img(\mathbf{w}_\Sigma)}\closure(\Omega^G_{\mathbf{u}})=\torus^n.
\end{align*}
Regarding part~\ref{p3:intersection}, suppose that for some $\mathbf{u}\ne
\mathbf{v}$, we have $\theta\in
\Omega^G_{\mathbf{u}}\cap\Omega^G_{\mathbf{v}}$. Then, $\mathbf{w}_\Sigma
(\theta)=\mathbf{u}=\mathbf{v}$, which is a contradiction. Therefore, we
have $\Omega^G_{\mathbf{u}}\cap \Omega^G_{\mathbf{v}}=\emptyset$.
\end{proof}

\subsection{Proof of Theorem~\ref{thm:linerization+correspondence}}

\begin{proof}
  Regarding part~\ref{p1:diff}, fix $\theta\in \Omega^G_{\mathbf{u}}$. By
  definition of the winding number, it is easy to see that, for every cycle
  $\sigma$, we have $w_{\sigma}(\theta) = \tfrac{1}{2\pi}
  v^{\top}_{\sigma}(B^{\top}\theta)$, {\color{black} where the vector
    $(B^{\top}\theta)$ is defined in equation~\eqref{def:BTtheta}.}
  Applying this formula to every cycle in the cycle basis $\Sigma$, we get
  $\mathbf{u} = \mathbf{w}_{\Sigma}(\theta) = \tfrac{1}{2\pi}
  C_{\Sigma}(B^{\top}\theta)$. Multiplying both side of this equality by
  $C^{\dagger}_{\Sigma}$, we get $2\pi C^{\dagger}_{\Sigma}\mathbf{u} =
  C^{\dagger}_{\Sigma}C_{\Sigma}(B^{\top}\theta)$. {\color{black}Note that,
    by properties of the Moore\textendash{}Penrose inverse, we have
    $C^{\dagger}_{\Sigma}C_{\Sigma}C^{\dagger}_{\Sigma} =
    C^{\dagger}_{\Sigma}$. This implies that
\begin{align}\label{eq:winding_vector_C}
C^{\dagger}_{\Sigma}C_{\Sigma}\left((B^{\top}\theta)-2\pi
  C^{\dagger}_{\Sigma}\mathbf{u}\right) = \vect{0}_m
\end{align}
On the other hand, by properties of the Moore\textendash{}Penrose inverse, we have 
$C_{\Sigma}C_{\Sigma}^{\dagger}C_{\Sigma} = C_{\Sigma}$. Now suppose that
$\alpha \in \Ker(C_{\Sigma}^{\dagger}C_{\Sigma})$. This means that
$C_{\Sigma}^{\dagger}C_{\Sigma}\alpha = \vect{0}_m$. Multiplying both
side of this equality by $C_{\Sigma}$, we get $C_{\Sigma}\alpha = \vect{0}_m$. This
implies that $\alpha \in \Ker(C_{\Sigma})$. Therefore, we can deduce
that  $\Ker(C^{\dagger}_{\Sigma} C_{\Sigma})\subseteq
\Ker(C_{\Sigma})$. Moreover it is easy to show that $\Ker(C_{\Sigma})\subseteq
\Ker(C^{\dagger}_{\Sigma} C_{\Sigma})$. Therefore, we get $\Ker(C_{\Sigma}) = \Ker(C_{\Sigma}^{\dagger}C_{\Sigma})$. In turn, by Lemma~\ref{thm:shift_property}\ref{p2:ker}, we have $\Ker({C}_{\Sigma}) =
\Img(B^{\top})$. This implies that
$\Ker(C_{\Sigma}^{\dagger}C_{\Sigma})=\Img(B^{\top})$ and thus,
using the equation~\eqref{eq:winding_vector_C}, there exists a unique $\mathbf{x}\in
\vect{1}_n^{\perp}$ such that $(B^{\top}\theta) - 2\pi
C^{\dagger}_{\Sigma}\mathbf{u}= B^{\top}\mathbf{x}$. In other
words, there exists $x\in \vect{1}_n^{\perp}$ such that
$(B^{\top}\theta) = B^{\top}\mathbf{x}+ 2\pi
C^{\dagger}_{\Sigma}\mathbf{u}$}. This completes the proof
of~\ref{p1:diff}. Regarding part~\ref{p1.5:connec-comp}, we define the map
  $\map{\iota}{\Omega^G_{\mathbf{u}}/\torus^1}{P_{\mathbf{u}}}$
  by $\iota([\theta]) = \mathbf{x}$, where $\mathbf{x}\in \vect{1}_n^{\perp}$ is such that {\color{black} $(B^{\top}\theta) =
  B^{\top}\mathbf{x} + 2\pi C^{\dagger}_{\Sigma}\mathbf{u}$}. We show
  that $\iota$ is a {\color{black}bijection} between
  $\Omega^G_{\mathbf{u}}/\torus^1$ and $P_{\mathbf{u}}$. Consider the linear equations 
\begin{align}\label{eq:linear-winding}
C_{\Sigma}\mathbf{z} = \mathbf{u}.
\end{align}
By Lemma~\ref{thm:shift_property}\ref{p3:integer-basis}, the system of
linear equations~\eqref{eq:linear-winding} has an integer solution
$\mathbf{z}$. By Lemma~\ref{thm:shift_property}\ref{p1:rank}, the matrix $C_{\Sigma}$ is of rank
$m-n+1$ and thus it has linearly independent rows. This implies that
$C_{\Sigma}C_{\Sigma}^{\dagger} = I_{m-n+1}$. Therefore,  we have $C_{\Sigma}C_{\Sigma}^{\dagger}\mathbf{u} =
\mathbf{u}$ and in turn $C^{\dagger}_{\Sigma}\mathbf{u}$ is another
solution for~\eqref{eq:linear-winding}. Since both
$\mathbf{z}$ and $C_{\Sigma}^{\dagger}\mathbf{u}$ are solutions for~\eqref{eq:linear-winding}, we have
$\mathbf{z} - C^{\dagger }\mathbf{u}\in \Ker(C_{\Sigma})$. Using
Theorem~\ref{thm:shift_property}\ref{p2:ker}, we have $\Ker(C_{\Sigma}) =
\Img(B^{\top})$ and there exists $\alpha\in \vect{1}^{\perp}_n$ such that
$\mathbf{z} - C^{\dagger}_{\Sigma}\mathbf{u} = B^{\top}\alpha$.
Now, we define the open set $D =\{\mathbf{x}\in \vect{1}_n^{\perp}\mid \|B^{\top}\mathbf{x} + 2\pi
C^{\dagger}_{\Sigma}\mathbf{u}\|_{\infty} < \pi \}$. We first show
that the map $\iota$ is injective. Suppose that $[\theta_1],[\theta_2]\in
\Omega^G_{\mathbf{u}}/\Scircle$ are such that $\iota([\theta_1]) =
\iota([\theta_2])=\mathbf{x}$. This implies that {\color{black}$(B^{\top}\theta_1) =
(B^{\top}\theta_2) = B^{\top}\mathbf{x} + 2\pi C^{\dagger}_{\Sigma}\mathbf{u}$} and as
a result, we get $\theta_1 = \mathrm{rot}_s(\theta_2)$, for some $s
\in [-\pi,\pi)$. This completes the proof of the fact that $\iota$ is
injective. Now we show that $\iota$ is surjective. Let $\mathbf{x}\in
D$. Then we define $\theta\in \Scircle$ by
\begin{align*}
\theta=\mathrm{mod}(\mathbf{x} -2\pi \alpha, 2\pi).
\end{align*}
Suppose that $e = (i,j)\in
\mathcal{E}$. Then we have 
\begin{align*}
\left(B^{\top}\theta\right)_e = \subscr{d}{cc}(\theta_i,\theta_j) = (x_i
  - 2\pi \alpha_i) -
  (x_j - 2\pi \alpha_j) + 2\pi z_i, 
\end{align*}
where the second equality is because of the fact that 
\begin{align*}
\left|(x_i - 2\pi \alpha_i)-(x_j - 2\pi \alpha_j) + 2\pi z_i\right| \le
  \|B^{\top}\mathbf{x} - 2\pi \alpha + 2\pi\mathbf{z}\|_{\infty} = \|B^{\top}\mathbf{x} + 2\pi C^{\dagger}_{\Sigma}\mathbf{u}\|_{\infty} < \pi.
\end{align*}
Therefore, we have {\color{black}$(B^{\top}\theta ) = B^{\top}(\mathbf{x} -
  2\pi \alpha) + 2\pi \mathbf{z} = B^{\top}\mathbf{x} + 2\pi
  C^{\dagger}_{\Sigma} \mathbf{u}$.} This implies that $\iota([\theta]) =
\mathbf{x}$ and completes the proof of bijection of
$\iota$. {\color{black}Now we show that
  $\map{\iota}{\Omega^G_{\mathbf{u}}/\torus^1}{P_{\mathbf{u}}}$ is a
  continuous map. Let us define the map
  $\map{\xi}{\Omega^G_{\mathbf{u}}}{P_{\mathbf{u}}}$ by $\xi(\theta) =
  \iota([\theta])$, for every $\theta \in \Omega^G_{\mathbf{u}}$. Then,
  by~\cite[Theorem 22.2]{JM:00}, the map $\iota$ is continuous if and only
  if the map $\xi$ is continuous. Suppose that $\phi,\psi\in
  \Omega^G_{\mathbf{u}}$ and $y = \iota([\phi])$ and $z =
  \iota([\psi])$. Note that, the graph $G$ is connected. Thus,
  by~\cite[Lemma 6.12(iii)]{FB:20}, we have $(BB^{\top})^{\dagger}BB^{\top}
  = I_n - \frac{1}{n}\vect{1}_n\vect{1}_{n}^{\top}$. As a result, we have
\begin{align}\label{eq:inequality-Btheta}
  \|y- z\|\le \|(BB^{\top})^{\dagger}B\|\|B^{\top}y - B^{\top}z\|= \|(BB^{\top})^{\dagger}B\|\|(B^{\top}\phi) - (B^{\top}\psi)\|.
\end{align}
On the other hand, since $\iota$ is a bijection, we have
$\left\|(B^{\top}\theta)\right\|_{\infty}<\pi$, for every $\theta\in
\Omega^{G}_{\mathbf{u}}$. As a result, the map $\theta \mapsto
(B^{\top}\theta)$ is continuous on $\Omega^{G}_{\mathbf{u}}$. This together
with the inequality~\eqref{eq:inequality-Btheta} imply that the map $\xi:
\theta \mapsto x$ is continuous and thus $\iota$ is continuous. Moreover,
$P_{\mathbf{u}}$ is a Hausdorff space. Therefore, by~\cite[Theorem
  26.6]{JM:00}, the map $\iota$ is a homeomorphism on any compact subset of
$\Omega^{G}_{\mathbf{u}}/\torus^1$.}\end{proof}

\section{Proofs of results in Section~\ref{sec:at-most-uniqueness}}

\subsection{Proof of Theorem~\ref{thm:at_most_unique}}\label{app:atmostuniqueness}

\begin{proof}
  Suppose that $(f,\phi)$ and $(g,\psi)$ are two solutions for the flow
  network problem~\eqref{eq:flow_form} with
  $\mathbf{w}_{\Sigma}(\phi)=\mathbf{w}_{\Sigma}(\psi)=\mathbf{u}$. Then,
  by Theorem~\ref{thm:equiv-kur-flow}, $f$ and $g$ satisfies the
  $\mathbf{u}$ winding balance equation~\eqref{eq:winding-form}. Thus, we
  have $f,g\in \Fsd$ and
  \begin{align*}
    T_{\mathbf{u}}(f) = f - \Hmin \prjcyc
    (h_{\gamma}^{-1}(\mathcal{A}^{-1}f) - 2\pi
    C^{\dagger}_{\Sigma}\mathbf{u}) = f
  \end{align*}
  where the last equality holds because $f$
  satisfies~\eqref{eq:winding-balance}. Similarly, one can show that $T_{\mathbf{u}}(g)
  = g$. Thus, by Theorem~\ref{thm:convergence_ifexists}\ref{p1:converge}, we get $f=g$. By the
  equivalence of parts~\ref{p1:check} and~\ref{p2:exists_f} in
  Theorem~\ref{thm:convergence_ifexists}, we get that $\phi =
  \psi$. As a result, there should exists at most one solution for the
  flow network problem~\ref{eq:flow_form} with the phase angle in the winding cell
  $\Omega^G_{\mathbf{u}}$. 
\end{proof}

\subsection{Proof of Corollary~\ref{cor:graphs-short-cycles}}

\begin{proof}
Regarding part~\ref{p1:shortbasis}, the result is already proved in
Theorem~\ref{thm:at_most_unique}. {\color{black}Regarding part~\ref{p2:shortbasis}, suppose that $(f,\theta)$ is a solution of the
flow network problem~\eqref{eq:flow_form} such that $\theta\in
\Omega^G_{\mathbf{u}}$. Then, for every $\sigma\in \Sigma$,
\begin{align*}
  |w_{\sigma}(\theta)| =
  \tfrac{1}{2\pi}\left|\sum_{i=1}^{n_{\sigma}}\subscr{d}{cc}(\theta_i,\theta_{i+1})\right|
  \le
  \tfrac{1}{2\pi}\sum_{i=1}^{n_{\sigma}}\left|\subscr{d}{cc}(\theta_i,\theta_{i+1})\right|
  \le \tfrac{1}{2\pi} n_{\sigma}\gamma, 
\end{align*}
where for the last equality we used 
$\left|\subscr{d}{cc}(\theta_i,\theta_{i+1})\right|=|\theta_i-\theta_j|
\le \gamma$. Since $w_{\sigma}(\theta)$ is an integer, we get
\begin{align}\label{eq:upperbound}
 |w_{\sigma}(\theta)| \le
  \left\lfloor\frac{n_{\sigma}\gamma}{2\pi}\right\rfloor. 
\end{align}
Moreover, we know that $n_{\sigma}\le k$. This implies
that $|w_{\sigma}(\theta)|\le \tfrac{1}{2\pi} \gamma\gamma$. As a result, we
should have $\|\mathbf{u}\|_{\infty}\le
\left\lfloor\frac{k\gamma}{2\pi}\right\rfloor$. This implies that, if $\|\mathbf{u}\|_{\infty}>
\left\lfloor\frac{k\gamma}{2\pi}\right\rfloor$, we have
$\mathbf{u}\not\in \Img(\mathbf{w}_{\Sigma})$ and thus there is no solution $(f,\theta)$
for the flow network problem~\eqref{eq:flow_form} such that $\theta\in
\Omega^G_{\mathbf{u}}$. Regarding part~\ref{p3:upperbound}, using the
inequality~\eqref{eq:upperbound}, we get $\Img(\mathbf{w}_{\Sigma})
\le \prod_{i=1}^{m-n+1}|w_{\sigma}(\theta)| \le \prod_{i=1}^{m-n+1}\left\lfloor\frac{n_{\sigma_i}\gamma}{2\pi}\right\rfloor$.}
\end{proof}

\subsection{Proof of Theorem~\ref{thm:prop-power-flow}}

\begin{proof}
{\color{black} Recall that, for $\theta\in \torus^n$, the vector $(B^{\top}\theta)$ is defined
    by equation~\eqref{def:BTtheta}.} Regarding part~\ref{p1:decomposition}, note that
  $\Img(\mathcal{A}B^{\top})\oplus\Ker(B) = \real^m$. Therefore, there exists a unique
  decomposition $f = \supscr{f}{cut} + \supscr{f}{cyc}$, where
  $\supscr{f}{cut}\in \Img(\mathcal{A}B^{\top})$ and $\supscr{f}{cyc}\in
  \Ker(B)$. Additionally, the {\color{black}edge} vector $f\in \real^m$ satisfies the
  flow balance equation~\eqref{eq:f-KCL}, that is $Bf =\pactive$. Note
  that $\supscr{f}{cyc}\in \Ker(B)$ and this implies that $B\supscr{f}{cut}
  = \pactive$. Moreover, $\subscr{f}{cut} \in \Img(\mathcal{A}B^{\top})$ and therefore
  there exists $x\in \vect{1}_n^{\perp}$ such that $\supscr{f}{cut} =
  \mathcal{A}B^{\top}x$. This implies that $B\mathcal{A} B^{\top} x = L x = \pactive $. Finally, one
  can compute $\supscr{f}{cut} = \mathcal{A}B^{\top}x = \mathcal{A}B^{\top}L^{\dagger}\pactive$.  This completes the proof of
  part~\ref{p1:decomposition}. Regarding part \ref{p3:f-unique}, note that, we have $\pactive = Bf
  = B g$. This implies that $B(f-g)=0$ and in turn
  $f-g\in \Ker(B)$. Regarding part \ref{p4:f-unique-to-u}, given $\phi=\psi$, it is trivial to see that $f=g$.
  Now suppose that, for $\phi,\psi$, we have $f=g$. This implies that,
  for every $e=(i,j)\in \mathcal{E}$, we have 
  $h_e(\phi_i-\phi_j) = h_e(\psi_i-\psi_j)$. Since $h_e$ is
  monotone on the interval $[-\gamma,\gamma]$, it is invertible on this
  interval. This implies that $\phi_i-\phi_j =
  \psi_i-\psi_j$, for every $(i,j)\in \mathcal{E}$. This means that {\color{black}$(B^{\top}\phi)=(B^{\top}\psi)$}. By Theorem~\ref{thm:linerization+correspondence}\ref{p1:diff}, there exists $\mathbf{x},\mathbf{y}\in
  \vect{1}^{\perp}_n$ such that {\color{black}
  \begin{align*}
    (B^{\top}\phi) = B^{\top}\mathbf{x} + 2\pi C^{\dagger}_{\Sigma}(\mathbf{w}_{\Sigma}(\phi)),\\
    (B^{\top}\psi) = B^{\top}\mathbf{y} + 2\pi C^{\dagger}_{\Sigma}(\mathbf{w}_{\Sigma}(\psi))
  \end{align*}}
  We multiply both side of the above equations by $C_{\Sigma}$. Note
  that, by 
  Lemma~\ref{thm:shift_property}\ref{p2:ker} we have
  $\Ker(C_{\Sigma})=\Img(B^{\top})$. This implies that
  \begin{align*}
    C_{\Sigma}C^{\dagger}_{\Sigma}(\mathbf{w}_{\Sigma}(\psi)-\mathbf{w}_{\Sigma}(\phi))
    = \vect{0}_{m-n+1}.
  \end{align*}
  By Lemma~\ref{thm:shift_property}\ref{p1:rank}, the matrix
  $C_{\Sigma}$ is of rank $m-n+1$ and therefore $C_{\Sigma}$ has
  linearly independent rows. Thus, $C_{\Sigma}C^{\dagger}_{\Sigma}= I_{m-n+1}$ and this implies that
  $\mathbf{w}_{\Sigma}(\psi)=\mathbf{w}_{\Sigma}(\phi)$. Now suppose
  that $\mathbf{w}_{\Sigma}(\psi)=\mathbf{w}_{\Sigma}(\phi)$, the
  $(f,\phi)$ and $(g,\psi)$ are two solutions for the flow network
  problem~\eqref{eq:flow_form} with the property that $\phi,\psi\in
  \Omega^G_{\mathbf{u}}$. Thus, by Theorem~\eqref{thm:at_most_unique}, we have $f=g$ and
  $\phi=\psi$ modulo rotations. Finally, if $\phi = \mathrm{rot}_s(\psi)$ for
  some $s\in [-\pi,\pi)$, then by equation~\eqref{eq:f-physics}, it is clear that
  $f=g$. Regarding part~\ref{p5:flow_comparison}, we start by introducing the
  function $h:\real^m\to \real^m$ defined by
  \begin{align*}
    (h(y))_{e} = h_e(y_e),\qquad\mbox{for every } e\in \mathcal{E}.
    \end{align*}
    By the angle
    constraint~\eqref{eq:f-constraints}, for every $e\in \mathcal{E}$,
    we have $|\theta_i-\theta_j|\le \gamma$ and $|\psi_i-\psi_j|\le
    \gamma$. Note also that, for every $e\in \mathcal{E}$, 
    the flow $h_e$ is a strictly increasing on the interval
    $[-\gamma,\gamma]$. This implies that
    \begin{align*}
      a_{ij}\left((\phi_i-\phi_j) -
      (\psi_i-\psi_j)\right)\left(h_e(\phi_i-\phi_j)-h_e(\psi_i-\psi_j)\right)
      \ge 0. 
      \end{align*}
      Summing the above equation over all $e=(i,j)\in \mathcal{E}$, we get
    \begin{align}\label{eq:inequality_good}
      \left(B^{\top}\phi -
      B^{\top}\psi\right)^{\top}\mathcal{A}\left(h(B^{\top}\phi) - h(B^{\top}\psi)\right) \ge 0.
    \end{align}
    Moreover, by Theorem~\ref{thm:partition}\ref{p1:diff}, there exists $\mathbf{x},\mathbf{x}'\in
    \vect{1}_n^{\perp}$ such that {\color{black}
    \begin{align}\label{eq:toomanylabel}
      (B^{\top}\phi) = B^{\top}\mathbf{x} + 2\pi
      C_{\Sigma}^{\dagger}w_{\sigma}(\phi),\qquad
      (B^{\top}\psi) = B^{\top}\mathbf{x}' + 2\pi C_{\Sigma}^{\dagger}w_{\sigma}(\psi).
    \end{align}}
    Replacing~\eqref{eq:toomanylabel} into~\eqref{eq:inequality_good}, we get 
    \begin{align}
      \left((B^{\top}\phi) - (B^{\top}\psi)\right)^{\top}&  \mcA
                                               (h(B^{\top}\phi) - h(B^{\top}\phi)) \nonumber \\
                                             & = \left(B^{\top}(\mathbf{x}-\mathbf{x}') + 2\pi
                                               C^{\dagger}_{\Sigma}(w_{\sigma}(\phi)-w_{\sigma}(\psi))\right)^{\top}
                                               \mcA  \left(h(B^{\top}\phi) - h(B^{\top}\phi)\right) \nonumber
      \\
                                             & = \left(B^{\top}(\mathbf{x}-\mathbf{x}')\right)^{\top}
                                               \mcA  \left(h(B^{\top}\phi) - h(B^{\top}\phi)\right) \label{eq:zero-term}
      \\
                                             &  \phantom{=}\enspace + \left(2\pi
                                               C^{\dagger}_{\Sigma}(w_{\sigma}(\phi)-w_{\sigma}(\psi))\right)^{\top}  \mcA
                                               \left(h(B^{\top}\phi) - h(B^{\top}\phi)\right). \label{eq:nonzero-term}
    \end{align}
    Note that the term~\eqref{eq:zero-term} is 
    \begin{equation*}
      \left(B^{\top}(\mathbf{x}-\mathbf{x}')\right)^{\top}
      \mcA \left(h(B^{\top}\phi) - h(B^{\top}\phi)\right)
      = (\mathbf{x} - \mathbf{x}')^{\top} B\mcA
      \left(h(B^{\top}\phi) - h(B^{\top}\phi)\right). 
    \end{equation*}
    Since $(f,\phi)$ and $(g,\psi)$ are solutions for the flow network problem~\eqref{eq:flow_form}, we get
    \begin{align*}
      B\mcA \left(h(B^{\top}\phi) - h(B^{\top}\phi)\right) =
      B (f - g) = \pactive-\pactive = \vect{0}_n. 
    \end{align*}
    Therefore, the term~\eqref{eq:zero-term} is equal to zero. Moreover,
    since $\sigma$ is the only cycle for $G$, we have $C_{\Sigma}^{\dagger} =
    \frac{1}{n}v_{\sigma}$. Therefore, the term~\eqref{eq:nonzero-term} can be written as 
    \begin{align*}
      \left(2\pi C^{\dagger}_{\Sigma}(w_{\sigma}(\phi)-w_{\sigma}(\psi))\right)^{\top}  &\mcA
                                                                                          \left(h(B^{\top}\phi) - h(B^{\top}\phi)\right)
      \\ & = \frac{1}{n}(w_{\sigma}(\phi)-w_{\sigma}(\psi)) v_{\sigma}^{\top} \mcA
           \left(h(B^{\top}\phi) - h(B^{\top}\phi)\right)
      \\  &= \frac{1}{n}(w_{\sigma}(\phi)-w_{\sigma}(\psi)) \left(v_{\sigma}^{\top}f -v_{\sigma}^{\top}g\right).
    \end{align*}
    Therefore, the inequality~\eqref{eq:inequality_good} can be written as
    $\frac{1}{n}(w_{\sigma}(\phi)-w_{\sigma}(\psi))
    \left(v_{\sigma}^{\top}f -v_{\sigma}^{\top}g\right)\ge
    0$. This completes the proof of part~\ref{p5:flow_comparison}.
  \end{proof}

\section{Proofs of results in Section~\ref{sec:complete-solver}}
It is worth mentioning that the proofs in this section are independent of
the results in Section~\ref{sec:at-most-uniqueness}. Therefore, we use some
of the results in this section to prove Theorem~\ref{thm:at_most_unique}.

  \subsection{Proof of Theorem~\ref{thm:equiv-kur-flow}}
  
  \begin{proof}
    {\color{black} Recall that, for $\theta\in \torus^n$, the vector $(B^{\top}\theta)$ is defined
    by equation~\eqref{def:BTtheta}.} Regarding~\ref{p1:kur_equi} $\Longrightarrow$ \ref{p2:loop_flow_equi},
    Suppose that $(f,\theta)$ is a solution of
    flow network problem~\eqref{eq:flow_form} with the property that $\theta \in
    \Omega^G_{\mathbf{u}}$. First note, for every $e=(i,j)\in
    \mathcal{E}$, we have $f_e =a_{ij} h_e(\theta_i-\theta_j)$ and thus $a_{ij}^{-1}f_e = h_e(\theta_i-\theta_j)$. Using the
    angle constraint~\eqref{eq:f-constraints}, we have
    $|\theta_i-\theta_j|\in [-\gamma,\gamma]$, for every $e=(i,j)\in
    \mathcal{E}$. This implies that {\color{black}
    \begin{align}\label{eq:1}
      |f_e|\le a_{ij}|h_e(\gamma)|,\qquad\mbox{for } e=(i,j)\in \mathcal{E}.
    \end{align}}
    On the other hand, for every $y\in [-\gamma,\gamma]$ and every
    $e=(i,j)\in \mathcal{E}$, we have $h_e(y)
    = (h_{\gamma})_e(y)$. This means that $f_e =a_{ij}
    h_e(\theta_i-\theta_j) =
    a_{ij}(h_{\gamma})_e(\theta_i-\theta_j)$. Since $(h_{\gamma})_e$ is
    monotone on $\real$, it is invertible and $\theta_i-\theta_j=
    (h_{\gamma})_e^{-1}(a_{ij}^{-1}f_e)$. In the vector form, we get {\color{black}
    \begin{align}\label{eq:angle}
      (B^{\top}\theta)= h_{\gamma}^{-1}(\mcA^{-1}f).
    \end{align}}
    Using Theorem~\ref{thm:linerization+correspondence}\ref{p1:diff}, there exists $\mathbf{x}\in
    \vect{1}_n^{\perp}$ such that {\color{black}$(B^{\top}\theta) = B^{\top}\mathbf{x} + 2\pi
    C^{\dagger}_{\Sigma}\mathbf{u}$}. Plugging in the equation~\eqref{eq:angle},
    we get 
    \begin{align}\label{eq:angle2}
      h_{\gamma}^{-1}(\mcA^{-1}f) - 2\pi C^{\dagger}_{\Sigma}\mathbf{u} = B^{\top}\mathbf{x}.
    \end{align} 
    Multiplying both side of equation~\eqref{eq:angle2} by $\prjcyc$ and
    noting the fact that $\prjcyc \Hmin \mcA B^{\top}\alpha = \vect{0}_m$, we get
    \begin{align}\label{eq:2}
      \prjcyc \Hmin \mcA \left(h_{\gamma}^{-1}(\mcA^{-1}f) - 2\pi
      C^{\dagger}_{\Sigma}\mathbf{u}\right)=\vect{0}_m.
    \end{align}
    Combining inequality~\eqref{eq:1} with the equality~\eqref{eq:2},
    we deduce that $f$ is a solution of the $\mathbf{u}$-winding balance
    equation~\eqref{eq:winding-form}. Regarding~\ref{p2:loop_flow_equi} $\Longrightarrow$~\ref{p1:kur_equi},
    suppose that $f\in \real^m$ is a solution for the $\mathbf{u}$-winding balance
    equation~\eqref{eq:winding-form}. Note that $\Ker(\prjcyc \Hmin \mcA) =
    \Img(B^{\top})$. Thus, by equality~\eqref{eq:winding-equations},
    there exists $\mathbf{x}\in \vect{1}_n^{\perp}$ such that
    \begin{align*}
      h_{\gamma}^{-1}(\mcA^{-1}f) = B^{\top}\mathbf{x} +  2\pi
      C^{\dagger}_{\Sigma}\mathbf{u}.
    \end{align*}
    Note that, by the constraint~\eqref{eq:winding-flow-constraint}, for every $e=(i,j)\in
    \mathcal{E}$, we have {\color{black}
    \begin{align*}
      \left|a_{ij} f_e \right|\le |h_e(\gamma)|=\max_{y\in [-\gamma,\gamma]}h_e(y). 
    \end{align*}}
    Since, for each $e\in \mathcal{E}$, $(h_{\gamma})_e$ is monotone on $\real$ and
    $(h_{\gamma})_e(y) = h_e(y)$, for every $y\in [-\gamma,\gamma]$,
    we get that $\left\|h^{-1}_{\gamma}(\mathcal{A}^{-1}f)\right\|_{\infty} \le
    \gamma$. As a result, we have
    \begin{align}\label{eq:inequality}
      \left\|B^{\top}\mathbf{x} +  2\pi C^{\dagger}_{\Sigma}\mathbf{u}\right\|_{\infty}\le \gamma
    \end{align}
    {\color{black}This means that $\mathbf{x}\in P_{\mathbf{u}}$}. Now, by Theorem~\ref{thm:linerization+correspondence}\ref{p1.5:connec-comp} there exists
    $\theta\in \Omega^G_{\mathbf{u}}$ such that {\color{black}$(B^{\top}\theta) = B^{\top}\mathbf{x} +  2\pi
    C^{\dagger}_{\Sigma}\mathbf{u}$ and thus $h_{\gamma}^{-1}(\mcA^{-1}f) = (B^{\top}\theta)$}. On the other hand, for every $y\in [-\gamma,\gamma]$ and every
    $e=(i,j)\in \mathcal{E}$, we have $h_e(y)
    = (h_{\gamma})_e(y)$. As a result, for
    every $e=(i,j)\in \mathcal{E}$,
    \begin{align}\label{eq:goodgood}
      f_e = a_{ij} h_e(\theta_i-\theta_j).
    \end{align}
    Moreover, the inequality~\eqref{eq:inequality} implies that 
    \begin{align}\label{eq:theta-inequality}
  |\theta_i-\theta_j|\le \gamma. 
    \end{align}
    The equations~\eqref{eq:goodgood},~\eqref{eq:winding-balance},
    and~\eqref{eq:theta-inequality} imply that
    $(f,\theta)$ is a solution for the flow network problem. Finally, we
    show that $\theta$ is the unique phase
    angle vector in $\Omega^G_{\mathbf{u}}$ for which $(f,\theta)$ is a solution of flow network
    problem~\eqref{eq:flow_form}. Suppose that $(f,\phi)$ is another
    solution for flow network problem~\eqref{eq:flow_form} with $\phi\in
    \Omega^G_{\mathbf{u}}$. For every $e=(i,j)\in \mathcal{E}$, we have $f_e=a_{ij}h_e(\theta_i-\theta_j) = a_{ij}
    h_e(\phi_i-\phi_j)$. Moreover,  $|\theta_i-\theta_j|\le \gamma$
    and $|\phi_i-\phi_j|\le \gamma$ and $h_e$ is monotone on
    $[-\gamma,\gamma]$. This implies that $\theta_i-\theta_j =
    \phi_i-\phi_j$, for every $(i,j)\in \mathcal{E}$. In the vector
    form, this leads to {\color{black}$(B^{\top}\theta) = (B^{\top}\phi)$}. Note that
    $\theta,\phi\in \Omega^G_{\mathbf{u}}$. Thus, there exists $\mathbf{x},\mathbf{y}\in
    P_{\mathbf{u}}$ such that the {\color{black}bijection} in
    Theorem~\ref{thm:linerization+correspondence}\ref{p1.5:connec-comp} maps $\theta$ to $\mathbf{x}$ and maps $\phi$
    to $\mathbf{y}$. Using Theorem~\ref{thm:linerization+correspondence}\ref{p1:diff}, we get  {\color{black}
    \begin{align*}
      (B^{\top}\theta) = B^{\top}\mathbf{x} + 2\pi
      C_{\Sigma}^{\dagger}\mathbf{u},\qquad (B^{\top}\phi) = B^{\top}\mathbf{y} + 2\pi C_{\Sigma}^{\dagger}\mathbf{u}.
      \end{align*}}Since {\color{black}$(B^{\top}\theta) = (B^{\top}\phi)$}, we get that
$B^{\top}(\mathbf{x}-\mathbf{y})=\vect{0}_m$. Since $G$ is strongly
connected, $\Ker(B^{\top})=\mathrm{span}\{\vect{1}_n\}$. Since both $\mathbf{x}$
and $\mathbf{y}$ are in $\vect{1}_n^{\perp}$, this implies that
$\mathbf{x}=\mathbf{y}$ and as a result $\theta = \phi$, modulo rotations.
\end{proof}

\subsection{Proof of Theorem~\ref{thm:convergence_ifexists}. }

\begin{proof}
Regarding parts~\ref{p1:converge} and~\ref{p2:speed}, note that, for every $\eta_1,\eta_2\in\Fsd$, 
\begin{align*}
T_{\mathbf{u}}(\eta_1) - T_{\mathbf{u}}(\eta_2) = \eta_1 - \eta_2 +
\prjcyc\Hmin \mcA\left(h_{\gamma}^{-1} (\eta_1) - h_{\gamma}^{-1} (\eta_2)\right).
\end{align*}
Since $\eta_1,\eta_2\in \Fsd$, we get $\eta_1-\eta_2\in \Ker(B)$. This
implies that $\eta_1-\eta_2 = \prjcyc (\eta_1-\eta_2)$. As a result, 
\begin{align*}
T_{\mathbf{u}}(\eta_1) - T_{\mathbf{u}}(\eta_2) = \prjcyc
\left(\eta_1 - \eta_2 - \Hmin \mcA\left(h_{\gamma}^{-1} (\eta_1) - h_{\gamma}^{-1} (\eta_2)\right)\right).
\end{align*}
For the brevity of notation, we introduce the positive diagonal matrix
$D\in \real^{m\times m}$ by $D = \Hmin\mcA$. By the Mean Value Inequality~\cite[Proposition
2.4.7]{RA-JEM-TSR:88}, 
\begin{align*}
\|T_{\mathbf{u}}(\eta_1) - T_{\mathbf{u}}(\eta_2)\|_{D} \le \sup_{x\in 
\real^m}\left\|\prjcyc\left(I_m - \Hmin\nabla h_{\gamma}^{-1}(x)\right)\right\|_{D}\|\eta_1-\eta_2\|_{D}. 
\end{align*} 
Using the fact that $\|\mathbf{x}\|_{D} =
\|D^{\frac{1}{2}}\mathbf{x}\|_{2}$, we get  
\begin{align*}
\left\|\prjcyc\left(I_m -\Hmin\nabla h_{\gamma}^{-1}(x)\right)\right\|_{D} =\left\|D^{\frac{1}{2}}\prjcyc\left(I_m -\Hmin\nabla h_{\gamma}^{-1}(x)\right) D^{-\frac{1}{2}}\right\|_2.  
\end{align*} 
Note that, by triangle inequality, we have
\begin{align*}
  \big\|D^{\frac{1}{2}}  \prjcyc\big(I_m - \Hmin\nabla h_{\gamma}^{-1}(x)\big) D^{-\frac{1}{2}}\big\|_2 
  & \le \left\|D^{\frac{1}{2}}\prjcyc
    D^{-\frac{1}{2}}\right\|_2\left\|D^{\frac{1}{2}}\left(I_m -
    \Hmin\nabla h_{\gamma}^{-1}(x)\right) D^{-\frac{1}{2}}\right\|_2 \\
  & = \left\|D^{\frac{1}{2}}\left(I_m - \Hmin\nabla h_{\gamma}^{-1}(x)\right)D^{-\frac{1}{2}}\right\|_2,
\end{align*}
where in the last inequality we used the fact that
$D^{\frac{1}{2}}\prjcyc
  D^{-\frac{1}{2}}$ is a symmetric idempotent matrix and
  therefore, its $2$-norm is equal to $1$. Moreover, for every $x\in
\real^m$, $\nabla h^{-1}_{\gamma}(x)$ is a diagonal matrix such that, for every
$i\in \{1,\ldots,m\}$, we have 
\begin{align*}
\left|(\nabla h^{-1}_{\gamma}(x))_{ii}\right| \le (\Hmax^{-1})_{ii},\qquad\mbox{ for all } x\in \real^m
\end{align*}
Additionally, by~\cite[Theorem 5.6.36]{RAH-CRJ:12}, we have 
\begin{align*}
\left\|D^{\frac{1}{2}}\left(I_m - \Hmin\nabla h_{\gamma}^{-1}(x)\right)D^{-\frac{1}{2}}\right\|_2 = \left\|I_m - \Hmin\nabla h_{\gamma}^{-1}(x)\right\|_{\infty} = \max_i\left\{\left|1 - (\Hmin)_{ii}\left(\nabla
  h^{-1}_{\gamma}(x)\right)_{ii}\right|\right\}. 
\end{align*}
Note that, for every $x\in \real^m$ and every $i\in\{1,\ldots,m\}$, we have 
\begin{align*}
\left|1 - (\Hmin)_{ii}\left(\nabla h^{-1}_{\gamma}(x)\right)_{ii}\right|\le \left\|I_m - \Hmin \Hmax^{-1}\right\|_{\infty}.
\end{align*}
This implies that 
\begin{align*}
\sup_{x\in \real^m} \left\|D^{\frac{1}{2}}\left(I_m - \Hmin\nabla
  h^{-1}_{\gamma}(x)\right)D^{-\frac{1}{2}}\right\|_2 = \sup_{x\in\real^m}\left\|I_m - \Hmin\nabla
  h_{\gamma}^{-1}(x)\right\|_{\infty}\le \left\|I_m - \Hmin \Hmax^{-1}\right\|_{\infty}. 
\end{align*}
As a result, we get 
\begin{align}\label{eq:contraction}
\|T_{\mathbf{u}}(\eta_1) - T_{\mathbf{u}}(\eta_2)\|_{D} \le 
\left\|I_m - \Hmin \Hmax^{-1}\right\|_{\infty}\|\eta_1-\eta_2\|_{D}. 
\end{align} 
Thus, $T_{\mathbf{u}}:\Fsd\to \Fsd$ is a contraction mapping with respect to
the norm $\|.\|_{\mcA}$ on the vector space $\Fsd$. Therefore, by the
Banach Fixed-point Theorem, there exists a unique $f^*\in \Fsd$ 
such that
$f^* = T_{\mathbf{u}}(f^*)$. This completes the proof of 
part~\ref{p1:converge}. Regarding part~\ref{p2:speed}, using the contraction property of
$T_{\mathbf{u}}$ in equation~\eqref{eq:contraction}, for every $k\in \mathbb{N}$,
we get
\begin{align*}
\|T_{\mathbf{u}}^{(k+1)}(f^{(0)}) - T_{\mathbf{u}}^{(k)}(f^{(0)})\|_{D} \le 
\left\|I_m - \Hmin \Hmax^{-1}\right\|_{\infty}\|T_{\mathbf{u}}^{(k)}(f^{(0)}) - T_{\mathbf{u}}^{(k-1)}(f^{(0)})\|_{D}.
\end{align*}
This implies that, for every $k\in \mathbb{N}$, we have 
\begin{align*}
\|f^{(k+1)} - f^{(k)}\|_{D} \le 
\left\|I_m - \Hmin \Hmax^{-1}\right\|_{\infty}^{k}\|T_{\mathbf{u}}(f^{(0)}) - f^{(0)}\|_{D} .
\end{align*}
Now we show the equivalence of the statements~\ref{p1:check}
and~\ref{p2:exists_f}. 

\ref{p1:check} $\Longrightarrow$ \ref{p2:exists_f}: 
Since $f^*_{\mathbf{u}}$ is the limit point of the converging
sequence $\{T^k_{\mathbf{u}}\}_{k\in \mathbb{N}}$,  we have
$T_{\mathbf{u}}(f^*_{\mathbf{u}}) = f^*_{\mathbf{u}}$. Since $\Hmin\in
\real^{m\times m}$ is an invertible diagonal matrix, we get 
\begin{align*}
\prjcyc\Hmin \mcA (h_{\gamma}^{-1}(\mathcal{A}^{-1}f^*_{\mathbf{u}}) - 2\pi
  C^{\dagger}_{\Sigma}\mathbf{u}) = \vect{0}_m
\end{align*}
Adding the condition that, for every $e=(i,j)\in \mathcal{E}$, we have
{\color{black}$\left|(f^*_{\mathbf{u}})_e\right|\le a_{ij}|h_e(\gamma)|$}, we deduce that
$f^*_{\mathbf{u}}$ satisfies $\mathbf{u}$-winding balance
equation~\eqref{eq:winding-form}. Now we show that $f^*_{\mathbf{u}}$ is
the unique solution of the $\mathbf{u}$-winding balance
equation~\eqref{eq:winding-form}. Suppose that there exists $g\in \real^m$
such that $g$ satisfies the $\mathbf{u}$-winding balance
equation~\eqref{eq:winding-form}. In this case, by
equation~\eqref{eq:winding-balance}, we have that $g\in \subscr{F}{sd}$ and
\begin{align*}
  T_{\mathbf{u}}(g) = g - \prjcyc \Hmin \mcA (h^{-1}_{\gamma}(\mathcal{A}^{-1}g)-2\pi
  C_{\Sigma}^{\dagger}\mathbf{u}) = g,
\end{align*}
where the last equality holds because of~\eqref{eq:winding-equations}. As a result $g$ is
a fixed-point for the map $T_{\mathbf{u}}$. Thus, by part~\ref{p1:converge}, we
should have $g = \lim_{k\to\infty} T^k_{\mathbf{u}}(g) =f^*_{\mathbf{u}}$.

\ref{p2:exists_f} $\Longrightarrow$ \ref{p1:check}: Suppose that
$f^*_{\mathbf{u}}$ is a solution for the $\mathbf{u}$-winding
balance equation~\eqref{eq:winding-form}. It should satisfies the flow
constraint condition~\eqref{eq:winding-flow-constraint}. This implies that
{\color{black}$\left|(f^*_{\mathbf{u}})_e\right|\le a _{ij}|h_e(\gamma)|$}, for every
$e=(i,j)\in \mathcal{E}$. Regarding \ref{p2:exists_f} $\iff$ \ref{p3:exists_theta}: the
proof follows from Theorem~\ref{thm:equiv-kur-flow} Moreover, if
$(f^*_{\mathbf{u}},\theta^*)$ is the unique solution of the flow
network problem~\eqref{eq:flow_form} with $\theta^*\in
\Omega^G_{\mathbf{u}}$, then, by the proof of Theorem~\ref{thm:equiv-kur-flow},
there exists $\mathbf{x}\in \vect{1}_n^{\perp}$ such that
$B^{\top}\mathbf{x} + 2\pi C_{\Sigma}^{\dagger}\mathbf{u} =
h_{\gamma}^{-1}(\mathcal{A}f^*_{\mathbf{u}})$. This implies that 
\begin{align*}
  \mathbf{x} = L^{\dagger}B\mathcal{A}(h_{\gamma}^{-1}(\mathcal{A}^{-1}f^*_{\mathbf{u}})-2\pi
  C^{\dagger}_{\Sigma}\mathbf{u}).
\end{align*}
By Theorem~\ref{thm:linerization+correspondence}\ref{p1.5:connec-comp}, we can identify $\mathbf{x}$ and $\theta^*$
and thus $\theta^* = L^{\dagger}B\mathcal{A}(h_{\gamma}^{-1}(\mathcal{A}^{-1}f^*_{\mathbf{u}})-2\pi
  C^{\dagger}_{\Sigma}\mathbf{u})$. 
\end{proof}

\subsection{Proof of Theorem~\ref{thm:compcomp}}\label{app:flop}

As a first step in the proof of Theorem~\ref{thm:compcomp}, we prove the following useful lemma. 
\begin{lemma}\label{lem:graph-comp}
Consider an undirected weighted connected graph $G$ with a cycle basis $\Sigma$. Then, for every
  $\mathbf{u}\in \Img(\mathbf{w}_{\Sigma})$, the system of linear
  equations 
\begin{align}\label{eq:shift-equations}
C_{\Sigma}\mathbf{z} = \mathbf{u}, 
\end{align}
can be solved in $\mathcal{O}(nm)$. 
\end{lemma}
\begin{proof}
Let $\mathcal{T}$ be a spanning tree in $G$ and, without loss of
generality, assume that $\mathcal{E}-\mathcal{E}_{\mathcal{T}} =
\{e_1,\ldots, e_{m-n+1}\}$. Then, for every $i\in\{1,\ldots,m-n+1\}$,
we define the cycle $\sigma'_i$ by $\sigma'_i = (i_1,\ldots,i_k,i_1)$, 
where $e_i = (i_1,i_k)$ and $(i_1,\ldots,i_k)$ is the unique simple path
in the spanning tree $\mathcal{T}$ between nodes $i_1$ and $i_k$. Then
one can easily show that $\Sigma' =
\{\sigma'_1,\ldots,\sigma'_{m-n+1}\}$ is a cycle basis for $G$. Since
$\Sigma$ is also a cycle basis for $G$, there exists an invertible matrix
$R=\{r_{ij}\}$ such that 
\begin{align}\label{eq:basis-relations}
v_{\sigma'_i} = \sum_{j=1}^{m-n+1} r_{ji} v_{\sigma_{j}}
\end{align}
Moreover, each cycle in $G$ has at most $n$ edges, it is easy to check that, at
most $n$ entries in each rows of $R$ are non-zero. Using
equations~\eqref{eq:basis-relations}, one can deduce that $C_{\Sigma'}
= R C_{\Sigma}$. This implies that $\mathbf{z}$ is a solution for equations~\eqref{eq:shift-equations} if and only if it is a solution
for $C_{\Sigma'}\mathbf{z} = R \mathbf{u}$. Note that $\mathbf{z} =
[\vect{0}_{n-1}, R \mathbf{u}]^{\top}$  is a solution to
$C_{\Sigma'}\mathbf{z} = R \mathbf{u}$ and therefore a solution to equations~\eqref{eq:shift-equations}. Thus, by computing
$R\mathbf{u}$, one can find the solution to the linear systems of
equations~\eqref{eq:shift-equations}. Note $R\in
\real^{(m-n+1)\times (m-n+1)}$ and in each row of $R$ there is at most
$n$ non-zero element. Thus, the computational complexity of finding
$R\mathbf{u}$ and finding a solution for~\eqref{eq:shift-equations} is $\mathcal{O}(nm)$. 
\end{proof}
Now we go back to the proof of Theorem~\ref{thm:compcomp}. 
\begin{proof}
{\color{black}Regarding part~\ref{p1:numberiter}, since we have
  $|w_{\sigma_i}|\le \left\lfloor{\frac{\gamma
        n_{\sigma_i}}{2\pi}}\right\rfloor$, for every
  $i\in \{1,\ldots,m-n+1\}$. Thus, the total number of times that the
  for-loop at step~\algostep{2} is executed is
  $\prod_{i=1}^{m-n+1}\left\lfloor{\frac{\gamma
        n_{\sigma_i}}{2\pi}}\right\rfloor$.} Regarding part~\ref{p2:compeachiter}, the computational time of the
iterations~\eqref{eq:seq_converge} to check for the existence and/or compute the solution of the flow network problem~\eqref{eq:flow_form} is 
\begin{align}\label{eq:computational}
\left(\mbox{Run time for each iterations}\right) \times \left(\mbox{the
  number of iterations}\right). 
\end{align}
If $\rho\in \real_{>0}$ is the tolerance of the numerical method,
then, using Theorem~\ref{thm:convergence_ifexists}\ref{p2:speed}, the number of iterations is bounded above by 
\begin{align*}
\frac{\log(\rho^{-1})-\log\left(\norm{T_{\mathbf{u}}(f^{(0)})-f^{(0)}}{\Hmin\mathcal{A}}\right)}{\log(\|I_m-\Hmin\Hmax^{-1}\|_{\infty})},
\end{align*} 
which is independent of $n$ and $m$. Therefore, the number of
iterations is {\color{black}$\mathcal{O}(\rho^{-1})$}. Now we investigate the computational
time for each iteration. For the iteration step $k+1$, we need to
compute the terms $f^{(k)}$, $\Hmin\prjcyc
h^{-1}_{\gamma}(\mathcal{A}^{-1}f^{(k)})$, and $\Hmin\prjcyc
C^{\dagger}_{\Sigma}\mathbf{u}$ and then add them together. In what
follows, we analyze the computational complexity of each of these terms. 

\begin{enumerate}
\item  The term $f^{(k)}$ has already been computed from step $k$.

\item\label{eq:cc-0} Let $D\in \real^m$ be a diagonal matrix. then, for every $\eta\in \real^m$, the computational time for
    $D \eta$ is $\mathcal{O}(m)$. 

\item\label{eq:cc-1} Note that the functions $h_{\gamma}$ and $h_{\gamma}^{-1}$ can be computed offline
  and therefore computing $h_{\gamma}^{-1}(\mathcal{A}^{-1}f^{(k)})$ can be done in $\mathcal{O}(m)$.

\item\label{eq:cc-2} For every
$\eta \in \real^m$, we have
\begin{align*}
\prjcyc\eta = \eta -
\Hmin \mathcal{A} B^{\top}(B\Hmin\mcA B^{\top})^{\dagger}B\eta
\end{align*}
In order to compute $\Hmin \mcA B^{\top}(B\Hmin \mcA B^{\top})^{\dagger}B\eta$, we consider the following
equality
\begin{align*}
B^{\top}(B\Hmin\mcA B^{\top})^{\dagger}B\eta = B^{\top}\mathbf{y},
\end{align*}
where $\mathbf{y}$ is the solution to the linear equations
$B\Hmin\mcA B^{\top}\mathbf{y} = B\eta$. Thus computing
$B^{\top}(B\Hmin\mcA B^{\top})^{\dagger}B\eta$ is equivalent to computing $\mathbf{y}$ and then multiplying it by $B^{\top}$. Using the $LU$ decomposition method, one can compute $\mathbf{y}$ in
$\mathcal{O}(n^3)$ time~\cite[\S 3.2]{GHG-CFvL:89}. Moreover, each row in matrix $B^{\top}$ has
exactly $2$ nonzero element. Thus, given $\mathbf{y}$, the term $B^{\top}\mathbf{y}$ can be computed
in $\mathcal{O}(n)$ in time. Therefore, the computational time for
computing the term $\prjcyc\eta$ is in $\mathcal{O}(n^3)$ in time.

\item By parts~\ref{eq:cc-0},~\ref{eq:cc-1}, and~\ref{eq:cc-2}, the computational time for the
  term $\prjcyc \Hmin \mcA h^{-1}_{\gamma}(\mathcal{A}^{-1}f^{(k)})$ is $\mathcal{O}(n^3)$. 
  
\item Since $\Ker(C_{\Sigma})=\Img(B^{\top}) = \Ker(\prjcyc \Hmin \mcA)$,
  then we have
  \begin{align*}
    2\pi \prjcyc\Hmin\mcA C^{\dagger}_{\Sigma}\mathbf{u} = 2\pi
    \prjcyc\Hmin\mcA\mathbf{z},
  \end{align*}
  where $\mathbf{z}$ is the solution of the
  linear system of equations $C_{\Sigma}\mathbf{z} = \mathbf{u}$. By
  Lemma~\ref{lem:graph-comp} $,\mathbf{z}$ can be computed in $\mathcal{O}(mn)$. By parts~\ref{eq:cc-0},~\ref{eq:cc-1},
  and~\ref{eq:cc-2}, $2\pi \prjcyc\Hmin\mcA\mathbf{z}$ can be computed in $\mathcal{O}(n^3)$. 
\end{enumerate}
Therefore, each iteration of the projection iteration can be computed in
at most $\mathcal{O}(n^3)$ and, by equation~\eqref{eq:computational},
the projection iteration~\eqref{eq:seq_converge} can check the
existence/compute the solutions of the flow network problem~\eqref{eq:flow_form} in
$\mathcal{O}(n^3)\times \mathcal{O}(\log(\rho^{-1})) =
\mathcal{O}(\log(\rho^{-1}) n^3)$ time. {\color{black}Regarding part~\ref{p3:compiter}, note that the run
  time of the for-loop in steps~\algostep{2}-\algostep{11} is given by
\begin{align}\label{eq:computational2}
\left(\mbox{Run time for one execution of for-loop}\right) \times \left(\mbox{the
  number of for-loops}\right). 
\end{align}
By part~\ref{p1:numberiter}, the run time for one execution of the
for-loop is $\mathcal{O}(\log(\rho^{-1}) n^3)$. Also, by
part~\ref{p2:compeachiter}, the number of time the for-loop is invoked
is $\prod_{i=1}^{m-n+1}\left\lfloor{\frac{\gamma n_{\sigma_i}}{2\pi}}\right\rfloor$. This means that the run
time for in steps~\algostep{2}-\algostep{11} is
\begin{align*}
\mathcal{O}(\log(\rho^{-1}) n^3) \mathcal{O}\left(\prod_{i=1}^{m-n+1}\left\lfloor{\frac{\gamma
        n_{\sigma_i}}{2\pi}}\right\rfloor\right) =
  \mathcal{O}\left(\log(\rho^{-1}) n^3\left(\tfrac{\gamma}{2\pi}\right)^{m-n+1}
  n_{\sigma_1}\ldots n_{\sigma_{m-n+1}}\right),
\end{align*}
where the last equality holds because $\frac{\gamma
        n_{\sigma_i}}{2\pi} - 1\le \left\lfloor{\frac{\gamma
        n_{\sigma_i}}{2\pi}}\right\rfloor \le \frac{\gamma
        n_{\sigma_i}}{2\pi}$, for every $i\in
  \{1,\ldots,m-n+1\}$. Regarding part~\ref{p4:comptotal}, suppose that
  $\Sigma = (\sigma_1,\ldots,\sigma_{m-n+1})$ is the cycle basis obtained from modified Horton
  algorithm in~\cite{KM-DM:09}. Then the length of $\Sigma$ is bounded above by
  $3(n-1)(n-2)/2$~\cite[Theorem 6]{JDH:87}. Therefore, we have 
\begin{align*}
  n_{\sigma_1}\ldots n_{\sigma_{m-n+1}}\le
  \left(\frac{n_{\sigma_1}+\ldots+n_{\sigma_{m-n+1}}}{m-n+1}\right)^{m-n+1} \le \left(\frac{3(n-1)(n-2)}{2(m-n+1)}\right)^{m-n+1}.
\end{align*}
As a result, by part~\ref{p3:compiter}, the run time of the for-loop
in steps~\algostep{2}-\algostep{11} is $\mathcal{O}\left(\log(\rho^{-1}) n^3\left(\frac{3(n-1)(n-2)}{m-n+1}\right)^{m-n+1}\right)$. Moreover, the run
time of the modified Horton algorithm in~\cite{KM-DM:09} is
$\mathcal{O}(m^2n/\log(n) + n^2m)$ (see
Remark~\ref{rem:cyclebasis}). Thus the result easily follows. }
\end{proof}


\myclearpage

\end{document}